\documentclass[10pt,a4paper]{article}
\usepackage{amsmath}
\usepackage{amsthm}
\usepackage{amsfonts}
\usepackage{multirow}
\usepackage[table,xcdraw]{xcolor}
\usepackage{subcaption}
\usepackage{mathtools}

\DeclarePairedDelimiter\floor{\lfloor}{\rfloor}
\usepackage{float}
\usepackage{algorithm}
\usepackage{algpseudocode}
\usepackage{hyperref}
\usepackage{tabularx}
\usepackage{amssymb}
\usepackage{graphicx}
\usepackage{tcolorbox}
\usepackage{booktabs}
\usepackage{multirow}
\usepackage{rotating,tabularx}

\usepackage[left=2cm,right=2cm,top=2cm,bottom=2cm]{geometry}
\newcolumntype{L}[1]{>{\raggedright\let\newline\\\arraybackslash\hspace{0pt}}m{#1}}
\newcolumntype{C}[1]{>{\centering\let\newline\\\arraybackslash\hspace{0pt}}m{#1}}
\newcolumntype{R}[1]{>{\raggedleft\let\newline\\\arraybackslash\hspace{0pt}}m{#1}}
\newtheorem{Assumption}{Assumption}
\newtheorem{Theorem}{Theorem}[section]
\newtheorem{Proposition}[Theorem]{Proposition}
\newtheorem{Remark}[Theorem]{Remark}
\newtheorem{Lemma}[Theorem]{Lemma}
\newtheorem{Corollary}[Theorem]{Corollary}
\newtheorem{Definition}[Theorem]{Definition}
\newtheorem{Example}[Theorem]{Example}

\usepackage{hyperref}
\hypersetup{colorlinks=true,urlcolor=blue}
\expandafter\let\expandafter\oldproof\csname\string\proof\endcsname
\let\oldendproof\endproof
\renewenvironment{proof}[1][\proofname]{
\oldproof[\ttfamily\scshape \bf #1.]
}{\oldendproof}

\def\tilde{\widetilde}
\def\emp{\emptyset}

\def\dom{{\rm dom}\,}

\def\epi{{\rm epi\,}}
\def\rge{{\rm rge\,}}
\def\min{\mbox{\rm minimize}}

\def\B{\mathbb B}

\def\ox{\overline{x}}
\def\oy{\overline{y}}
\def\oz{\overline{z}}

\def\disp{\displaystyle}
\def\tto{\rightrightarrows}
\def\Hat{\widehat}
\def\Tilde{\widetilde}
\def\Bar{\overline}
\def\ra{\rangle}
\def\la{\langle}

\def\epsilon{\varepsilon}
\def\ox{\bar{x}}
\def\oy{\bar{y}}
\def\oz{\bar{z}}

\def\ov{\bar{v}}

\def\gph{\mbox{\rm gph}\,}
\def\epi{\mbox{\rm epi}\,}

\def\dom{\mbox{\rm dom}\,}
\def\ker{\mbox{\rm ker}\,}

\def\dn{\downarrow}
\def\O{\Omega}
\def\ph{\varphi}
\def\emp{\emptyset}
\def\st{\stackrel}
\def\oR{\Bar{\R}}
\def\lm{\lambda}

\def\al{\alpha}

\def \N{{\rm I\!N}}
\def \R{{\rm I\!R}}

\def\Limsup{\mathop{{\rm Lim}\,{\rm sup}}}

\def\Limsup{\mathop{{\rm Lim}\,{\rm sup}}}

\numberwithin{equation}{section}

\title{\bf Coderivative-Based Newton Methods in Structured Nonconvex and Nonsmooth Optimization}
\author{Pham Duy Khanh\footnote{Group of Analysis and Applied Mathematics, Department of Mathematics, Ho Chi Minh City University of Education, Ho Chi Minh City, Vietnam. E-mails: pdkhanh182@gmail.com, khanhpd@hcmue.edu.vn.}\quad Boris S. Mordukhovich\footnote{Department of Mathematics, Wayne State University, Detroit, Michigan, USA. E-mail: aa1086@wayne.edu. }\quad Vo Thanh Phat\footnote{Department of Mathematics and Statistics,  University of North Dakota, Grand Forks, North Dakota, USA. E-mail: thanh.vo.1@und.edu.}}
\begin{document}
\maketitle
\vspace*{-0.2in}
\noindent
{\small{\bf Abstract}. This paper proposes and {develops} new Newton-type methods to solve structured nonconvex and nonsmooth optimization problems with justifying their fast  local and global convergence by means of advanced tools of variational analysis and generalized differentiation. The objective functions belong to a broad class of prox-regular functions with specification to constrained optimization of nonconvex structured sums. We also develop a novel line search method, which is an extension of the proximal gradient algorithm while allowing us to globalize the proposed coderivative-based Newton methods by incorporating the machinery of forward-backward envelopes.  {Applications and numerical experiments, which are provided for nonconvex least squares regression models, Student’s t-regression with an $\ell^0$-penalty, and  image restoration problems, demonstrate the efficiency of the proposed algorithms.}\\[0.5ex]
}\vspace*{-0.15in}

\section{Introduction}\label{intro}\vspace*{-0.05in}

For  a function $\varphi: \R^n \rightarrow \R$, which is twice continuously differentiable (${\cal C}^2$-smooth) around a certain point $\ox$, the {\em classical Newton method} to solve the unconstrained optimization problem
\begin{equation}\label{unconpb}
\min \quad \varphi(x) \quad \text{subject to }\; x\in \R^n
\end{equation}
starts with a given initial point $x^0\in\R^n$ and constructs the iterative procedure
\begin{equation}\label{clas-newton} 
x^{k+1}:=x^k+d^k\;\mbox{ for all }\;k\in\N:=\big\{1,2,\ldots\big\}, 
\end{equation} 
where $d^k$ is a solution to the linear system 
\begin{equation}\label{newton-iter}	
-\nabla\varphi(x^k)=\nabla^2\varphi(x^k)d^k,\quad k\in \N,
\end{equation} 
expressed via the Hessian matrix $\nabla^2\varphi(x^k)$ of $\varphi$ at $x^k$. In classical optimization, Newton's algorithm in \eqref{clas-newton} and \eqref{newton-iter} is well-defined, meaning that the equations in \eqref{newton-iter} are solvable for $d^k$ and the sequence of its iterations ${x^k}$ quadratically converges to the solution $\ox$ of \eqref{unconpb} if $x^0$ is chosen sufficiently close to $\ox$ and if the Hessian matrix $\nabla^2\varphi(\ox)$ is {\em positive-definite}. However, this method necessitates information on Hessian matrices, which are unavailable in many practical optimization problems where the cost function may not be $\mathcal{C}^2$-smooth. Therefore, the development of generalized Newton methods is required with replacing the Hessian matrix by a certain generalized second-order derivative.

Over the years, various generalized Newton methods have been developed in nonsmooth optimization; we refer the reader to the books \cite{JPang,Solo14,kummer} and the bibliographies therein for more details. 

Our approach in this paper is different from the majority  of developments on generalized Newton methods in nonsmooth optimization and related areas. It is based on the machinery of variational analysis and generalized differentiation that is revolved around the {\em second-order subdifferentials/generalized Hessians} introduced by Mordukhovich \cite{m92} in the dual ``derivative-of-derivative" scheme by taking a coderivative of a subgradient mapping; see Section~\ref{sec:pre}. The importance and benefits of such constructions have been well recognized in variational analysis, optimization, and numerous applications (see, e.g., the books \cite{Mordukhovich06,Mor24} with the vast bibliographies therein), while their usage in the design and justification of numerical algorithms, as well as in practical applications to machine learning and other disciplines, has been revealed quite recently in 
\cite{amp23,Helmut,BorisKhanhPhat,kmptjogo,kmptmp,Mor24,BorisEbrahim}.\vspace*{0.01in} 

The {\em main goals} of this paper include developing novel coderivative-based generalized Newton methods to solve broad classes of optimization-related problems that are different from the known achievements in this direction by the methodology, design, imposed assumptions, obtained results, and applications. Let us discuss some major issues and their placements in the paper .\vspace*{0.01in}

$\bullet$ In Section~\ref{sec:general} (after the basic definitions and preliminary results given in Section~\ref{sec:pre}), we address solving {\em subgradient inclusions} of the type $0\in\partial\ph(x)$,  where $\partial\ph$ stands for the {\em limiting subdifferential} of an extended-real-valued, lower semicontinuous (l.s.c.), and prox-bounded function $\ph\colon\R^n\to\oR:=(-\infty,\infty]$. Using the aforementioned second-order subdifferential $\partial^2\ph$, we design a generalized Newton algorithm to solve such problems that doesn't require, in contrast to \cite{BorisKhanhPhat}, the nonemptiness of the limiting subdifferential $\partial\ph(x^k)$ along the iterative sequence. The well-posedness and {\em local superlinear convergence} of the new algorithm to find stationary points are justified under natural assumptions.\vspace{0.01in}

$\bullet$ Section~\ref{sec:localNMcompo} develops a novel Newton-type algorithm to solve {\em structured optimization problems} 
\begin{equation}\label{ncvoptintro}
\min \quad \varphi(x):=f(x)+g(x) \quad \text{subject to }\;x \in \R^n,
\end{equation}
where $f:\R^n \to \R$ is a $\mathcal{C}^2$-smooth function while the {\em regularizer} $g:\R^n\to\overline{\R}$ is l.s.c.\ and prox-bounded. Note that both $f$ and $g$ are generally {\em nonconvex}. This makes \eqref{ncvoptintro} appropriate for applications to signal and image processing, machine learning, statistics, control, system identification, etc. The newly proposed algorithm to solve \eqref{ncvoptintro} uses, in contrast to the previous developments in \cite{kmptjogo,kmptmp}, the proximal mapping only for the regularizer $g$, while not for the cost function $\ph$. This is often easy to compute and enables us to address broader classes in constrained optimization and applications.  {To proceed efficiently, we use the notion of {\em normal maps} for \eqref{ncvoptintro} and establish the local superlinear convergence of iterates in the proposed algorithm.}\vspace*{0.01in}

$\bullet$ Sections~\ref{sec:newlinesearch} and \ref{sec:globalNMcompo} are devoted to the {\em globalization} of developed algorithms to solve nonconvex structured optimization problems \eqref{ncvoptintro}. In Section~\ref{sec:newlinesearch}, we propose and justify a novel {\em line search method}, which can be viewed as a {\em generalized proximal gradient} line search. To proceed, we exploit the machinery of {\em forward-backward envelopes} (FBE) introduced in \cite{pb}. The newly proposed line search algorithm {doesn't involve  any} generalized second-order constructions and ensures the {\em global linear convergence} of iterates. Involving the {\em second-order information} in Section~\ref{sec:globalNMcompo} allows us to design the {\em globalized coderivative-based Newton method} (GCNM), which exhibits the {\em   global convergence with local superlinear rate}  under the assumptions that are less restrictive compered to the previous developments in \cite{kmptjogo,kmptmp}. In particular, we don't impose the positive-definiteness of the generalized Hessian on the entire space $\R^n$ (while only at the solution point in some settings) and don't restrict ourselves to quadratic convex functions $f$ in \eqref{ncvoptintro}. This  {dramatically extends} the scope of applications.\vspace*{0.01in}

$\bullet$  {From an application perspective, Section~\ref{sec:appl0l2} demonstrates the use of GCNM in solving nonconvex problems such as the {\em $\ell_0$-$\ell_2$ regularized least squares regression} and Student’s $t$-regression with an $\ell_0$-penalty, both of which frequently arise in machine learning and statistics. For these types of nonsmooth and nonconvex problems, all algorithmic components and parameters are computed via the problem data. Numerical experiments are conducted on randomly generated instances, and the results are compared with other second-order approaches, including the SCD semismooth$^*$ Newton method~\cite{G25} and the ZeroFPR method with BFGS update~\cite{stp2}. GCNM is applied to nonconvex image restoration problems to further demonstrate its efficiency.}\vspace*{-0.15in}

\section{Preliminaries from Variational Analysis}\label{sec:pre}\vspace*{-0.05in}

This section overviews some notions and results of first-order and second-order variational analysis and generalized differentiation broadly used in the paper. We refer the reader to the monographs \cite{Mordukhovich06,Mor24,Rockafellar98} and the more recent papers mentioned below. 

Given a set-valued mapping/multifunction $F\colon\R^n\tto\R^m$ between finite-dimensional spaces, its (Painlev\'e-Kuratowski) {\em outer limit} at $\ox$ is defined by
\begin{equation*}
\Limsup_{x\to\ox} F(x):
=\big\{y\in\R^n\;\big|\;\exists\,\mbox{ sequences }\;x_k\to\bar x,\;y_k\rightarrow y\;\mbox{ with }\;y_k\in F(x_k),\;\;k\in\N:=\{1,2,\ldots\}\big\}.
\end{equation*}
Using the notation $z\st{\O}{\to}\oz$ meaning that $z\to\oz$ with $z\in\O$ for a given nonempty set $\O\subset\R^s$, the (Fr\'echet) {\em regular normal cone} to $\Omega$ at $\bar{z}\in\Omega$ is
\begin{equation}\label{regnm}
\widehat{N}_\Omega(\bar{z}):=\Big\{v\in\R^s\;\Big|\;\limsup_{z\overset{\Omega}{\rightarrow}\bar{z}}\frac{\langle v,z-\bar{z}\rangle}{\|z-\bar{z}\|}\le 0\Big\},
\end{equation}
while the (Mordukhovich) {\em limiting normal cone} to $\Omega$ at $\bar{z}\in\Omega$ is defined by
\begin{equation}\label{lnc}
N_\Omega(\bar{z}):=  \underset{z\overset{\Omega}{\to} \bar{z}} {\Limsup}\;\widehat{N}_\Omega (z) =\big\{v\in\R^s\;\big|\;\exists\,z_k\st{\O}{\to}\bar{z},\;v_k\to v\;\text{ as }\;k\to\infty\;\text{ with }\;v_k\in\widehat{N}_\Omega(z_k)\big\}.
\end{equation} 
The	(Bouligand-Severi) {\em tangent/contingent cone} to $\Omega$ at $\oz$ is	
\begin{equation}\label{tan}
T_\Omega(\oz):=\big\{w\in\R^s\;\big|\;\exists\,t_k\dn 0,\;w_k\to
w\;\mbox{ as }\;k\to\infty\;\mbox{ with }\;\oz+t_k w_k\in\Omega\big\}.
\end{equation} 
Note the duality relation between the regular normal and tangent cones
\begin{equation}\label{dua}
\Hat
N_\Omega(\oz)=\left[T_\Omega(\oz)\right]^*=\big\{v\in\R^s\;\big|\;\la v,w\ra\le
0\;\mbox{ for all }\;w\in T_\Omega(\oz)\big\} 
\end{equation} 
and observe that the limiting normal cone \eqref{lnc} cannot be dual,  {in view of its} {\em intrinsic nonconvexity}, to any tangential approximation of the set in question.

For a set-valued mapping $F\colon\R^n\tto\R^m$, we define its {\em graph}, \textit{domain, kernel}, and \textit{range} by 
\begin{equation*}
\gph F:=\big\{(x,y)\in\R^n\times\R^m\;\big|\;y\in F(x)\big\},\quad\dom F:=\big\{x\in \R^n\;\big|\;F(x) \ne \emp\big\},
\end{equation*}
\begin{equation*}
\ker F:=\big\{x\in \R^n\;\big|\; 0 \in F(x)\big\},\;\mbox{ and }\;
\rge F:=\big\{y \in \R^m\;\big|\;\exists\ x \in
\R^n\quad \text{with }\; y\in F(x)\big\}, 
\end{equation*}
respectively. The {\em inverse mapping} of $F$ is the set-valued mapping $F^{-1}:\R^m \rightrightarrows \R^n$ given by
$$
F^{-1}(y):=\big\{x \in \R^n\;\big|\; y \in F(x)\big\}, \quad y \in \R^m. 
$$

Based on the normal cones \eqref{regnm} and \eqref{lnc} to the graphs, we define the {\em regular coderivative} and the {\em limiting
coderivative} of $F$ at $(\ox,\oy)\in\gph F$ by, respectively,
\begin{equation}\label{reg-cod} 
\Hat D^*F(\ox,\oy)(v):=\big\{u\in\R^n\;\big|\;(u,-v)\in\Hat N_{{\rm
gph}\,F}(\ox,\oy)\big\},\quad v\in\R^m, 
\end{equation}
\begin{equation}\label{lim-cod}
D^*F(\ox,\oy)(v):=\big\{u\in\R^n\;\big|\;(u,-v)\in N_{{\rm
gph}\,F}(\ox,\oy)\big\},\quad v\in\R^m. 
\end{equation} 
The {\em graphical derivative} of $F$ at
$(\ox,\oy)\in\gph F$ is defined via the tangent cone \eqref{tan} by
\begin{equation}\label{gra-der}
DF(\ox,\oy)(u):=\big\{v\in\R^m\;\big|\;(u,v)\in T_{{\rm
gph}\,F}(\ox,\oy)\big\},\quad u\in\R^n. 
\end{equation}
Due to the regular normal-tangent cone duality \eqref{dua}, we have the duality correspondence between \eqref{reg-cod} and \eqref{gra-der}, while the nonconvex-valued limiting coderivative \eqref{lim-cod} is not tangentially generated. In the case
where $F(\bar{x})$ is the singleton $\{\bar{y}\}$, we omit $\oy$ in the notation of \eqref{reg-cod}--\eqref{gra-der}. Note that if
$F\colon\R^n\to\R^m$ is ${\cal C}^1$-smooth around $\ox$ (in fact, just strictly differentiable at this point), then
\begin{equation*} 
DF(\bar{x})=\nabla F(\bar{x})\;\mbox{ and
}\;\widehat{D}^*F(\bar{x})=D^*F(\bar{x})=\nabla F(\bar{x})^*,
\end{equation*} 
where $\nabla F(\bar{x})^*$ is the adjoint/transpose
matrix of the Jacobian $\nabla F(\bar{x})$.\vspace*{0.03in}

To proceed further, recall that the {\em distance function} associated with
a set $\Omega\subset\R^s$ is \begin{equation*} {\rm
dist}(x;\Omega):=\inf\big\{\|w-x\|\;\big|\;w\in\Omega\big\},\quad
x\in\R^s,
\end{equation*} 
and define the {\em Euclidean projector} of $x \in \R^s$ to $\Omega$ by
$$
\Pi_\Omega(x):=\big\{w \in \Omega\;\big|\; \|x-w\|={\rm
dist}(x;\Omega)\big\}.
$$
A mapping $\widehat{F}\colon\R^n\tto\R^m$ is a {\em localization} of
$F\colon\R^n\tto\R^m$ at $\bar{x}$ for $\bar{y}\in
F(\bar{x})$ if there exist neighborhoods $U$ of $\bar{x}$ and $V$ of
$\bar{y}$ such that we have $\gph\widehat{F}=\gph F\cap(U\times V)$. Here are the important well-posedness properties of set-valued mappings used in what follows.\vspace*{-0.1in}
	
\begin{Definition}[\bf metric regularity and subregularity of mappings]\label{met-reg} \rm  Let $F\colon\R^n\tto\R^m$ be a set-valued mapping, and let $(\bar{x},\bar{y})\in\gph F$. We say that:  

{\bf(i)} $F$ is {\em metrically regular} around $(\bar{x},\bar{y})$	with modulus $\mu>0$ if there exist neighborhoods $U$ of $\bar{x}$	and $V$ of $\bar{y}$ providing the estimate 
\begin{equation*} {\rm dist}\big(x;F^{-1}(y)\big)\le\mu\,{\rm dist}\big(y;F(x)\big)\;\text{ for all }\;(x,y)\in U\times V.
\end{equation*}
If in addition $F^{-1}$ has a single-valued localization around $(\bar{y},\bar{x})$, then $F$ is {\sc  strongly metrically regular} around $(\bar{x},\bar{y})$ with modulus $\mu>0$.

{\bf(ii)} $F$ is {\em metrically subregular} at $(\bar{x},\bar{y})$ with modulus $\mu>0$ if there is a neighborhood $U$ of $\bar{x}$ such that
\begin{equation*} {\rm dist}\big(x;F^{-1}(\bar{y})\big)\le\mu\,{\rm	dist}\big(\bar{y};F(x)\big)\;\text{ for all }\;x\in U.
\end{equation*} 
If $F^{-1}(\bar{y})\cap U=\{\bar{x}\}$, then $F$ is {\sc strongly metrically subregular at} $(\bar{x},\bar{y})$
with modulus $\mu>0$. We say that $F$ is  {\sc strongly metrically subregular around}  $(\bar{x},\bar{y})$ if there exist neighborhoods $U$ of $\bar{x}$
and $V$ of $\bar{y}$ such that $F$ is  strongly metrically subregular at every point $(x,y) \in \gph F \cap (U\times V)$. 
\end{Definition} \vspace*{-0.03in}
\noindent
Observe that the metric regularity properties in Definition~\ref{met-reg}{(i) are preserved in a neighborhood of the point $(\ox,\oy)$, i.e., they are \textit{robust} with respect to small perturbations of the reference point, which is not always the case for metric subregularity from Definition~\ref{met-reg}{(ii). Recall that $F\colon\R^n\tto\R^m$, which graph is closed around $(\ox,\oy)\in\gph F$, is {\em metrically regular} around this point {\em if and only if} we have the coderivative condition
\begin{equation}\label{cod-cr} 
\ker D^*F(\ox,\oy):=\big\{v\in\R^m\;\big|\;0\in
D^*F(\ox,\oy)(v)\big\}=\{0\} \end{equation} 
established by Mordukhovich \cite[Theorem~3.6]{Mordu93} via his limiting	coderivative \eqref{lim-cod} and then labeled as the {\em Mordukhovich criterion} in Rockafellar and Wets \cite[Theorems~7.40 and 7.43]{Rockafellar98}. Broad applications of this result are based on the robustness and {\em full calculus} available for \eqref{lim-cod}; see \cite{Mordukhovich06,Mor18,Rockafellar98} for more details and references. A parallel characterization of the (nonrobust) {\em strong metric subregularity} property of a closed-graph mapping $F\colon\R^n\tto\R^m$ at $(\ox,\oy)\in\gph F$ is given by $\ker DF(\ox,\oy)=\{0\}$ via the (nonrobust) graphical derivative \eqref{gra-der} of $F$ at $(\ox,\oy)$ and is known as the {\em Levy-Rockafellar criterion}; see \cite[Theorem~4E.1]{Donchev09} for more discussions.\vspace*{0.03in} 
 
Next we consider an extended-real-valued function $\ph\colon\R^n\to\oR$ with the domain and epigraph
\begin{equation*}
\dom\ph:=\big\{x\in\R^n\;\big|\;\ph(x)<\infty\big\},\quad\epi\ph:=\big\{(x,\al)\in\R^{n+1}\;\big|\;\al\ge\ph(x)\big\}.
\end{equation*}
The {\em regular subdifferential} and {\em limiting subdifferential} of $\ph$ at $\ox\in\dom\ph$ are defined geometrically via the corresponding normal cones to the epigraph by, respectively,
$$ 
\widehat{\partial}\varphi(\ox):=\big\{v\in\R^n\;\big|\;(v,-1)\in \widehat{N}_{{\rm\text{\rm epi}}\,\varphi}\big(\bar{x},\varphi(\bar{x})\big)\big\},
$$ 
\begin{equation}\label{lim-sub}
\partial\varphi(\ox):=\big\{v\in\R^n\;\big|\;(v,-1)\in N_{{\rm \text{\rm epi}}\,\varphi}\big(\bar{x},\varphi(\bar{x})\big)\big\}
\end{equation}
Note that the limiting subdifferential \eqref{lim-sub} admits various analytic representations and satisfies comprehensive calculus rules that can be found in \cite{Mordukhovich06,Mor18,Rockafellar98}. Observe the useful {\em scalarization formula}
\begin{equation}\label{scal}
D^*F(\bar{x})(v)=\partial\langle v,F\rangle(\bar{x})\;\mbox{ for all }\;v\in\R^m 
\end{equation}
connecting the coderivative \eqref{lim-cod} of a locally Lipschitzian mapping $F\colon\R^n\to\R^m$ and the subdifferential \eqref{lim-sub} of the scalarized function $x\mapsto\langle v,F\rangle(x)$ whenever $v\in\R^m$.

Following \cite{m92}, we define the {\em second-order subdifferential}, or {\em generalized Hessian}, $\partial^2\ph(\ox,\ov)\colon\R^n\tto\R^n$ of $\ph\colon\R^n\to\oR$ at $\ox\in\dom\ph$ for $\ov\in\partial\ph(\ox)$ as the coderivative of the subgradient mapping
\begin{equation}\label{2nd}
\partial^2\ph(\ox,\ov)(u):=\big(D^*\partial\ph\big)(\ox,\oy)(u)\;\mbox{ for all }\;u\in\R^n.
\end{equation}
If $\ph$ is ${\cal C}^2$-smooth around $\ox$, then we have (no Hessian transposition is needed) that
$$ 
\partial^2\ph(\ox)(u)=\big\{\nabla^2\ph(\ox)u\big\}\;\mbox{ for all }\;u\in\R^n.
$$
If $\ph$ is of {\em class $\mathcal{C}^{1,1}$} around $\ox$ (i.e., it is ${\cal C}^1$-smooth with a locally Lipschitzian gradient, then
$$ 
\partial^2\ph(\ox)(u)=\partial\big\la u,\nabla\ph\big\ra(\ox)\;\mbox{ for all }\;u\in\R^n
$$ 
by \eqref{scal}. In general, \eqref{2nd} is a positively homogeneous multifunction, which possesses well-developed {\em calculus rules} and variety of applications while being {\em explicitly computed} in terms of the initial data for large classes of extended-real-valued functions that  {frequently appear} in variational analysis, optimization, and numerous applications. We refer the reader to the recent book \cite{Mor24} for a rather comprehensive account on \eqref{2nd}; see also Section~\ref{sec:appl0l2} for new applications to models of machine learning and statistics.\vspace*{0.03in}

Let us now review some classes of extended-real-valued functions $\ph\colon\R^n\to\oR$ broadly used in the paper.  The notion of  {\em variational $s$-convexity}  for extended-real-valued functions is a major one we intend to investigate in this paper. For the general case of $s\in\R$, it was recently introduced in \cite{varprox}, while its earlier versions for $s\ge 0$ appeared in \cite{r19} and were later studied and applied in finite-dimensional spaces in \cite{KKMP23-3,kmp22convex,roc}, and in infinite-dimensional settings in \cite{kkmpbook,KKMP23-2,Mor24}. The case of $s\leq 0$ is related to \textit{prox-regularity}, a concept originally introduced in \cite{Poliquin} and subsequently developed and widely applied in variational analysis and optimization; see \cite{Rockafellar98}. Before introducing the notion of variational $s$-convexity, we recall that a function $\ph:\R^n \to \overline{\R}$ is {\em $s$-convex} for $s \in \R$ if the quadratically $s$-shifted function $\ph - \frac{s}{2}\|\cdot\|^2$ is convex. In particular, when $s > 0$, the function $\ph$ is referred to as \textit{$s$-strongly convex}, while for $s < 0$, it is labeled as \textit{$(-s)$-weakly convex}.\vspace*{-0.05in}

\begin{Definition}[\bf variationally convex functions]\label{vr} \rm Given $s\in\R$, a proper l.s.c.\ function $\varphi\colon\R^n \to \overline{\R}$ is {\em variationally $s$-convex} at $\bar{x}\in\dom\ph$ for $\bar{v} \in \partial \ph(\bar{x})$ if there exist an $s$-convex function $\widehat{\ph}$ as well as convex neighborhoods $U$ of $\bar{x}$ and  $V$ of $\bar{v}$ together with $\varepsilon > 0$ such that $\widehat{\ph} \le \ph$ on $U$ and 
$$
(U \times V) \cap \gph \partial \widehat{\ph} = (U_\varepsilon \times V) \cap \gph \partial \ph, \quad\widehat{\ph}(x) = \ph(x)\;\text{ at common elements }\;(x,v), 
$$
where $U_\varepsilon := \{x \in U \mid \ph(x) < \ph(\bar{x}) + \varepsilon\}$. {When $s \le 0$, the property is called {\em $(-s)$-level prox-regularity} of $\varphi$ at $\bar{x}$ for $\bar{v}$; we simply say that $\varphi$ is {\em prox-regular} at $\bar{x}$ for $\bar{v}$ if this holds for some $s \le 0$. 
The function $\varphi$ is {\em variationally convex} at $\bar{x}$ for $\bar{v}$ when it is $0$-level prox-regular. 
When $s>0$, $\varphi$ is {\em $s$-strongly variationally convex} at $\bar{x}$ for $\bar{v}$, and is said to be {\em strongly variationally convex} if this holds for some $s>0$.}
\end{Definition}\vspace*{-0.07in}

The next proposition taken from \cite[Theorem 1]{varprox} lists the characterizations of variationally convex functions via some properties of their limiting subdifferentials.\vspace*{-0.07in}

\begin{Proposition}[\bf subgradient characterizations of variational convexity]\label{subgradientvar} Let $\varphi:\R^n\to\overline{\R}$  be l.s.c.\ with $\bar{x}\in\dom\varphi$ and $\bar{v}\in \widehat{\partial}\varphi(\bar{x})$. Then for any $s \in \R$, the following are equivalent:

{\bf(i)} $\varphi$ is variationally $s$-convex at $\ox$ for $\ov$.

{\bf(ii)} There are neighborhoods $U$ of $\bar{x}$, $V$ of $\bar{v}$, and a number $\epsilon>0$ such that  
$$ 
\langle v_1 - v_2, u_1 - u_2\rangle \geq s\|u_1-u_2\|^2 \quad \text{for all }\; (u_1,v_1),(u_2,v_2) \in \gph\partial\varphi\cap (U_\epsilon \times V),
$$
where $U_\epsilon:=\{x\in U\;|\;\varphi(x)<\varphi(\ox)+\epsilon\}$. 

{\bf(iii)} There are neighborhoods $U$ of $\bar{x}$, $V$ of $\bar{v}$, and a number $\epsilon>0$ such that 
$$ 
\varphi(x)\ge\varphi(u)+\langle v, x-u\rangle +\frac{s}{2}\|x-u\|^2\quad\text{for all }\;x\in U,\;(u,v)\in\gph\partial\varphi\cap (U_\epsilon\times V). 
$$  
\end{Proposition}\vspace*{-0.07in}
 
We say that $\varphi$ is {\em subdifferentially continuous} at $\bar{x}$ for $\ov\in \partial\varphi(\bar{x})$ if for any $\epsilon>0$, there is $\delta>0$ with $|\varphi(x)-\varphi(\bar{x})|<\epsilon$ whenever $(x,v)\in\gph\partial\varphi\cap (\mathbb{B}_\delta(\ox)\times\mathbb{B}_\delta(\ov))$. {Functions that are both prox-regular and subdifferentially continuous are called {\em continuously prox-regular}.  When the level of prox-regularity is specified, we refer to it as {\em continuous $r$-level prox-regularity} for some $r \ge 0$.}
{Note that the continuous prox-regularity is a robust property, i.e., if $\varphi$ is continuously prox-regular at $\ox$ for $\ov$, then it is also continuously prox-regular at any $(x,v)$ sufficiently close to $(\ox,\ov)$; see \cite[Proposition 3.6.]{hs24}}.  This is a major class of extended-real-valued functions in second-order variational analysis, being a common roof for particular collections of functions important for applications as, e.g., convex and $\mathcal{C}^{1,1}$ functions, amenable functions, etc.; see \cite[Chapter~13]{Rockafellar98}.
  
Given a function $\varphi:\R^n\to\overline{\R}$ and a parameter value $\lambda>0$, we always assume that $\ph$ is {\em proper} (i.e., $\dom\ph\ne\emp$) and define the {\em Moreau envelope} $e_\lambda\varphi$ and {\em proximal mapping} $\textit{\rm Prox}_{\lambda\varphi}$ by, respectively,
\begin{equation}\label{Moreau}
e_\lambda\varphi(x):=\inf\Big\{\varphi(y)+\frac{1}{2\lambda}\|y-x\|^2\;\Big|\;y\in \R^n\Big\},\quad x\in\R^n,
\end{equation}	
\begin{equation}\label{ProxMapping} 
\operatorname{Prox}_{\lambda\varphi}(x):= \operatorname{argmin}\Big\{\varphi(y)+ \frac{1}{2\lambda}\|y-x\|^2\;\Big|\;y\in\R^n\Big\},\quad x\in\R^n.
\end{equation}	
A function $\varphi$ is said to be {\em prox-bounded} if there exists $\lambda>0$ such that $e_\lambda\varphi(x)>-\infty$ for some $x\in\R^n$. The supremum over all such $\lambda$ is the \textit{threshold} $\lambda_\varphi$ of prox-boundedness for $\varphi$. Prox-boundedness admits various equivalent descriptions and ensures a favorable behavior of Moreau envelopes and proximal mappings. We recall the following results from \cite[Theorem~1.25 and Example~10.32]{Rockafellar98}.\vspace*{-0.05in}

\begin{Proposition}[\bf Moreau envelopes and proximal mappings for prox-bounded functions]\label{pbimplycontinuity} Let $\varphi:\R^n\to\overline{\R}$ be l.s.c.\ and prox-bounded with threshold $\lambda_\varphi>0$. Then for all $\lambda\in(0,\lambda_\varphi)$, we have:

{\bf(i)} $e_\lambda\varphi$ is locally Lipschitzian on $\R^n$.

{\bf(ii)} $\operatorname{Prox}_{\lambda\varphi}(x)\ne\emp$ for any $x\in \R^n$.

{\bf(iii)} If $w^k\in\operatorname{Prox}_{\lambda\varphi}(x^k)$ for $k\in \N$ such that $x^k \rightarrow \bar{x}$ as $k\to\infty$, then the sequence $\left\{w^k\right\}_{k \in \N}$ is bounded and all its accumulation points lie in $\operatorname{Prox}_{\lambda\varphi}(\ox)$.
\end{Proposition}\vspace*{-0.05in}

{The next proposition presents some properties of Moreau envelopes and proximal mappings that provides, in particular, the constructive choice of the parameter $\lm$ of the Moreau envelope \eqref{Moreau} in comparison with \cite[Theorem~4.4]{Poliquin}, which is important for the design and justification of algorithms.}}\vspace*{-0.05in} 
		
\begin{Proposition}[\bf Moreau envelopes and proximal mappings for prox-regular functions]\label{C11}  {Let $r\geq 0$ and $\varphi: \R^n \to \overline{\R}$ be  prox-bounded,  $r$-level prox-regular at $\bar{x}$ for $\bar{v} \in \partial \ph(\bar{x})$. Then we have
\begin{equation}\label{thresholdlm}
P:= \left\{\lambda \in \left(0,\frac{1}{r}\right)  \Big|\; \varphi(x) \geq  \varphi(\ox) + \langle \ov, x-\ox\rangle - \frac{1}{2\lambda}\|x-\ox\|^2 \quad \text{for all }\; x \in \R^n \right\} \ne \emptyset
\end{equation} 
with the convention that $1/0: = \infty.$ 
Let $\alpha \in P$ and $\bar{\lambda} \in (0,\alpha)$. Moreover, for all  $\lambda \in \left(0, \bar{\lambda} \right)$, there exists a neighborhood $U_\lambda$ of $\bar{x} + \lambda\bar{v}$ on which the following hold:}  

{\bf(i)} The Moreau envelope $e_\lambda\varphi$ from \eqref{Moreau} is of class ${\cal C}^{1,1}$ on the set $U_\lambda $.

{\bf(ii)} The proximal mapping $\text{\rm Prox}_{\lm\ph}$ from \eqref{ProxMapping}  is single-valued  and Lipschitz continuous on $U_\lm$ while satisfying the condition ${\rm Prox}_{\lm\ph}(\ox+\lm\ov)=\ox$.

{\bf(iii)} The gradient of $e_\lm\ph$ is calculated by:
\begin{equation}\label{GradEnvelope} 
\nabla e_\lambda\varphi(x)=\frac{1}{\lambda}\Big(x-\text{\rm Prox}_{\lambda\varphi}(x)\Big)\;\mbox{ for all }\;x\in U_\lambda.
\end{equation}
If in addition $\varphi$ is subdifferentially continuous at $\ox$ for $\ov$, then we also have
\begin{equation}\label{gradenvsub}
\nabla e_\lambda\varphi(x) = \left(\lambda I + (\partial \varphi)^{-1}\right)^{-1}(x) \quad \text{for all} \quad x \in U_\lambda.
\end{equation}
\end{Proposition}
\begin{proof} First we show that $P \ne \emptyset$ for $P$ from \eqref{thresholdlm}. Indeed, 
the $r$-level prox-regularity of $\varphi$ at $\ox$ for $\ov$ yields 
$$
\varphi(x) \geq \varphi(\ox) + \langle \ov, x -\ox\rangle - \frac{r}{2}\|x -\ox\|^2 \quad \text{for all }\; x \; \text{sufficiently near }\; \ox.
$$
Combining the latter with the prox-boundedness of $\varphi$, we deduce from the proof of \cite[Proposition~8.46(f)]{Rockafellar98} that there exists $\tilde{\lambda} \in (0,1/r)$ such that 
$$ 
\varphi(x) \geq \varphi(\ox) + \langle \ov, x -\ox\rangle - \frac{1}{2\tilde{\lambda}}\|x -\ox\|^2 \quad \text{for all }\; x \in \R^n,
$$
which ensures that $P \ne \emptyset$ and thus justifies \eqref{thresholdlm}. Next we show that assertions (i)--(iii) hold for all $\lambda \in \left(0, \bar{\lambda} \right)$. To this ens, fix such $\lambda$ and consider the function $\phi:\R^n \to \oR$ defined by
\begin{equation}\label{fshift}
\phi(x) := \varphi(x+\ox) -\varphi(\ox) - \langle \ov, x\rangle \quad \text{for all }\; x \in \R^n.
\end{equation}
It is clear that $\phi(0) =0$,  $0 \in \partial \phi(0)$, and
\begin{equation}\label{proxg}
\text{\rm Prox}_{\lambda \phi} (x) = \text{\rm Prox}_{\lambda \varphi}(x+\ox+\lambda \ov) -\ox  
\quad \text{and }\; 
e_\lambda \phi(x)= e_\lambda \varphi(x+\ox+\lambda \ov)  - \langle \ov,x\rangle - \varphi(\ox) - \frac{\lambda}{2}\|\ov\|^2,
\end{equation}
for all $x\in \R^n.$
Since $\alpha \in P$, and $\bar{\lambda} <\alpha$, we deduce from  \eqref{fshift} that  
\begin{equation}\label{proxbm}
\phi(x) > -\frac{1}{2\bar{\lambda}}\|x\|^2  \quad \text{for all }\;x \ne 0.
\end{equation}
Moreover, the $r$-level prox-regularity of $\varphi$ at $\ox$ for $\ov$ ensures the $r$-level prox-regularity of $\phi$ at $0$ for $0$, which also tells us that $\phi$ is $1/{\bar{\lambda}}$-level prox-regular at $0$ for $0$. Combining this with \eqref{proxbm} and applying \cite[Theorem~4.4]{Poliquin} to the function $\phi$, we can find a neighborhood $X_\lambda$ of $0$ such that $\text{\rm Prox}_{\lambda \phi}$ is single-valued and Lipschitz continuous on $X_\lambda$ with $\text{\rm Prox}_{\lambda \phi}(0) =0$, while the Moreau envelope $e_\lambda \phi$ is of class $\mathcal{C}^{1,1}$ on $X_\lambda$ with 
\begin{equation}\label{gradofenvg}
\nabla e_\lambda \phi(x)= \frac{1}{\lambda}\left(x - \text{\rm Prox}_{\lambda \phi} (x) \right)  \quad \text{for all }\; x \in X_\lambda.
\end{equation}
Moreover, the following equation holds if $\phi$ is subdifferentially continuous at $0$ for $0$: 
\begin{equation}\label{gradofenvg2}
\nabla e_\lambda \phi(x)= (\lambda I+(\partial \phi)^{-1})^{-1}(x) \quad \text{for all }\; x \in X_\lambda.
\end{equation}
Defining $L_\lambda(x):= x+\ox +\lambda \ov$ for all $x \in \R^n$, we deduce that $U_\lambda:=L_\lambda(X_\lambda)$ is a neighborhood of $\ox + \lambda \ov$. Fix an arbitrary $x \in U_\lambda$ and find $x^\prime \in X_\lambda$ such that $x = x^\prime + \ox +\lambda \ov$. By \eqref{proxg}, it follows that 
$$
\text{\rm Prox}_{\lambda \varphi}(x)= \text{\rm Prox}_{\lambda \phi}(x^\prime) + \ox \quad \text{and }\;  \nabla e_\lambda \varphi(x) = \nabla e_\lambda \phi(x^\prime) + \ov.
$$
Combining the latter with \eqref{gradofenvg}, we have
\begin{equation}\label{proof(ii)}
\nabla e_\lambda \varphi(x) = \nabla e_\lambda \phi(x')+\ov =  \frac{1}{\lambda}\big(x'- \text{\rm Prox}_{\lambda \phi}(x')\big)  + \ov  = \frac{1}{\lambda}\big(x- \text{\rm Prox}_{\lambda \varphi}(x)\big),
\end{equation}
which verifies \eqref{GradEnvelope}. Supposing further that $\varphi$ is subdifferentially continuous at $\ox$ for $\ov$, we show that \eqref{gradenvsub} holds. Indeed, due to \eqref{fshift}, $\phi$ is also subdifferentially continuous at $0$ for $0$, and thus we deduce from \eqref{gradofenvg2} that $\nabla e_\lambda \phi(x')= (\lambda I+ (\partial \phi)^{-1})^{-1}(x') $. This implies that
$$
\nabla e_\lambda \varphi(x) = \nabla e_\lambda \phi(x') +\ov \in   \partial \phi (x' - \lambda \nabla e_\lambda \phi(x')) +\ov = \partial \varphi(\ox+x'-\lambda \nabla e_\lambda \phi(x')) = \partial\varphi(x-\lambda \nabla e_\lambda \varphi(x)),
$$
which tells us in turn that
\begin{equation}\label{gradften}
\nabla e_\lambda \varphi(x) \in (\lambda I + (\partial\varphi)^{-1})^{-1}(x) \ne \emptyset.
\end{equation}
Let us next verifies that $(\lambda I + (\partial\varphi)^{-1})^{-1}(x)$ is singleton. Suppose that $v_1, v_2 \in (\lambda I + (\partial\varphi)^{-1})^{-1}(x)$ and get  that $v_1 \in \partial\varphi(x- \lambda v_1)$ and $v_2 \in \partial\varphi(x-\lambda v_2)$. The continuous $r$-level prox-regularity of $\varphi$ at $\ox$ for $\ov$ yields  
$$
\langle x - \lambda v_1 - x + \lambda v_2 , v_1 -v_2 \rangle \geq -r\| x -\lambda v_1 - x+\lambda v_2\|^2,
$$
or $(1- \lambda r) \|v_1-v_2\|^2 \leq 0$, which tells us that $v_1 = v_2$ since $\lambda \in (0,1/r)$. Therefore, $ (\lambda I + (\partial\varphi)^{-1})^{-1}(x)$ is singleton. Combining the latter with \eqref{gradften}, we obtain \eqref{gradenvsub} and  thus complete the proof.
\end{proof}

\begin{Remark}[\bf on the choice of  the  Moreau envelope parameter] \label{choiceMr}\rm As in Proposition~\ref{C11}, the parameter $\lambda$ ensuring the fulfillment of the assertions in this result depends on the value of $\bar{\lambda}$. The determination of $\lambda$ can be simplified in the following situations:

\medskip
{\bf(i)} When $\ox \in \text{\rm argmin}\,\varphi$, we have $\varphi(x) \geq \varphi(\ox)$ for all $x \in \R^n$. In this case, it follows that $P= (0,1/r)$ and thus $\lambda$ can be chosen arbitrarily in $(0,1/r)$, where $r$ is the level of prox-regularity of $\varphi$ at $\ox$ for $\ov$.

\medskip
{\bf(ii)} When $\varphi$ is $r$-weakly convex, we always have
$$
\varphi(y) \geq \varphi(x) + \langle v, y - x\rangle - \frac{r}{2}\|y - x\|^2
\quad \text{for all } x, y \in \R^n,
$$
and hence $\lambda$ can be chosen in $(0,1/r)$. In the special case where $\varphi$ is convex, we have $r=0$, which implies that $\lambda$ can be chosen arbitrarily in $(0,\infty)$.

\medskip 
{\bf(iii)} When $\ox \in \text{\rm Prox}_{\alpha \varphi} (\ox+\alpha\ov)$ for some $\alpha >0$, we have $\alpha \in P$. Indeed, it follows from \eqref{ProxMapping} that
$$
\varphi(x) + \frac{1}{2\alpha}\|x- \ox -\alpha\ov\|^2 \geq \varphi(\ox) + \frac{1}{2\alpha}\|\ox-\ox-\alpha \ov \|^2 \quad \text{for all }\; x \in \R^n,
$$
which is equivalent to
$$
\varphi(x) \geq \varphi(\ox) + \langle \ov, x-\ox\rangle - \frac{1}{2\alpha}\|x-\ox\|^2 \quad x\in\R^n,
$$
and thus $\alpha \in P$. In this case, we can choose $\lambda \in (0,\alpha)$ ensuring the fulfillment of Proposition~\ref{C11}.
\end{Remark}

The fundamental notion of {\em tilt stability} for local minimizers \cite{Poli} is formulated as follows.\vspace*{-0.05in}

\begin{Definition}[\bf tilt-stable local minimizers]\label{def:tilt} \rm Given $\varphi\colon\R^n\to\oR$, a point $\ox\in\dom\varphi$ is a {\em tilt-stable local minimizer} of $\varphi$ if there exists a number $\gamma>0$ such that the mapping
\begin{equation*}
M_\gamma\colon v\mapsto{\rm argmin}\big\{\varphi(x)-\langle v,x\rangle\;\big|\;x\in\B_\gamma(\ox)\big\}
\end{equation*}
is single-valued and Lipschitz continuous on some neighborhood of $0\in\R^n$ with $M_\gamma(0)=\{\ox\}$. By a {\sc modulus} of tilt stability of $\ph$ at $\ox$, we understand a Lipschitz constant of $M_\gamma$ around the origin.
\end{Definition}\vspace*{-0.05in}

The next proposition, taken from \cite[Proposition~2.9]{kmp22convex} and \cite[Proposition~3.1]{DL}, provides a precise relationship between strong variational convexity, tilt stability, and strong metric regularity of subgradient mappings associated with continuously prox-regular functions.\vspace*{-0.06in}

\begin{Proposition}[\bf strong variational convexity and tilt stability]\label{equitiltstr} Let $\sigma>0$ and  $\varphi:\R^n\to\overline{\R}$ be continuously prox-regular at $\bar{x}\in\dom\varphi$ for $0\in\partial\varphi(\bar{x})$. The following assertions are equivalent:
		
{\bf(i)} $\varphi$ is $\sigma$-strongly variationally convex at $\ox$ for $\ov=0$. 
		
{\bf(ii)} $\ox$ is a tilt-stable local minimizer of $\varphi$ with modulus $\sigma^{-1}$.\\[0.05ex]
In this case, $\partial \varphi$ is strongly metrically regular around $(\ox,0).$ 
\end{Proposition}\vspace*{-0.06in}
 
Now we recall some other notions of second-order generalized differentiation in variational analysis. Given a function $\ph\colon\R^n\to\oR$ with $\ox\in\dom\ph$, consider the family of second-order finite differences
\begin{equation*}
\Delta^2_\tau\varphi(\bar{x},v)(u):=\frac{\varphi(\bar{x}+\tau u)-\varphi(\bar{x})-\tau\langle v,u\rangle}{\frac{1}{2}\tau^2}
\end{equation*}
and define the {\em second subderivative} of $\varphi$ at $\ox$ for $v\in\R^n$ and $w\in\R^n$ by
\begin{equation*}
d^2\varphi(\ox,v)(w):=\liminf_{\genfrac{}{}{0pt}{}{\tau\downarrow 0}{u\to w}}\Delta^2_\tau\varphi(\ox,v)(u).
\end{equation*}
{Then $\ph$ is {\em twice epi-differentiable} at $\ox$ for $v$ if for every $w\in\R^n$ and $\tau_k\downarrow 0$ there is a sequence $w^k\to w$ with $
\Delta^2_{\tau_k}\varphi(\bar{x},v)(w^k) \to d^2\varphi(\ox,v)(w)\;\mbox{ as }\;k\to\infty.$}
Twice epi-differentiability has been recognized as an important concept of second-order variational analysis with numerous applications to optimization; see, e.g., \cite{Rockafellar98}. To formulate the related notion of \textit{generalized twice differentiability} of functions, {we proceed following \cite{roc}}. Given a quadratic matrix $A \in \R^{n \times n}$, denote the {\em quadratic form} associated with the matrix $A$ by
$$
q_A(w) := \la w, Aw \ra\;\mbox{ for all }\; w \in \R^n.
$$ 
A function $q: \R^n \to \overline{\R}$ is called a \textit{generalized quadratic form} if it is expressible as $q = q_A + \delta_L$, where $L$ is a linear subspace of $\R^n$, and where $A \in \R^{n \times n}$ is a symmetric matrix. It is said that $\varphi: \R^n \to \overline{\R}$ is \textit{generalized twice differentiable} at $\bar{x}$ for $\bar{v} \in \partial \varphi(\bar{x})$ if it is twice epi-differentiable at $\bar{x}$ for $\bar{v}$ and if the second-order subderivative $d^2 \varphi(\bar{x},\bar{v})$ is a generalized quadratic form. The next notion was first proposed in \cite{roc} for convex functions and then extended in \cite{gtdquad,quadcharvar,roc24} for a general class of extended-real-valued functions $\ph\colon\R^n\to\oR$. The \textit{quadratic bundle} $\mathrm{quad}\,\varphi (\bar{x}|\bar{v})$ of $\varphi$ at $\bar{x} \in \dom \varphi$ for $\bar{v} \in \partial \varphi(\bar{x})$ is defined as the collection of generalized quadratic forms $q$ for which there exists $(x_k,v_k) \xrightarrow{\Omega_\varphi} (\bar{x},\bar{v})$ such that $\varphi(x_k) \to \varphi(\bar{x})$ and the sequence of generalized quadratic forms $q_k = \frac{1}{2} d^2 \varphi(x_k,v_k)$ converges epigraphically to $q$ as $k\to\infty$, where
$$
\Omega_\varphi :=\big\{(x,v) \in \gph \partial \varphi\;\big|\;\varphi\;\text{ is generalized twice differentiable at }x \;\text{for } v \big\}.
$$
For any $q \in \mathrm{quad}\,\varphi (\bar{x}|\bar{v})$, by combining \cite[Lemma~3.7 and Proposition~3.33]{GfOu22}, we obtain the following relationship between the second-order subdifferential in \eqref{2nd} and the quadratic bundle as follows.\vspace*{-0.07in}

\begin{Proposition}[\bf second-order subdifferentials and quadratic bundles]\label{2ndandquad} Let $\varphi:\R^n\to\oR$ be continuously prox-regular at $\ox \in \dom \varphi$ for $\ov \in \partial\varphi(\ox)$. Then we have the inclusion
 $$
\gph\partial q \subset \gph\partial^2 \varphi(\ox,\ov) \quad \text{for any }\; q \in \mathrm{quad}\,\varphi (\bar{x}|\bar{v}).
$$
\end{Proposition}
\vspace*{-0.25in}

\section{Coderivative-Based Newton Method for Subgradient Systems}\label{sec:general}\vspace*{-0.05in}

In this section, we design and justify a novel generalized Newton method to solve {\em subgradient inclusions}
\begin{equation}\label{sub-incl}
0 \in \partial\varphi(x),\quad x\in\R^n,
\end{equation}
where $\ph\colon\R^n\to\overline{\R}$ is an l.s.c.\ and \textit{prox-bounded} function, and where $\partial$ stands for the limiting subdifferential \eqref{lim-sub}. Recall that in the case of a $\mathcal{C}^2$-smooth function $\varphi$, the classical Newton method constructs the iterative procedure given in \eqref{clas-newton} and \eqref{newton-iter}. A natural idea, which we explored in our previous publications, is to replace $\nabla\varphi(x^k)$ and $\nabla^2\varphi(x^k)$ with a vector $v^k \in \partial\varphi(x^k)$ and $\partial^2\varphi(x^k,v^k)$, respectively. In this way, the iterative sequence $\{x^k\} \subset \R^n$ is constructed by
\begin{equation}\label{developnm}
{x^{k+1}=x^k+d^k} \quad \text{with } {- v^k \in \partial^2\varphi(x^k,v^k)(d^k), \;  v^k \in \partial\varphi(x^k),\quad k \in \N}.
\end{equation}
However, to proceed as in \eqref{developnm} is not possible if the subgradient sets $\partial\varphi(x^k)$ are empty for some $k\in\N$, which is not generally excluded. We now propose to replace $x^k$ with another \textit{approximate vector} $\Hat{x}^k$ that ensures that the limiting subdifferential $\partial\varphi(\Hat{x}^k)$ is nonempty and that $\Hat{x}^k$ is not far from $x^k$. This leads us to the iterative sequence $\{x^k\} \subset \R^n$ constructed by
\begin{equation}\label{newton-iternonsmooth}
{x^{k+1}=\Hat{x}^k+d^k} \quad \text{with } {- \Hat{v}^k \in \partial^2\varphi(\Hat{x}^k,\Hat{v}^k)(d^k), \quad k \in\N},
\end{equation}
where $(\Hat{x}^k,\Hat{v}^k)\in \gph\partial\varphi$. This general framework plays a crucial role in designing and justifying our new coderivative-based Newton-type algorithms for solving nonsmooth and nonconvex optimization problems in Sections~\ref{sec:localNMcompo} and  \ref{sec:globalNMcompo}. 
 
Based on \eqref{newton-iternonsmooth}, a key aspect in the development and validation of our Newtonian algorithms is to establish the existence of and explicitly determine a direction $d \in \R^n$ that satisfies the \textit{generalized Newton system}  (or the \textit{second-order subdifferential inclusion}) 
\begin{equation}\label{newton-inc}  
-v \in \partial^2\varphi(x,v)(d)
\end{equation}  
for a given $(x,v) \in \gph \partial\varphi$.  Efficient conditions for {\em solvability} of \eqref{newton-inc} presented in \cite[Theorems~3.1 and 3.2]{BorisKhanhPhat} ensure that the proposed algorithm is {\em well-defined}. Observe that the local closedness assumption for the limiting subdifferential graph imposed below holds in very general settings, especially when $\ph\colon\R^n\to\oR$ is subdifferentially continuous; see Section~\ref{sec:pre}. The following results are {taken from} \cite[Theorems~3.1 and 3.2]{BorisKhanhPhat}.\vspace*{-0.05in}

\begin{Proposition}[\bf solvability of generalized Newton systems]\label{solvabilityMOR}  
Let $\ph\colon\R^n\to\oR$ be such that $\gph \partial \varphi$ is locally closed around $(\ox,\ov) \in \gph \partial\varphi$. The following assertions hold:

{\bf(i)} If the subgradient mapping $\partial\ph\colon\R^n\tto\R^n$ is strongly metrically subregular at $(\ox,\ov)$, then there exists $d\in\R^n$ satisfying the generalized Newton system \eqref{newton-inc} for $(x,v):=(\ox,\ov)$.

{\bf(ii)} If the subgradient mapping $\partial\ph$ is strongly metrically regular around $(\ox,\ov)$, then there exists a neighborhood $U\times V$ of $(\ox,\ov)$ such that for each $(x,v)\in\gph\partial\ph\cap(U\times V)$, we find a direction $d\in\R^n$ satisfying the generalized Newton system \eqref{newton-inc}.  
\end{Proposition} \vspace*{-0.05in} 

Proposition~\ref{solvabilityMOR} only provides a condition that guarantees the existence of a direction $d \in \R^n$ satisfying \eqref{newton-inc} without offering explicit methods to find this direction. In practice, an explicit calculation of such a direction is crucial for the implementation of our algorithms. Let us now obtain new sufficient conditions that not only guarantee the existence of the required direction but also furnish a constructive procedure to calculate it for some important classes of functions. Before establishing the theorem, we present a preliminary result from \cite[Proposition~4.2]{GfOu22} needed in the proof.\vspace*{-0.05in}

\begin{Lemma}[\bf subspace property]\label{linearsubpacelemma} Let $L$ be a subspace of $\R^{2n}$ satisfying the condition 
$(\R^n \times \{0\}) \cap L = \{(0,0)\}$. Then there exists a unique $n \times n$ matrix $C_L$ such that $(C_Le_i,e_i)\in L$ for $i = 1,\ldots, n$, where $e_i$ stands for the $i$th unit vector of $\R^n$. 
\end{Lemma}\vspace*{-0.15in}

\begin{Theorem}[\bf generalized Newton directions via solving linear systems]\label{newsolNM} If $\varphi:\R^n~\to~\oR$ is continuously prox-regular at $\ox$ for $\ov \in \partial\varphi(\ox)$, then the quadratic bundle  $\mathrm{quad}\,\varphi (\bar{x}|\bar{v})$ is nonempty. Assuming that the mapping $\partial\ph\colon\R^n\tto\R^n$ is  metrically  regular around $(\ox,\ov)$, we have for each $q \in \mathrm{quad}\,\varphi (\bar{x}|\bar{v})$ that $\partial q(w) \ne \emptyset$ whenever $w \in \R^n$. Taking further $u_i \in \partial q(e_i)$ for $i = 1,\ldots, n$, define the $n\times n$ matrix $A(\ox,\ov)$ such that its $i$th column is $u_i$. Then we have the assertions:

{\bf(i)} The matrix $A(\ox,\ov)$ is nonsingular. 

{\bf(ii)} If $d \in \R^n$ solves the linear system $A(\ox,\ov) d = -\ov$, 
then the direction $d$ satisfies the generalized Newton system \eqref{newton-inc} for $(x,v):=(\ox,\ov)$. 
\end{Theorem}\vspace*{-0.15in}
\begin{proof} The nonemptiness of the quadratic bundle $\mathrm{quad}\,\varphi (\bar{x}|\bar{v})$ is proved in \cite[Theorem 6.11]{gtdquad}, and thus we can choose an element $q \in \mathrm{quad}\,\varphi (\bar{x}|\bar{v})$. It follows from Proposition~\ref{2ndandquad} that
\begin{equation}\label{relaquad}
\gph\partial q \subset \gph\partial^2 \varphi(\ox,\ov) = \gph D^* \partial\varphi (\ox,\ov).
\end{equation}
Utilizing the metric regularity of $\partial\varphi$ around $(\ox,\ov)$ and applying the Mordukhovich criterion \eqref{cod-cr} together with  \eqref{relaquad} bring us to the relationships 
$$
(\R^n \times \{0\}) \cap \gph\partial q \subset (\R^n \times \{0\}) \cap \gph D^*\partial\varphi(\ox,\ov)=  \{(0,0)\}. 
$$
Note that $\gph\partial q$ is a linear subspace of $\R^{2n}$ since $q$ is a generalized quadratic form.  By Lemma~\ref{linearsubpacelemma}, we find a unique $n \times n$ matrix $C$ such that $C  e_i \in \partial q(e_i)$ for all $i=1,\ldots,n$. 
This implies that $\partial q(w) \ne \emptyset$ whenever $w \in \R^n$ since $\gph\partial q$ is a linear subspace of $\R^{2n}$.
Pick further any $u_i \in \partial q(e_i)$ for $i = 1,\ldots, n$ and define the $n\times n$ matrix $A(\ox,\ov)$ such that its $i$th column is $u_i$. To verify (i), observe that $w \in \R^n$ satisfies the system of linear equations $A(\ox,\ov)w =0$ if and only if $\sum_{i=1}^n w_i u_i  =0$. It follows from \eqref{relaquad} that 
$$
0 = \sum_{i=1}^n w_i u_i  \in \partial q\left(\sum_{i=1}^n w_ie_i \right) = \partial q (w) \subset \partial^2\varphi(\ox,\ov)(w),
$$
which implies by using the Mordukhovich criterion \eqref{cod-cr} that $w=0$ and thus ensures the nonsingularity of $A(\ox,\ov)$. Taking a unique solution $d$ to $A(\ox,\ov) d = -\ov$ tells us by Proposition \ref{2ndandquad} that
$$
A(\ox,\ov)d \in \partial q(d) \subset \partial^2\varphi(\ox,\ov)(d),
$$
which verifies (ii) and completes the proof of the theorem. 
\end{proof}\vspace*{-0.05in}

Theorem~\ref{newsolNM} provides a sufficient condition for the existence of a direction in the generalized Newton system \eqref{newton-inc}. Unlike our previous results in \cite{BorisKhanhPhat}, the new one is constructive and offers an explicit method to determine generalized Newton directions. Note that the new result applies to a broader class of functions than those considered in \cite{BorisKhanhPhat}. In particular, it drops the twice epi-differentiability assumption in \cite[Theorem~4.2]{BorisKhanhPhat}.\vspace*{0.03in}

Now we present an alternative approach to find a direction $d \in \R^n$ that satisfies the generalized Newton system \eqref{newton-inc} by solving an {\em optimization subproblem} involving second subderivatives. \vspace*{-0.05in}

\begin{Theorem}[\bf generalized Newton directions via solving optimization subproblems] \label{newsolvability2}  Let $\varphi:\R^n\to\overline{\R}$ be strongly variationally convex, subdifferentially continuous, and twice epi-differentiable at $\bar{x}\in\dom\varphi$ for $\bar{v}\in\partial\varphi(\bar{x})$. Then the unconstrained optimization problem 
\begin{equation}\label{subproblem2nd}
\min \quad \frac{1}{2}d^2 \varphi(\ox,\ov)(w) + \langle \ov, w\rangle\;\text{ subject to }\; w \in \R^n
\end{equation}    
admits a unique optimal solution $d \in \R^n$, which satisfies the generalized Newton system \eqref{newton-inc} for $(x,v):=(\ox,\ov)$.   
\end{Theorem}\vspace*{-0.15in}
\begin{proof} {It follows from Proposition \ref{subgradientvar} that $\partial \varphi$ is locally strongly monotone around $(\ox,\ov)$ in the sense that there are $s>0$ and a neighborhood $W$ of $(\ox,\ov)$ such that
$$
\langle v_1 -v_2, u_1 -u_2\rangle \geq s\|u_1-u_2\|^2 \quad \text{for all }\; (u_1,v_1),(u_2,v_2) \in \gph \partial\varphi \cap W. 
$$}
Using \cite[Corollary~6.3]{Poliquin} and \cite[Proposition~13.5]{Rockafellar98}, we deduce that $d^2\varphi(\ox,\ov)$ is an l.s.c.\ strongly convex function, and thus \eqref{subproblem2nd} admits a unique optimal solution $d \in \R^n$. By 
\cite[Proposition~1.30]{Mor18}, we have the inclusion
\begin{equation}\label{fermatrule2ndsub} 
0 \in \partial \left(\frac{1}{2}d^2 \varphi(\ox,\ov) + \langle \ov, \cdot\rangle \right) (d) = \partial\left(\frac{1}{2}d^2 \varphi(\ox,\ov)   \right) (d) + \ov. 
\end{equation}
Employing the fundamental result of \cite[Theorem~13.40]{Rockafellar98} gives us the relationship between second subderivatives and graphical derivatives as follows:  
\begin{equation}\label{2ndsubandgraphical}
 \partial\left(\frac{1}{2}d^2 \varphi(\ox,\ov)   \right) (d) = (D\partial\varphi)(\ox,\ov)(d). 
\end{equation}
Moreover, it follows from \cite[Theorem~13.57]{Rockafellar98}  that 
\begin{equation}\label{relagrapcoderivative}
(D\partial\varphi)(\ox,\ov)(d) \subset  (D^*\partial\varphi)(\ox,\ov)(d) =  \partial^2\varphi(\ox,\ov) (d)
\end{equation}
Combining \eqref{fermatrule2ndsub}, \eqref{2ndsubandgraphical}, and \eqref{relagrapcoderivative}, we arrive at the claimed result. 
\end{proof}\vspace*{-0.12in}

\begin{Remark}[\bf solving the generalized Newton system]\label{findNMremark} \rm  In general, solving the second-order subdifferential inclusion \eqref{newton-inc} may be a challenging issue. Nevertheless, our approach  offers computational advantages in some rather {broad   scenarios}.

{\bf(i)} Theorem~\ref{newsolNM} provides a method to find a solution to \eqref{newton-inc} by solving {\em linear systems generated by quadratic bundles}. Quadratic bundles are defined for nonsmooth and nonconvex functions and have been extensively studied regarding optimality conditions and variational sufficiency in \cite{dsw25,quadcharvar,roc,roc24}, with additional calculus rules developed in \cite{dsw25,gtdquad}.

{\bf(ii)} Theorem~\ref{newsolvability2} enables us to obtain a solution to \eqref{newton-inc} by solving a {\em strongly convex optimization subproblem} derived from second subderivatives. Calculus rules for second subderivatives have been thoroughly explored in \cite{mms,Rockafellar98} for practical and {broad classes of} functions overwhelmingly arising  in optimization and variational analysis. Moreover, when the cost function is generalized twice differentiable, the subproblem  in \eqref{subproblem2nd} reduces to a {\em quadratic program over a subspace}, for which many efficient numerical algorithms are widely available.

{\bf(iii)} When the cost function is the {\em sum of two nonconvex functions}, which is the primary focus of this paper, Theorem~\ref{newsolvability} offers an alternative approach to finding the Newton direction by computing the Bouligand Jacobian of the proximal mapping $\partial_B \text{\rm Prox}_{\lambda g}$. Explicit formulas for these computations are provided in, e.g., \cite[Section~5]{pb14} covering a wide range of important optimization problems.

\end{Remark}\vspace*{-0.05in}
Here is our novel coderivative-based Newton method for solving the subgradient system \eqref{sub-incl}. \vspace*{-0.1in}

\begin{algorithm}[H]
\caption{\bf(general framework for coderivative-based Newton-type methods)}\label{NMcod}
\begin{algorithmic}[1]
\Require {$x^0\in \R^n$, $\eta >0$}
\For{\texttt{$k=0,1,\ldots$}}
\State\text{If $0\in \partial\varphi(x^k)$, stop; otherwise  go to the next step}
\State  {\bf Approximate step:} \text{Find  $(\Hat{x}^k,\Hat{v}^k)\in\gph\partial\varphi$  satisfying} 
\begin{equation}\label{stopcri}
\|(\Hat{x}^k, \Hat{v}^k)- (\ox,0)\| \leq \eta \|x^k- \ox\|
\end{equation}
\State {\bf Newton step:} Choose $d^k\in\R^n$ satisfying
$$ 
-\Hat{v}^k \in\partial^2\varphi(\Hat{x}^k,\Hat{v}^k)(d^k)
$$
\State \text{Set $x^{k+1}:=\Hat{x}^k+  d^k$}
\EndFor
\end{algorithmic}
\end{algorithm}\vspace*{-0.2in}

\begin{Remark}[\bf implementation of the approximate step in Algorithm \ref{NMcod}]\label{apprst} \rm The idea in the approximate step of Algorithm~\ref{NMcod} comes from  the semismooth$^*$ Newton method suggested in \cite{Helmut}.  We will compare our approach with the latter method at the end of this section. Note that formally the approximate step in Algorithm~\ref{NMcod} is always well-defined since $(\Hat{x}^k,\Hat{v}^k)$  can be chosen as $(\ox,0)$. However, in practice the solution $\ox$ is unknown. Depending on the structure of $\varphi$, we will adopt below a specific approach to selecting $\Hat{x}^k$ and $\Hat{v}^k$ that facilitates the implementation of the algorithm.

It follows from \cite[Corollary~2.29]{Mordukhovich06} that the collection of points $(x,v)\in\gph\partial\ph$ with $\partial\ph(x)\ne\emp$ is {\em dense} in the subdifferential graph, and thus we can always select $(\Hat{x}^k,\Hat{v}^k)\in\gph\partial\varphi$ arbitrarily close to $(\ox,0)$. The issue is about the  {fulfillment of estimate} \eqref{stopcri}. Let us show how to do it in the following major settings without knowing the solution $\ox$. 

{\bf(i)} When $\varphi$ is of class $\mathcal{C}^{1,1}$ around $\ox$, there exist a neighborhood $U$ of $\ox$ and  $\ell>0$ such that 
$$
\|\nabla \varphi(x) - \nabla \varphi(\ox)\| \leq \ell \|x-\ox\| \quad \text{for all }\; x \in U.  
$$
Supposing that $x^k\in U$ at the $k^{\text{th}}$ iteration, take
$\Hat{x}^k:= x^k$ and $\Hat{v}^k :=\nabla \varphi(x^k)$, which gives us
$$
\|(\Hat{x}^k,\Hat{v}^k) - (\ox,0)\|  \leq \|\Hat{x}^k -\ox\| + \|\Hat{v}^k\| =\|x^k-\ox\|+ \|\nabla \varphi(x^k)\| \leq (1+\ell)\|x^k -\ox\|
$$
and thus justifies the required condition \eqref{stopcri}.

{\bf(ii)} If $\varphi\colon\R^n\to\oR$ is {continuously $r$-level prox-regular} at $\ox$ for $0$, it follows from Proposition~\ref{C11} that for sufficiently small $\lambda >0$, there is a neighborhood $U_\lambda$ of $\ox$ such that $ \text{\rm Prox}_{\lambda \varphi}$ is single-valued and Lipschitz continuous with modulus $L_\lambda>0$ on $U_\lambda$. Having $x^k \in U_{\lambda}$ at the $k^{\text{th}}$ iteration, choose $\Hat{x}^k$ and $\Hat{v}^k$ as
\begin{equation}\label{choicproxregular} 
\Hat{x}^k:= \text{\rm Prox}_{\lambda \varphi}(x^k) \quad \text{and }\; \Hat{v}^k:= \frac{1}{\lambda}\left(x^k -  \text{\rm Prox}_{\lambda \varphi}(x^k)\right).
\end{equation} 
Therefore, we get the estimates
\begin{align*}
\|(\Hat{x}^k,\Hat{v}^k) - (x^k,0)\|\le\left(1 + \lambda ^{-1}\right)\|x^k - \text{\rm Prox}_{\lambda \varphi}(x^k)\|\\
\leq \left(1 + \lambda ^{-1} \right) (\|x^k -\ox\| + \|\text{\rm Prox}_{\lambda \varphi}(x^k) -\ox\|) \\
=\left(1 + \lambda ^{-1} \right) (\|x^k -\ox\| + \|\text{\rm Prox}_{\lambda \varphi}(x^k) -\text{\rm Prox}_{\lambda \varphi}(\ox)\|) \\
\leq  \left(1 + \lambda ^{-1} \right)  \left(\|x^k-\ox\| +L_\lambda \|x^k-\ox\| \right)= M \|x^k-\ox\|
\end{align*}
with $M:= (1+\lambda^{-1})(1+L_\lambda)$, which yields \eqref{stopcri}.  
In Section~\ref{sec:localNMcompo}, we present a choice of $(\Hat{x}^k,\Hat{v}^k)$ in the approximate step of the version of Algorithm~\ref{NMcod} for minimizing sums of two nonconvex functions. 
\end{Remark}\vspace*{-0.05in}

To justify  the {\em superlinear} local convergence of our generalized Newton method, we need the concept of {\em semismoothness}$^*$ introduced in \cite{Helmut}. To formulate this notion, recall the construction of the {\em directional limiting normal cone} to a set $\Omega\subset\R^s$ at $\oz\in\O$ in the direction $d\in\R^s$ defined in \cite{gin-mor} by
\begin{equation}\label{dir-nc}
N_\Omega(\oz;d):=\big\{v\in\R^s\;\big|\;\exists\,t_k\dn 0,\;d_k\to d,\;v_k\to v\;\mbox{ with }\;v_k\in\widehat{N}_\Omega(\oz+t_k d_k)\big\}.
\end{equation}
It is obvious that \eqref{dir-nc} agrees with the limiting normal cone \eqref{lnc} for $d=0$. The {\em directional limiting coderivative} of $F\colon\R^n\tto\R^m$ at $(\ox,\oy)\in\gph F$ in the direction $(u,v)\in\R^n\times\R^m$ is defined in \cite{g} by
\begin{equation*}
D^*F\big((\ox,\oy);(u,v)\big)(v^*):=\big\{u^*\in\R^n\;\big|\;(u^*,-v^*)\in N_{\text{gph}\,F}\big((\ox,\oy);(u,v)\big)\big\}\;\mbox{ for all }\;v^*\in\R^m,
\end{equation*}
while being instrumental to describe the aforementioned property from \cite{Helmut}.\vspace*{-0.05in}

\begin{Definition}[\bf semismooth$^*$ property of set-valued mappings]\label{semi*} \rm  A mapping $F\colon\R^n\tto\R^m$ is {\sc  semismooth$^*$} at $(\bar{x},\bar{y})\in\gph F$ if whenever $(u,v)\in\R^n\times\R^m$ we have
\begin{equation*}
\langle u^*,u\rangle=\langle v^*,v\rangle\;\mbox{ for all }\;(v^*,u^*)\in\gph D^*F\big((\ox,\oy);(u,v)\big).
\end{equation*}
\end{Definition}

Recall from \cite[Proposition~2.4]{BorisEbrahim}  that for $F: = \nabla \varphi$, where $\varphi:\R^n\to\R$ is of class $\mathcal{C}^{1,1}$ around $\ox$, the semismoothness$^*$ of $F$ at $(\ox,\nabla\varphi(\ox))$ is equivalent to 
\begin{equation}\label{Gsemismoothstar}
\|\nabla \varphi(x) -\nabla \varphi(\ox) + v_x \| = o(\|x-\ox\|) 
\end{equation}
whenever $x \to \ox$ and $v_x \in \partial^2 \varphi(x)(\ox-x)$. This is essential for the next lemma and subsequent results.\vspace*{-0.05in}

\begin{Lemma}[\bf estimates for ${\cal C}^{1,1}$ functions]\label{estimateNMC11} Let $\varphi:\R^n\to\R$ be of class $\mathcal{C}^{1,1}$ around $\ox\in\R^n$. If $\nabla \varphi$ is metrically regular around $(\ox,\nabla \varphi(\ox))$, then there exist $c>0$ and a neighborhood $U$ of $\ox$ such that for any $x \in U$ and $d\in\R^n$ satisfying $-\nabla\varphi(x)\in\partial^2\varphi(x)(d)$, we  find $v_x \in \partial^2\varphi(x)(\ox-x)$ with
\begin{equation}\label{NMC11esti1} 
\|x+d-\ox\| \leq c\|\nabla \varphi(x) +v_x\|. 
\end{equation}
If $\nabla \varphi(\ox)=0$ and $\nabla\varphi$ is semismooth$^*$ at $(\ox,0)$, then for every $\epsilon>0$ there is a neighborhood $U$ of $\ox$  such that 
\begin{equation}\label{NMC11esti2}
\|x+d-\ox\| \leq \epsilon\|x-\ox\| \quad \text{whenever }\; x \in U\;\mbox{ and }\;-\nabla\varphi(x)\in\partial^2\varphi(x)(d). 
\end{equation}
\end{Lemma}
\begin{proof} By the metric regularity of $\nabla \varphi$ around $(\ox,\nabla\varphi(\ox))$ and
its characterization from \cite[Lemma 5.1]{BorisKhanhPhat}, we find $c > 0$ and a neighborhood $U$ of $\ox$ satisfying
\begin{equation}\label{anothercharmetric}
\|v\| \geq c^{-1}\|u\| \quad \text{for all }\; v \in \partial^2 \varphi(x)(u),\; x \in U, \; u \in \R^n.  
\end{equation} 
Using the subadditivity of coderivatives from \cite[Lemma 5.3]{BorisKhanhPhat} tells us that
$$
\partial^2 \varphi(x)(d)= \partial^2 \varphi(x)(x+d-\ox+\ox - x) \subset \partial^2\varphi(x)(x+d-\ox)+\partial^2\varphi(x)(\ox-x). 
$$
Hence for any $x \in U$ and $d\in \R^n$ satisfying $-\nabla \varphi(x) \in \partial^2\varphi(x)(d)$, we find $v_x \in \partial^2 \varphi(x)(\ox-x)$ with
$$
-\nabla \varphi(x) -v_x \in \partial^2 \varphi(x)(x+d-\ox). 
$$
Furthermore, it follows from \eqref{anothercharmetric} that 
$$
\|x+d-\ox\| \leq c\|\nabla \varphi(x) +v_x\|
$$
verifying \eqref{NMC11esti1}. Under the additional assumptions of the lemma, we deduce from \eqref{Gsemismoothstar} that 
$$
\|\nabla \varphi(x) +v_x\|=\|\nabla \varphi(x) -\nabla\varphi(\ox) +v_x\| = o(\|x-\ox\|) \quad \text{as }\;x \to \ox,
$$
which justifies \eqref{NMC11esti2} and thus completes the proof. 
\end{proof}

Estimate \eqref{NMC11esti1} was proved in \cite{BorisEbrahim} under the assumption that $\ox$ is a tilt-stable local minimizer of $\varphi$, which is stronger than the metric regularity of $\partial\varphi$ around $(\ox,0)$ assumed in Lemma~\ref{estimateNMC11}. Moreover, our proof of \eqref{NMC11esti1} is fully independent from the proof of \cite{BorisEbrahim}.\vspace*{0.05in} 

The following technical lemma is borrowed from \cite[Lemma 6.1]{BorisKhanhPhat}.\vspace*{-0.07in} 
 
\begin{Lemma}[\bf second-order subdifferentials of Moreau envelopes]\label{2ndofMoreau} In the setting of Proposition~{\rm\ref{C11}}, {take any 
$\lambda\in(0,\bar{\lambda})$}, $x\in U_\lambda$ with $v = \nabla e_\lambda \varphi(x)$. Then we have the equivalence 
$$
x^* \in \partial^2 e_\lambda \varphi(x,v)(v^*) \iff x^* \in \partial^2 \varphi(x-\lambda v, v)(v^*-\lambda x^*).
$$
\end{Lemma}\vspace*{-0.07in}

Our next theorem provides a crucial estimate, which is significant for the subsequent justification of the coderivative-based generalized Newton algorithm.\vspace*{-0.07in}

\begin{Theorem}[\bf estimates of Newton iterations via coderivatives] \label{estiNMitercod} Let $\ox$ solve \eqref{sub-incl}, where $\partial\varphi$ is metrically regular around $(\ox,0)$ and semismooth$^*$ at this point. Assume further that $\varphi$ is continuously prox-regular at $\ox$ for $0$. Then for any $\epsilon>0$, there exists a neighborhood $W$ of $(\ox,0)$ such that 
$$ 
\|x +d - \ox\| \leq \epsilon \|(x,v)-(\ox,0)\|
$$
{whenever $(x,v)\in \gph\partial\varphi \cap W$} and $d \in \R^n$ satisfy $-v\in\partial^2\varphi(x,v)(d)$. 
\end{Theorem}\vspace*{-0.15in}
\begin{proof} Proposition~\ref{C11} tells us that the assumed  prox-boundedness and continuous prox-regularity properties of $\ph$ guarantee the {existence of a  sufficiently small $\lambda \in (0,1)$} such  that  there is a neighborhood $U_\lambda$ of  $\bar{x} $ on which $e_\lambda\varphi$ is of class $\mathcal{C}^{1,1}$, that the proximal mapping $\text{\rm Prox}_{\lambda\varphi}$ is single-valued and Lipschitz continuous on $U_\lambda$ with ${\rm Prox}_{\lm\ph}(\ox)=\ox$, and that the gradient expression for the Moreau envelope \eqref{GradEnvelope} holds. Since $\partial\varphi$ is semismooth$^*$ at $(\ox,0)$, it follows from  \cite[Theorem 6.2, Claim 2]{BorisEbrahim} that $\nabla e_\lambda \varphi$ is semismooth$^*$ at $(\ox,0)$. Given $\epsilon>0$, we find a neighborhood $U$ of $\ox$ on which we get by Lemma~\ref{estimateNMC11} that
\begin{equation}\label{nextiterenvelope}
\|x^\prime + d^\prime - \ox\| \leq \frac{\epsilon}{\sqrt{2}}\|x^\prime -\ox\|\;\text{ if }\; x^\prime \in U\;\mbox{ and }\;d^\prime\in \R^n \; \text{ satisfy }\;-\nabla e_\lambda\varphi(x^\prime) \in \partial^2e_\lambda\varphi(x^\prime)(d^\prime).  
\end{equation}
Choose further a neighborhood $W$ of $(\ox,0)$ such that
\begin{equation}\label{neighborNM}
x+\lambda v\in U\cap U_\lambda\;\mbox{ whenever }\; (x,v) \in W 
\end{equation}
and suppose that $(x,v)\in\gph\partial\varphi \cap W$ and that $d \in \R^n$ satisfies
\begin{equation}\label{2ndinclus}
-v\in\partial^2\varphi(x,v)(d)=\partial^2 \varphi((x+\lambda v) - \lambda v,v)(d).
\end{equation}
Then \eqref{GradEnvelope} and \eqref{neighborNM} yield 
$\nabla e_\lambda\varphi(x+\lambda v)=v$. Combining this with \eqref{neighborNM}, \eqref{2ndinclus}, and Lemma~\ref{2ndofMoreau} gives us 
$$
-\nabla e_\lambda \varphi(x+\lambda v)=-v \in \partial^2 e_\lambda\varphi(x+\lambda v, v)(d-\lambda v).
$$
{This ensures, with taking \eqref{nextiterenvelope} and \eqref{neighborNM} into account, that 
$$
\|x+d-\ox\|=\|(x +\lambda v) +(d-\lambda v) -\ox\| \leq \frac{\epsilon}{\sqrt{2}}\|x+\lambda v - \ox\|. 
$$
Therefore, we arrive at the estimates
$$
\|(x,v)- (\ox,0)\|^2 = \|x-\ox\|^2 + \|v\|^2 \geq \|x-\ox\|^2 + \|\lambda v\|^2  \geq \frac{1}{2}(\|x-\ox\| +  \|\lambda v\|)^2,
$$
which imply in turn that 
$$
\|x + d-\ox\| \leq \frac{\epsilon}{\sqrt{2}}\|x+\lambda v - \ox\| \leq \frac{\epsilon}{\sqrt{2}}\left(\|x-\ox\|+ \|\lambda v\| \right) \leq \epsilon \|(x,v)-(\ox,0)\|
$$}
and thus complete the proof of theorem. 
\end{proof} \vspace*{-0.05in}

Now we are ready to establish the main result of this section justifying the local superlinear convergence of the Newtonian coderivative-based Algorithm~\ref{NMcod}. \vspace*{-0.05in}

\begin{Theorem}[\bf local superlinear convergence of the generalized Newton algorithms] \label{localconvergeNMcod}  {Let $\ox$ solve \eqref{sub-incl}, where $\partial\varphi$ is metrically regular around $(\ox,0)$ and semismooth$^*$ at this point. Suppose that $\varphi$ is continuously prox-regular at $\ox$ for $0$.} Then there is a neighborhood $U$ of $\ox$ such that, whenever $x^0\in U$, Algorithm~{\rm\ref{NMcod}} either stops after finitely many iterations, or produces a sequence  $\{x^k\}$ that $Q$-superlinearly converges to $\ox$ as $k\to\infty$. 
\end{Theorem}\vspace*{-0.15in}
\begin{proof}  {Due to the robustness of the continuous prox-regularity of $\varphi$ at $\ox$ for $0$ and the metric regularity of $\partial \varphi$ around $(\ox,0)$, Theorem \ref{newsolNM} implies that for each $(x,v)\in\gph\partial\ph$ near $(\ox,\ov)$, there exists a direction $d\in\R^n$ satisfying the generalized Newton system \eqref{newton-inc}. Combining this with Theorem~\ref{estiNMitercod} allows us to find such $\delta >0$ that for any $x \in \mathbb{B}_\delta(\ox)$ there exists $(\Hat{x},\Hat{v})\in \gph\partial\varphi$ satisfying the estimate}
\begin{equation}\label{etaapprox2}
\|(\Hat{x},\Hat{v})-(\ox,0)\| \leq \eta \|x-\ox\|,
\end{equation}
and  that for any $(x,v) \in \gph\partial\varphi\cap \mathbb{B}_\delta(\ox,0)$ we get $d \in \R^n$ with 
\begin{equation}\label{directionmetricrgl}
-v\in \partial^2\varphi(x,v)(d) \;\text{and } \|x +d - \ox\| \leq \frac{1}{2\eta} \|(x,v)-(\ox,0)\|.
\end{equation}
Fix $x^0 \in U:=\mathbb{B}_{\delta/\eta}(\ox)\cap \mathbb{B}_\delta(\ox)$ and deduce from \eqref{etaapprox2} the existence of $\left(\Hat{x}^0,\Hat{v}^0\right)\in\gph\partial\varphi$ such that 
$$
\left\|\left(\Hat{x}^0,\Hat{v}^0\right)-(\ox,0)\right\|\leq \eta\|x^0-\ox\|\leq \delta,
$$
which readily implies that $\left(\Hat{x}^0,\Hat{v}^0\right) \in \gph\partial\varphi\cap \mathbb{B}_\delta(\ox,0)$. Furthermore,
by \eqref{directionmetricrgl} we find $d^0\in \R^n$ satisfying $-\Hat{v}^0 \in \partial^2 \varphi(\Hat{x}^0,\Hat{v}^0)(d^0)$, and hence
$$
\|x^1-\ox\|=\|\Hat{x}^0 + d^0 -\ox\| \leq \frac{1}{2\eta}\|(\Hat{x}^0,\Hat{v}^0) -(\ox,0)\| \leq  \frac{1}{2}\|x^0 -\ox\|.
$$
The latter ensures, in particular, that $x^1 \in U$. Employing the standard induction arguments allows us to conclude that Algorithm~\ref{NMcod}  either stops after finitely many iterations, or produces  sequences  $\{x^k\}$, $\{\Hat{x}^k\}$, and $\{\Hat{v}^k\}$ so that the pairs $(\Hat{x}^k,\Hat{v}^k)\in \gph\partial \varphi\cap \mathbb{B}_\delta(\ox,0)$ satisfy the estimates
$$
\|(\Hat{x}^k,\Hat{v}^k)-(\ox,0)\| \leq \eta \|x^k - \ox\|,\quad
\|x^{k+1}-\ox\|\leq \frac{1}{2}\|x^k -\ox\|\; \text{ for all } \; k =0,1,\ldots. 
$$
This clearly implies that both sequences $\{x^k\}$ and $\{\Hat{x}^k\}$ converge to the solution  $\ox$  of \eqref{sub-incl} while the sequence $\{\Hat{v}^k\}$ converges to $0$ as $k \to \infty$. To verify finally the $Q$-superlinear convergence of the iterative sequence $\{x^k\}$, pick any $\epsilon >0$ and deduce from Theorem~\ref{estiNMitercod} that 
$$
\|x^{k+1}-\ox\|=\|\Hat{x}^k + d^k -\ox\| \leq\frac{\epsilon}{\eta}\|(\Hat{x}^k, \Hat{v}^k) -(\ox,0) \| \leq \epsilon\|x^k -\ox\| 
$$
when $k$ is sufficiently large. This means that the quotient $\|x^{k+1}-\ox\|/ \|x^k-\ox\|$ converges to $0$ as $k \to\infty$, which thus completes the proof of the theorem.
\end{proof}

We conclude this section with brief discussions on some recent developments of generalized Newton methods to
solve set-valued inclusions involving the subgradient systems \eqref{sub-incl}.\vspace*{-0.05in}

\begin{Remark}[\bf comparison with some local Newton-type methods to solve inclusions] \rm $\,$

{\bf(i)} The semismooth$^*$ Newton method of \cite{Helmut} addresses solving the inclusion $0 \in F(x)$ for $F:\R^n\rightrightarrows \R^n$, which reduces to \eqref{sub-incl} when $F:=\partial\varphi$ for some $\varphi:\R^n\to\overline{\R}$. The underlying difference between our approach and the method of \cite{Helmut} is in finding \textit{generalized Newton directions}. To be more specific, the directions {$d^k$ in the semismooth$^*$ Newton method are given by}
$$
d^k := -A^{-1}B \Hat{v}^k, \quad k =0,1,\ldots,
$$
where the pair $(\Hat{x}^k,\Hat{v}^k)\in\gph\partial\varphi$ is selected for each $k$ from \eqref{stopcri}, and where the matrices $A, B \in \R^{n\times n}$  are generated as follows. Choose $(v_i^*,u_i^*)\in\gph D^*\partial\varphi(\Hat{x}^k,\Hat{v}^k)$ such that the $i$th rows of $A$ and $B$ are $u_i^*$ and $v_i^*$, respectively, and that $A$ is nonsingular. However, explicit constructions for generating such matrices remain unclear in the case where $F$  is a set-valued mapping. In \cite{Helmut}, such constructions are provided only for the case where $F$ is a single-valued and locally Lipschitzian mapping. Meanwhile, our new method provides explicit ways to select generalized Newton directions when $\varphi$ belongs to a significantly large class of extended-real-valued functions as discussed in Remark~\ref{findNMremark}.  Moreover, the semismooth$^*$ Newton method in \cite{Helmut} is only guaranteed to converge under the assumption of strong metric regularity, while Theorem~\ref{localconvergeNMcod}  shows that the local superlinear convergence of Algorithm~\ref{NMcod} holds under merely metric regularity, which is a less restrictive condition. To overcome this disadvantage of the semismooth$^*$ Newton method, one of the possible ideas developed in the more recent paper \cite{GfOu22} is to construct generalized Newton directions by using SCD (subspace containing derivative) mappings instead of coderivatives. This adaption requires a certain SCD regularity, which is weaker than the metric regularity of the subgradient mapping $\partial\varphi$. However, calculus rules for SCD mappings are more limited in comparison with the extensive  calculus available for coderivatives. 

{\bf (ii)} Our previous paper \cite{BorisKhanhPhat} develops local coderivative-based Newton methods to solve inclusions of the form $0 \in \partial \varphi(x)$. The algorithms proposed therein heavily rely on the proximal mapping of $\varphi$, which is not always easy to compute. Moreover, in \cite{BorisKhanhPhat} we focus on the case where $\varphi = f + g$, with $f$ being quadratic and $g$ prox-regular. In contrast, our new Algorithm~\ref{NMcod} does not require such structural assumptions. 
\end{Remark}\vspace*{-0.25in}

\section{Local Convergence Analysis of Coderivative-Based Newton Method for Structured Nonconvex Optimization}\label{sec:localNMcompo}\vspace*{-0.05in}

As discussed in Section~\ref{sec:general}, Algorithm~\ref{NMcod} offers a variety of possibilities to implement the \textit{approximate step}. In this section, we present a way to choose the pair $(\Hat{x}^k, \Hat{v}^k)\in\gph\partial\varphi$ in the approximate step of Algorithm~\ref{NMcod} for solving the class of {\em nonsmooth and nonconvex optimization problems} of the type
\begin{equation}\label{ncvopt}
\min \quad \varphi(x):=f(x) +g(x)\quad \text{subject to }\; x \in \R^n, 
\end{equation}
where $f:\R^n \to \R$ is {\em $\mathcal{C}^2$-smooth}, and where $g:\R^n\to\overline{\R}$ is proper, l.s.c., and {\em prox-bounded} with threshold $\lambda_g\in (0,\infty]$. We start with the corresponding definitions needed in what follows.\vspace*{-0.03in}

\begin{Definition}[\bf prox-gradient and normal mappings]\label{normal map}\rm Let $\varphi\colon\R^n\to\oR$ {be given as in} \eqref{ncvopt}. For each $\lambda \in(0,\lambda_g)$, define the following multifunctions:

{\bf(i)} The \textit{prox-gradient map} $\mathcal{G}_\lambda:\R^n\rightrightarrows\R^n$ associated with \eqref{ncvopt} is
\begin{equation}\label{prox-grad-map}
\mathcal{G}_\lambda(x):= \text{\rm Prox}_{\lambda g}\big(x-\lambda \nabla f(x)\big), \quad x \in \R^n. 
\end{equation}

{\bf(ii)} The \textit{normal map} $\mathcal{F}_\lambda^{\text{\rm nor}}: \R^n\rightrightarrows \R^n$ associated with \eqref{ncvopt} is
\begin{equation}\label{normalmap}
\mathcal{F}_\lambda^{\text{\rm nor}}(x):= \Big\{\nabla f(\Hat{x})+ \frac{1}{\lambda}(x- \Hat{x})\;\Big|\; \Hat{x}\in \text{\rm Prox}_{\lambda g} (x) \Big\},\quad x \in \R^n. 
\end{equation}
\end{Definition} 

When the proximal mapping $\text{\rm Prox}_{\lambda g}$ is single-valued, the normal map \eqref{normalmap} can be written as 
\begin{equation}\label{normalmapsingle} 
\mathcal{F}_\lambda^{\text{\rm nor}}(x)= \nabla f\big(\text{\rm Prox}_{\lambda g}(x)\big)+ \frac{1}{\lambda}\big(x- \text{\rm Prox}_{\lambda g}(x)\big),
 \quad x \in \R^n. 
\end{equation} 
The terminology ``normal map" goes back to variational inequalities being useful there for characterizing solution sets; see, e.g., \cite{JPang}. In the setting of variational inequalities
\begin{equation}\label{VI}
\text{find } \ox \in \Omega\;\text{ such that }\;\langle F(\ox),x-\ox\rangle \geq 0\;\mbox{ for all }\;x\in \Omega, 
\end{equation}
where $F:\R^n\to\R^n$ is a single-valued operator, and where $\Omega$ is a closed and convex set, the normal map associated with \eqref{VI} is defined by 
\begin{equation}\label{normalmapVI}
\mathcal{F}^{\text{\rm nor}}(x):=  F\big(\Pi_\Omega(x)\big)+x-\Pi_\Omega(x),
\quad x \in \R^n. 
\end{equation}

Note that both normal maps in \eqref{normalmap} and \eqref{normalmapVI} agree when $F:=\nabla f$, $g:=\delta_\Omega$, and $\lambda=1$. Various error bounds for variational inequalities based on the normal map \eqref{normalmapVI} can be found in the book \cite{JPang} and the references therein, while new ones by using this normal map have been recently obtained in \cite{KhanhKha} for the class of strongly pseudomonotone variational inequalities.  {One of the first extensions of the normal map to composite problems was introduced by Pieper in his PhD thesis \cite{p15}. His approach combines semismooth Newton steps for solving the nonsmooth equation $\mathcal{F}^{\text{nor}}_\lambda (x) = 0$ with a trust-region-type globalization mechanism. The subsequent paper \cite{mr20} analyzed local properties of a quasi-Newton variant of Pieper’s normal map-based trust-region method. Pieper’s approach has also been successfully applied in various contexts, particularly in optimal control; see, e.g., \cite{bt17,kpr16,raks18}. A more comprehensive convergence analysis, covering both local and global aspects of the normal map trust-region method, was recently presented in \cite{om24}. Additionally, algorithmic extensions of normal map-based techniques to stochastic settings have been proposed in \cite{mq23,qlm23}. In \cite{om24}, the authors introduce the normal map as an intermediate step to find a solution to the problem of minimizing the sum of a smooth and a nonsmooth convex functions by applying the semismooth Newton method to the equation $\mathcal{F}_\lambda^{\text{\rm nor}}(z) = 0$, where the local Lipschitz continuity of this mapping is ensured by the convexity of the regularizer. In contrast, our approach bypasses this transformation and directly tackles the optimization problem, where the objective function is the sum of two nonconvex functions. In this framework, the normal map plays a critical role in implementing the approximate step of Algorithm~\ref{NMcod} and in supporting the development of our new generalized Newton and line-search methods introduced in this section and further detailed in Sections~\ref{sec:newlinesearch} and~\ref{sec:globalNMcompo}.}\vspace*{0.05in}

Recall that the stationarity condition $0\in\partial\varphi(\ox)$, where $\varphi$ is given in form \eqref{ncvopt}, is only a necessary condition for the optimality of $\ox$ being sufficient if $\varphi$ is convex. In what follows, we consider different concepts of stationary and critical points for the function $\varphi$ of type \eqref{ncvopt} along the lines of general optimization theory.\vspace*{-0.05in}

\begin{Definition}[\bf stationary and critical points]\label{stat} \rm  {Let $\varphi$ {be given as in} \eqref{ncvopt}, and let $\ox \in \dom \varphi$. Then $\ox$ is:}

{\bf(i)} \textit{F$($r\'echet$)$-stationary point} if $0\in\widehat{\partial}\varphi(\ox)$.

{\bf(ii)} \textit{M$($ordukhovich$)$-stationary point} if $0\in {\partial}\varphi(\ox)$.

{\bf(iii)} {\textit{$\lambda$-critical point}  if $\ox \in \mathcal{G}_\lambda(\ox)$ where $\lambda \in (0,\lambda_g)$.} 
\end{Definition}\vspace*{-0.1in}

\begin{Remark} \label{crit->Fsta}\rm 
Observe the following:

\medskip 
{\bf (i)} If $\ox$ is a $\lambda$-critical point of \eqref{ncvopt} for some $\lambda \in (0,\lambda_g)$, then $\ox$ is also a $\lambda^\prime$-critical point of \eqref{ncvopt} for any $\lambda^\prime \in (0,\lambda].$ Indeed, fix such $\lambda^\prime$, the $\lambda$-criticality of $\ox$ implies that $\ox \in \mathcal{G}_\lambda(\ox)$, i.e.,
$$
g(x) + \frac{1}{2\lambda} \|x - \ox +\lambda \nabla f(\ox)\|^2  \geq g(\ox) + \frac{1}{2\lambda}\|\ox -\ox+\lambda \nabla f(\ox)\|^2 \quad \text{for all }\; x \in \R^n,
$$
which ensures in turn that 
$$
g(x) \geq g(\ox) +\langle -\nabla f(\ox), x-\ox\rangle -  \frac{1}{2\lambda}\|x-\ox\|^2 \geq g(\ox) +\langle -\nabla f(\ox), x-\ox\rangle -  \frac{1}{2\lambda^\prime }\|x-\ox\|^2 
$$
for all $x \in \R^n,$ i.e., we have the estimate
$$
g(x) + \frac{1}{2\lambda^\prime} \|x - \ox +\lambda^\prime \nabla f(\ox)\|^2  \geq g(\ox) + \frac{1}{2\lambda^\prime}\|\ox -\ox+\lambda^\prime \nabla f(\ox)\|^2 \quad \text{for all }\; x \in \R^n. 
$$
This means that $\ox \in \mathcal{G}_{\lambda^\prime}(\ox)$, or $\ox$ is a $\lambda^\prime$-critical point of \eqref{ncvopt}. 

{\bf (ii)} If $\ox$ is a {$\lambda$-critical point of \eqref{ncvopt} for some $\lambda \in (0,\lambda_g)$}, then $\ox$ is a F-stationary (and hence M-stationary) point. Indeed, we have from the critical point definition that
$$
\ox \in \text{\rm argmin} \Big\{g(x) + \frac{1}{2\lambda}\|x - \ox +\lambda \nabla f(\ox)\|^2 \; \Big|\; x \in \R^n\Big\}.
$$
It then follows from \cite[Proposition~1.114]{Mordukhovich06} that
$$
0 \in \widehat{\partial} g(\ox) + \frac{1}{\lambda}\big(\ox-\ox +\lambda \nabla f(\ox)\big) = \nabla f(\ox) + \widehat{\partial} g(\ox) = \widehat{\partial}\varphi(\ox) \subset \partial\varphi(\ox).
$$
If $f$ and $g$ are convex and $\ox$ is a {$\lambda$-critical point of \eqref{ncvopt} for some $\lambda >0$}, then $\ox$ is an optimal solution to \eqref{ncvopt}; see \cite[Corollary~27.3]{Bauschke}. 
\end{Remark}\vspace*{-0.05in}

Next we show that the criticality of $\ox$ is a necessary optimality condition in \eqref{ncvopt}.\vspace*{-0.07in}

\begin{Proposition}[\bf necessary condition for nonconvex optimization problems] \label{necessncvopt} Let $\varphi$ be given as in \eqref{ncvopt}, and let $\ox$ be an optimal solution to problem \eqref{ncvopt}. Assume that $\nabla f$ is  {Lipschitz continuous on $\R^n$} with modulus $L_f>0$. Then $\ox$ is a {$\lambda$-critical point} of \eqref{ncvopt}, and $\mathcal{G}_\lambda$ is single-valued at $\ox$ for any $\lambda \in (0,\text{\rm min}\{ 1/L_f, \lambda_g\})$. 
\end{Proposition}\vspace*{-0.15in}
\begin{proof} Fix $\lambda \in (0,\lambda_g)$. Since $g$ is prox-bounded with threshold $\lambda_g$, we have by Proposition~\ref{pbimplycontinuity} that $\text{\rm Prox}_{\lambda g}(\ox- \lambda \nabla f(\ox))$ is a nonempty set. Picking any of $\Hat{x}$ from this set, let us to show that $\Hat{x} = \ox$. Indeed, the Lipschitz continuity of $\nabla f$ ensures by the descent lemma in \cite[Lemma~A.11]{Solo14} that
\begin{equation}\label{descentlemma} 
f(\Hat{x}) \leq f(\ox)+\langle \nabla f(\ox),\Hat{x}-\ox\rangle + \frac{L_f}{2}\|\Hat{x}-\ox\|^2.
\end{equation} 
Taking into account that $\Hat{x}\in \text{\rm Prox}_{\lambda g}(\ox- \lambda \nabla f(\ox))$, we have the estimate
\begin{equation}\label{proxgdef}
g(\Hat{x})+ \frac{1}{2\lambda}\|\Hat{x}- (\ox -\lambda \nabla f(\ox))\|^2 \leq g(\ox) + \frac{1}{2\lambda}\|\ox- (\ox -\lambda \nabla f(\ox))\|^2= g(\ox) +\frac{\lambda}{2}\|\nabla f(\ox)\|^2.
\end{equation}
Since $\ox$ is an optimal solution to \eqref{ncvopt} and $g$ is proper, it tells us that $g(\ox)$ is finite, and so is $g(\Hat x)$ by \eqref{proxgdef}.  
Combining \eqref{descentlemma} and \eqref{proxgdef} provides the conditions 
\begin{align*}
f(\Hat{x})+g(\Hat{x})& \leq f(\ox)+\langle \nabla f(\ox),\Hat{x}-\ox\rangle + \frac{L_f}{2}\|\Hat{x}-\ox\|^2 + g(\Hat{x})\\
& = f(\ox) + g(\Hat{x})+ \frac{1}{2\lambda}\|\Hat{x}- (\ox -\lambda \nabla f(\ox))\|^2 - \frac{\lambda}{2} \|\nabla f(\ox)\|^2 +\left(\frac{L_f}{2} - \frac{1}{2\lambda} \right)\|\Hat{x}-\ox\|^2\\
& \leq f(\ox) +g(\ox) + \left(\frac{L_f}{2} - \frac{1}{2\lambda} \right)\|\Hat{x}-\ox\|^2\leq f(\Hat{x}) +g(\Hat{x}) + \left(\frac{L_f}{2} - \frac{1}{2\lambda} \right)\|\Hat{x}-\ox\|^2,
\end{align*}
which ensure that $\|\Hat{x}-\ox\|^2 \leq 0$, and thus $\Hat{x}=\ox$. Since $\Hat{x}$ was chosen arbitrarily, the set $\text{\rm Prox}_{\lambda g}(\ox- \lambda \nabla f(\ox))$ is a singleton, and $\ox$ is its unique element as claimed. 
\end{proof}\vspace*{-0.05in}

Using Remark~\ref{crit->Fsta} and Proposition~\ref{necessncvopt} gives us the implications
$$
\text{optimal solution} \Longrightarrow \text{critical point} \Longrightarrow  \text{F-stationary point}  \Longrightarrow  \text{M-stationary point}. 
$$
To proceed further, we involve the normal maps \eqref{normalmap} to characterize critical points.\vspace*{-0.05in}

\begin{Proposition}[\bf characterization of critical points via normal maps] \label{characsol} Let $\varphi$ {be given as in} \eqref{ncvopt}. Then for each $\lambda \in (0,\lambda_g)$, the following assertions hold:

{\bf(i)} If $\ox$ is a {$\lambda$-critical point} of \eqref{ncvopt}, then $0 \in \mathcal{F}_\lambda^{\text{\rm nor}}(\ox-\lambda \nabla f(\ox))$.

{\bf(ii)} If $\nabla f$ is Lipschitz continuous around $\ox$ with modulus $L_f>0$ and if  $0 \in\mathcal{F}_\lambda^{\text{\rm nor}}(\ox-\lambda \nabla f(\ox))$ with $\lambda \in (0,1/L_f)$, then $\ox$ is a {$\lambda$-critical point} of \eqref{ncvopt}.
\end{Proposition}\vspace*{-0.15in}
\begin{proof} To verify (i), we get by the {$\lambda$-criticality} $\ox$ in \eqref{ncvopt} that $\ox \in \mathcal{G}_\lambda (\ox)$, which yields
$$
0=\nabla f(\ox) + \frac{1}{\lambda}\big(\ox -\lambda \nabla f(\ox) - \ox\big) \in \mathcal{F}_\lambda^{\text{\rm nor}}\big(\ox-\lambda \nabla f(\ox)\big).
$$
To prove (ii), deduce from $0 \in \mathcal{F}_\lambda^{\text{\rm nor}}(\ox-\lambda \nabla f(\ox))$ that there exists $\Hat{x}\in \mathcal{G}_\lambda(\ox)$ satisfying
\begin{equation}\label{hatxoxequal1}
0 = \nabla f(\Hat{x}) + \frac{1}{\lambda}\big(\ox-\lambda \nabla f(\ox) -\Hat{x}\big) = \nabla f(\Hat{x}) - \nabla f(\ox) + \frac{1}{\lambda}(\ox-\Hat{x}).
\end{equation}
Due to the Lipschitz continuity of $\nabla f$ with modulus $L_f>0$ around $\ox$, it follows from \eqref{hatxoxequal1} that
$$
\frac{1}{\lambda}\|\Hat{x}-\ox\|=\|\nabla f(\Hat{x})-\nabla f(\ox)\| \leq L_f \|\Hat{x}-\ox\|,
$$
which tells us by $0<\lambda < 1/L_f$ that $\|\Hat{x}-\ox\|=0$, i.e., $\Hat{x}=\ox$. Thus we get $\ox\in\mathcal{G}_\lambda(\ox)$.
\end{proof}\vspace*{-0.05in}

The next proposition allows us to generate a point in the graph of the subgradient mapping for $\ph$ in \eqref{ncvopt} directly from the prox-gradient \eqref{prox-grad-map} and normal \eqref{normalmap} maps. It plays a crucial role in designing and justifying our new algorithm below in this section.\vspace*{-0.05in}

\begin{Proposition}[\bf generating elements of subgradient graphs]\label{xvingraph} Let $\varphi$ {be given as in}  \eqref{ncvopt}. For each $\lambda \in (0,\lambda_g)$ and $x \in \R^n$, take $\Hat{x} \in \mathcal{G}_\lambda(x)$ and denote
$$
\Hat{v}:= \nabla f(\Hat{x})- \nabla f(x)+ \frac{1}{\lambda}(x-\Hat{x}).
$$
Then we have $\Hat{v} \in \mathcal{F}_\lambda^{\text{\rm nor}}(x-\lambda \nabla f(x))$ and $(\Hat{x},\Hat{v})\in \gph \widehat{\partial}\varphi$.
\end{Proposition}
\begin{proof} The first inclusion follows immediately from the definitions of the prox-gradient and normal maps. To verify the second one, deduce from $\Hat{x}\in\mathcal{G}_\lambda(x)$ that
$$
\Hat{x} \in \text{\rm argmin}\Big\{g(u) + \frac{1}{2\lambda}\|u - x +\lambda \nabla f(x)\|^2 \; \Big|\; u \in \R^n\Big\}.
$$
It follows from  the elementary calculus rules for regular subgradients \cite[Proposition~1.30]{Mor18} that  
$$
0 \in \widehat{\partial} g(\Hat{x}) + \frac{1}{\lambda}\big(\Hat{x}-x +\lambda \nabla f(x)\big) = \nabla f(x) + \widehat{\partial} g(\Hat{x}) + \frac{1}{\lambda}(\Hat{x}-x),
$$
which implies in turn that 
$$
\Hat{v} = \nabla f(\Hat{x})- \nabla f(x)+ \frac{1}{\lambda}(x-\Hat{x}) \in \nabla f(\Hat{x}) + \widehat{\partial} g(\Hat{x}) =\widehat{\partial}\varphi(\Hat{x}), 
$$
and therefore completes the proof of the proposition. 
\end{proof}\vspace*{-0.05in}

Now we involve the prox-regularity of the regularizer in \eqref{ncvopt}, which allows us to justify the single-valuedness of the prox-gradient and normal maps.\vspace*{-0.05in}

\begin{Proposition}[\bf single-valuedness of prox-gradient and normal maps]\label{singlevalueproxmap} Let $\ox$ be an {M-stationary point of \eqref{ncvopt}, and let $g$ be  $r$-level prox-regular at $\ox$ for $\ov:=-\nabla f(\ox)$ for some $r>0$}. {Suppose that $\ox$ is a $\bar{\lambda}$-critical point of \eqref{ncvopt} for some $\bar{\lambda} \in (0,1/r].$ Then for each $\lambda \in \left(0,  \bar{\lambda} \right)$,} there exists a neighborhood $U$ of $\ox$ on which the maps $x\mapsto \mathcal{G}_\lambda(x)$ and $x\mapsto \mathcal{F}_\lambda^{\text{\rm nor}}(x-\lambda \nabla f(x))$ are single-valued and Lipschitz continuous on $U$. Moreover, we find $\eta_1,\eta_2 >0$ such that
\begin{equation}\label{Fnorest}
\|\mathcal{G}_\lambda(x)- \ox\| \leq \eta_1 \|x-\ox\|\;\text{ and }\; 
\|\mathcal{F}_\lambda^{\text{\rm nor}}(x-\lambda \nabla f(x)) \|\leq     \eta_2\|x-\ox\|\;\text{ for all }\;x \in U.
\end{equation}
\end{Proposition} \vspace*{-0.15in}
\begin{proof} Since $f$ is $\mathcal{C}^2$-smooth, there is a neighborhood $\Tilde{U}$ of $\ox$ such that $\nabla f$ is Lipschitz continuous on $\Tilde{U}$. Since $\ox$ is a {$\bar{\lambda}$-critical point of \eqref{ncvopt}, then $\ox \in \mathcal{G}_{\bar{\lambda}} (\ox)=\text{\rm Prox}_{\bar{\lambda} g}\big(\ox+ \bar{\lambda}\ov\big)$. 
It follows from Proposition~\ref{C11} and Remark~\ref{choiceMr}(iii) that for any $\lambda \in (0,\bar{\lambda})$,  the proximal mapping $\text{\rm Prox}_{\lambda g}$ is single-valued  and locally Lipschitzian on a neighborhood $U_\lambda$ of $\ox +\lambda\ov $ and that $\text{\rm Prox}_{\lambda g}(\ox+\lambda\ov)=\ox$.} Choose a neighborhood $U\subset \Tilde{U}$ of $\ox$ so that 
\begin{equation} \label{h(U)} 
(I-\lambda \nabla f)(U) \subset  U_\lambda\; \text {and }\;\text{\rm Prox}_{\lambda g}\big((I-\lambda \nabla f)(U)\big) \subset \Tilde{U},
\end{equation}
which implies that the prox-gradient map $\mathcal{G}_\lambda$ is single-valued and Lipschitz continuous on $U$.  By \eqref{normalmapsingle} and \eqref{h(U)}, the normal map $x\mapsto\mathcal{F}_\lambda^{\text{\rm nor}}(x-\lambda \nabla f(x))$ is single-valued  and satisfies the equality 
$$
\mathcal{F}_\lambda^{\text{\rm nor}}\big(x-\lambda \nabla f(x)\big) = \nabla f\big(\mathcal{G}_\lambda(x)\big) - \nabla f(x) + \frac{1}{\lambda}\big(x-\mathcal{G}_\lambda(x)\big)\;\mbox{ for all }\;x\in U,
$$
which also implies that this map is Lipschitz continuous on $U$. 

To verify further \eqref{Fnorest}, let $\ell>0$ and $L>0$ be Lipschitz moduli of $\mathcal{G}_\lambda$ and $\nabla f$ on $U$ and $\Tilde{U}$, respectively.  Since $\mathcal{G}_\lambda(\ox)=\ox$, for each $x \in U$ we have 
$$ 
\|\mathcal{G}_\lambda(x)-\ox\| = \|\mathcal{G}_\lambda(x)- \mathcal{G}_\lambda (\ox)\| \leq \ell \|x-\ox\|,
$$ 
which yields the first estimate in \eqref{Fnorest} with $\eta_1:=\ell$. To verify the second one, for each $x \in U$ we deduce from \eqref{h(U)} that $\mathcal{G}_\lambda(x) \in \Tilde{U}$. The Lipschitz continuity of $\nabla f$ with modulus $L >0$ on $\Tilde{U}$ leads us to the relationships 
\begin{align*}
\big\|\mathcal{F}_\lambda^{\text{\rm nor}}\big(x-\lambda \nabla f(x)\big)\big\| & = \big\|\nabla f\big(\mathcal{G}_\lambda(x)\big) -\nabla f(x) + \frac{1}{\lambda}\big(x-\mathcal{G}_\lambda(x)\big)\big\| \\
& \leq\big\|\nabla f\big(\mathcal{G}_\lambda(x)\big) - \nabla f(x)\big\| + \frac{1}{\lambda}\big\|x-\mathcal{G}_\lambda(x)\big\| \leq\big(L + \frac{1}{\lambda}\big)\|\mathcal{G}_\lambda(x) - x\|\\
& \le\big(L + \frac{1}{\lambda}\big)\big(\|\mathcal{G}_\lambda(x) - \ox\| + \|x-\ox\|\big)\leq\big(L + \frac{1}{\lambda}\big)(1+\ell)\|x -\ox\|,
\end{align*}
which justify the second estimate in \eqref{Fnorest} with $\eta_2:= (L+\lambda^{-1})(1+\ell)$. 
\end{proof}\vspace*{-0.05in}

The obtained Proposition~\ref{singlevalueproxmap} immediately provides the justification of the approximate step in the version of Algorithm~\ref{NMcod} for the structured optimization model \eqref{ncvopt}.\vspace*{-0.07in} 

\begin{Corollary}[\bf approximate step in structured nonconvex optimization] \label{apstepncv} In the setting of Proposition~{\rm\ref{singlevalueproxmap}}, for each {$\lambda \in (0,\bar{\lambda})$} there exist $\eta>0$ and a neighborhood $U$ of $\ox$  such that whenever $x \in U$, we get a pair $(\Hat{x},\Hat{v}) \in \gph \partial\varphi$ satisfying the estimate
$$
\|(\Hat{x},\Hat{v})-(\ox,0)\| \leq \eta \|x-\ox\|,
$$
where the approximation vectors $\Hat{x}$ and $\Hat{v}$ are selected by 
$$
\Hat{x}:=\mathcal{G}_\lambda(x)\;\text{ and }\;\Hat{v}:= \mathcal{F}_\lambda^{\text{\rm nor}}(x-\lambda \nabla f(x)).
$$
\end{Corollary}\vspace*{-0.05in}

Let us now discuss how to implement the approximate step in Algorithm~\ref{NMcod} to solve the structured optimization problem \eqref{ncvopt}. The simplest way of choosing $(\Hat{x}^k,\Hat{v}^k)\in\gph\partial\varphi$ in Algorithm~\ref{NMcod} is 
$$
\Hat{x}^k:= x^k\;\text{ and } \; \Hat{v}^k \in \partial\varphi(x^k).
$$
The difficulty in the above choice is that  we cannot guarantee that $\partial\varphi(x)$ is nonempty for an arbitrary $x \in \R^n$. The alternative way is choosing $\Hat{x}^k$ and $\Hat{v}^k$ as in \eqref{choicproxregular}. However, one of the difficulties here is to compute exactly the proximal mapping of $\varphi$, which is
the sum of two nonconvex functions. Thanks to Proposition~\ref{xvingraph} and Corollary~\ref{apstepncv}, we can choose $\Hat{x}^k$ and $\Hat{v}^k$ as follows:
\begin{equation}\label{xkvkcho} 
\Hat{x}^k \in  \text{\rm Prox}_{\lambda g}(x^k-\lambda \nabla f(x^k)),  \text{ and }  \Hat{v}^k:= \nabla f(\Hat{x}^k)- \nabla f(x^k)+ \frac{1}{\lambda}(x^k-\Hat{x}^k).
\end{equation} 
Note that $\Hat{x}^k$ and $\Hat{v}^k$ in \eqref{xkvkcho} are always well-defined since $\text{\rm Prox}_{\lambda g}(x)$ is nonempty for any $x\in \R^n$ due to the prox-boundedness of $g$. Moreover, \eqref{xkvkcho} requires to know the proximal mapping only for the regularizer $g$ (not for the whole function $\ph$ in \eqref{ncvopt}), which is often  easy to compute. The corresponding Newton algorithm using the choice \eqref{xkvkcho} is formulated as follows.\vspace*{-0.1in}

\begin{algorithm}[H]
\caption{\bf(coderivative-based pure Newton algorithm in structured optimization)}\label{NMcompositelocal}
\begin{algorithmic}[1]
\Require {$x^0\in\R^n$,  $\lambda \in (0,\lambda_g)$}
\For{\texttt{$k=0,1,\ldots$}}
\State\text{If $x^k \in \text{\rm Prox}_{\lambda g}(x^k -\lambda \nabla f(x^k))$, stop; otherwise go to the next step}
\State  \text{Define  $\Hat{x}^k$ and $\Hat{v}^k$ by} 
$$
\Hat{x}^k \in  \text{\rm Prox}_{\lambda g}(x^k -\lambda \nabla f(x^k)),\quad\Hat{v}^k:= \nabla f(\Hat{x}^k) -\nabla f(x^k) +\frac{1}{\lambda} (x^k  -\Hat{x}^k) 
$$
\State {\bf Newton step:} Choose $d^k\in\R^n$ satisfying
$$
-\Hat{v}^k - \nabla^2 f(\Hat{x}^k)d^k \in 
 \partial^2 g(\Hat{x}^k,\Hat{v}^k- \nabla f(\Hat{x}^k)(d^k)
$$
\State \text{Set $x^{k+1}:=\Hat{x}^k+ d^k$}
\EndFor
\end{algorithmic}
\end{algorithm}\vspace*{-0.2in}

\begin{Remark}[\bf iterative sequence generated by Algorithm \ref{NMcompositelocal}] \label{NMcodandNMcompo} \rm  Proposition~\ref{xvingraph} and Corollary~\ref{apstepncv}  allow us to show that Algorithm~\ref{NMcompositelocal} is a special case of the general scheme proposed in Algorithm~\ref{NMcod}. Indeed, Remark~\ref{crit->Fsta} tells us that Step~2 of Algorithm~\ref{NMcompositelocal} implies that $0 \in \partial \varphi(x^k)$. Moreover,  by Proposition~\ref{xvingraph} and the constructions of $\{\Hat{x}^k\}$ and $\{\Hat{v}^k\}$ in Algorithm \ref{NMcompositelocal}, we have that $(\Hat{x}^k,\Hat{v}^k)\in\gph\partial\varphi$. Then Corollary~\ref{apstepncv} ensures that 
$$
\|(\Hat{x}^k,\Hat{v}^k) -(\ox,0)\| \leq\eta \|x^k -\ox\|  
$$
whenever $x^k$ is near $\ox$. This means that the  choice of $(\Hat{x}^k,\Hat{v}^k)$ in Algorithm~\ref{NMcompositelocal} satisfies the approximate step of Algorithm~\ref{NMcod}. The Newton step of Algorithm~\ref{NMcompositelocal} can be written as 
$$
-\Hat{v}^k \in \nabla^2 f(\Hat{x}^k)d^k + \partial^2 g\big(\Hat{x}^k, \Hat{v}^k - \nabla f(\Hat{x}^k)\big)(d^k),
$$
which is equivalent to $-\Hat{v}^k \in \partial^2\varphi(\Hat{x}^k,\Hat{v}^k)(d^k)$ by the second-order sum rule from \cite[Proposition~1.121]{Mordukhovich06}.
\end{Remark}\vspace*{-0.05in}

The next theorem establishes the well-posedness and local superlinear convergence of iterates in Algorithm~\ref{NMcompositelocal} under fairly general assumptions.\vspace*{-0.05in}

\begin{Theorem}[\bf local superlinear convergence of the coderivative-based pure Newton algorithm]\label{localNMcomposite} Let $\ox$ be an {M-stationary point} of the problem \eqref{ncvopt} for which $\nabla f+\partial g$ is metrically regular around $(\ox,0)$ and $\partial g$ is semismooth$^*$   at $(\ox,-\nabla f(\ox))$.   Suppose further that   $g$ is {continuously $r$-level prox-regular at $\ox$ for $-\nabla f(\ox)$ for some $r>0$, and $\ox$ is a $\bar{\lambda}$-critical point of \eqref{ncvopt} for some $\bar{\lambda} \in (0,1/r]$}. Then for {$\lambda \in (0,\bar{\lambda})$}, there exists a neighborhood $U$ of $\ox$ such that for all $x^0\in U$ we have that Algorithm~{\rm\ref{NMcompositelocal}} either stops after finitely many iterations, or produces a sequence  $\{x^k\}$ that converges $Q$-superlinearly to $\ox$.
\end{Theorem}\vspace*{-0.15in}
\begin{proof} The limiting subdifferential sum rule from \cite[Proposition~1.107]{Mordukhovich06} tells us that
$$
\partial\varphi(x) = \nabla f(x) + \partial g(x)\;\text{ for all }\;x\in \R^n,
$$
which ensures by the imposed assumptions that $\partial\varphi$ is metrically regular around $(\ox,0)$, and semismooth$^*$ at this point due to \cite[Proposition~3.6]{Helmut}. It follows from Remark~\ref{NMcodandNMcompo} that Algorithm~\ref{NMcompositelocal} reduces to Algorithm~\ref{NMcod} with $\varphi=f+g$. Thus the assumptions of Theorem~\ref{localconvergeNMcod} are satisfied, and we are done with the proof. 
\end{proof}\vspace*{-0.05in}

  As discussed in Section \ref{sec:general}, finding directions that satisfy the generalized Newton system is crucial for practical implementations of our algorithm. Next we present an explicit formula for computing the direction in the Newton step of Algorithm~\ref{NMcompositelocal}. To proceed, consider the {\em Bouligand’s Jacobian} of $f: \R^n \to \R^m$ at $\bar{x} \in \R^n$ defined by  
$$
\partial_Bf(\bar{x}):=\big\{ H\in\R^{m \times n}\;\big|\;\exists\ x_k \to \bar{x} \text{ such that } f \text{ is  differentiable at } x_k\;\text{ and }\;\nabla f(x_k) \to H \big\}.
$$ 
The classical Rademacher theorem tells us that if $f$ is locally Lipschitz around $\bar{x}$, then $\partial_B f(\bar{x})$ is nonempty. Moreover, if $\varphi$ is of class $\mathcal{C}^{1,1}$ around $\ox$, we get  

\begin{equation}\label{BandM}
\partial_B \nabla \varphi(\ox) w \subset \partial^2 \varphi(\ox)(w) \quad \text{for all }\; w \in \R^n. 
\end{equation}  

The following result offers a sufficient condition and a constructive method to find a generalized Newton direction when the objective function in given in form \eqref{ncvopt}.\vspace*{-0.05in} 
 
\begin{Theorem}[\bf finding generalized Newton directions]\label{newsolvability} Let $\varphi:\R^n \to \oR$ be defined in \eqref{ncvopt}. For $\ox \in \dom \varphi$ and $\ov \in \partial\varphi(\ox)$, {suppose that $g:\R^n\to\overline{\R}$ is $r$-weakly convex for some $r>0$.} If the subgradient mapping $\partial\ph$ is  metrically regular around $(\ox,\ov)$, then for any $\lambda \in (0,1/r)$ and $A \in \partial_B(\text{\rm Prox}_{\lambda g})(\ox +\lambda (\ov -\nabla f(\ox)))$, we have the assertions:

{\bf(i)} The matrix $I-A+\lambda A \nabla^2  f(\ox)$ is nonsingular.

{\bf (ii)} If $d\in \R^n$ satisfies the system  of linear equations 
\begin{equation}\label{linearNM}
-\lambda A \ov = (I-A +\lambda  A \nabla^2 f(\ox))d,
\end{equation}
then $d$ solves the generalized Newton system \eqref{newton-inc} for $(x,v):=(\ox,\ov)$.   
\end{Theorem}
\begin{proof} To verify {\bf (i)}, pick $\lambda \in (0,1/r)$ and  $A \in \partial_B(\text{\rm Prox}_{\lambda g})(\ox +\lambda(\ov -\nabla f(\ox)))$. {It follows from Proposition~\ref{C11} and Remark~\ref{choiceMr}(ii) that $\text{\rm Prox}_{\lm g}$ is single-valued, Lipschitz continuous, and
\begin{equation}\label{gradenvgprx}
\nabla e_\lm g (x) = \frac{1}{\lm}(x- \text{\rm Prox}_{\lm g} (x)) \quad \text{for all }\; x \in \R^n. 
\end{equation}
Taking any $u \in \R^n$ satisfying
\begin{equation}\label{inverseeq}
(I-A+\lambda A\nabla^2 f(\ox))u = 0,
\end{equation}
our goal is to show that $u =0$.} To proceed, rewrite \eqref{inverseeq} in the form
\begin{equation}\label{inverseeq2}
-\nabla^2 f(\ox) u = \frac{1}{\lambda}(I-A)(u -\lambda \nabla^2 f(\ox)u). 
\end{equation}
Using  \eqref{gradenvgprx} and an elementary sum rule for Bouligand’s Jacobian gives us
$$
\partial_B(\text{\rm Prox}_{\lambda g})(\ox +\lambda (\ov -\nabla f(\ox))) = I - \lambda   \partial_B \nabla e_{\lambda} g(\ox +\lambda (\ov -\nabla f(\ox))),
$$
which implies in turn that 
\begin{equation}\label{inBouligand}
\frac{1}{\lambda}(I-A) \in \partial_B \nabla e_{\lambda} g(\ox +\lambda (\ov -\nabla f(\ox))).
\end{equation}
Combining \eqref{BandM}, \eqref{inverseeq2}, and \eqref{inBouligand},  we arrive at the inclusion 
$$
-\nabla^2 f(\ox) u \in \partial^2 e_\lambda g(\ox+ \lambda (\ov -\nabla f(\ox)))(u-\lambda \nabla^2 f(\ox)u).
$$
Then Lemma~\ref{2ndofMoreau} tells us that
$$
-\nabla^2 f(\ox) u \in \partial^2 g (\ox,\ov- \nabla f(\ox))(u),
$$
which is equivalent to $0 \in \partial^2\varphi(\ox,\ov)(u)$ by  the second-order sum rule from \cite[Proposition~1.121]{Mordukhovich06}. Since $\partial \varphi$ is metrically regular around $(\ox,\ov)$, it follows from the Mordukhovich criterion \eqref{cod-cr} that $u =0$ as claimed in (i). 
 
To verify (ii), let $d$ be a solution to \eqref{linearNM} and rewrite \eqref{linearNM} in the form
$$
-\ov -\nabla^2 f(\ox) d = \frac{1}{\lambda}(I-A)(d +\lambda (-\ov -\nabla^2 f(\ox)d)), 
$$
which readily implies  the inclusion
$$
-\ov - \nabla^2 f(\ox) d \in \partial^2  e_{\lambda} g(\ox +\lambda (\ov -\nabla f(\ox)))(d +\lambda (-\ov -\nabla^2 f(\ox)d)).
$$
Using Lemma~\ref{2ndofMoreau} again, we deduce that 
$$
-\ov - \nabla^2 f(\ox) d \in \partial^2 g (\ox,\ov - \nabla f(\ox))(d),
$$
and thus $d$ satisfies \eqref{newton-inc} by the second-order sum rule from \cite[Proposition~1.121]{Mordukhovich06}. \end{proof}\vspace*{-0.15in}

\begin{Remark}[\bf improvements of known results] \rm Our previous paper \cite{BorisKhanhPhat} verifies the local superlinear convergence of the coderivative-based Newton algorithm in \cite[Algorithm~7.1]{BorisKhanhPhat}) to solve \eqref{ncvopt} in the case where $f$ is a quadratic function and $g$ is continuously prox-regular around the reference point. The aforementioned algorithm is a special case of the novel Algorithm~\ref{NMcompositelocal}. Let us emphasize that now we {\em don't require} as that $f$ is {\em quadratic}, which was a serious restriction in \cite{BorisKhanhPhat} essential in the proof therein. 
\end{Remark}\vspace*{-0.25in}

\section{Generalized Line Search Proximal Gradient Method}\label{sec:newlinesearch}\vspace*{-0.05in}

In the line of the classical {\em pure} Newton method, the generalized one from 
Section~\ref{sec:localNMcompo} provides fast convergence when the iterative sequence is near the solution point. To get its {\em globalization}  that converges for any starting point in $\R^n$ in the nonconvex composite optimization problem \eqref{ncvopt}, a natural approach is to replace Step~5 in Algorithm~\ref{NMcompositelocal} by a {\em damped} step of the form 
$$
x^{k+1}:= \Hat{x}^k + \tau_k d^k
$$
with an appropriate stepsize selection $\tau_k \in (0,1]$. To accomplish this, we develop in this section a {\em new line search method} based on the \textit{forward-backward envelope} (FBE) $\ph_\lm$ of the function $\ph$ from \eqref{ncvopt}, which is defined below to guarantee the {\em descent property} at each iteration, i.e., 
\begin{equation}\label{dpofFB}
\varphi_\lambda(\Hat{x}^k +\tau_k d^k) < \varphi_\lambda(x^k), \quad k=0,1,\ldots. 
\end{equation}
To be more specific, we employ the construction of FBE introduced in \cite{pb} for convex composite functions and then largely used to design numerical methods in composite optimization; see. e.g., in \cite{kmptmp,Mor24,stp,stp2}. In general, the FBE $\varphi_\lambda$ approximates $\varphi$ from below while linearizing the smooth part of the function $\varphi$ in \eqref{ncvopt}. This construction plays a crucial role in developing our new line search method to guarantee the {\em global convergence} of the generalized coderivative-based Newton method presented below in Section \ref{sec:globalNMcompo}.\vspace*{-0.05in}

\begin{Definition}[\bf forward-backward envelopes] \rm Let	$\varphi$ {be given as in}   \eqref{ncvopt}. The {\em forward-backward envelope} $($FBE$)$ of $\varphi$ with the	parameter value $\lambda>0$ is given by \begin{equation}\label{FBE}	\varphi_\lambda(x):= \inf_{y \in \R^n} \Big\{f(x) + \langle \nabla	f(x),y -x\rangle + g(y) + \frac{1}{2\lambda}\|y-x\|^2\Big\}.
\end{equation} 
\end{Definition} 
\noindent
The following properties are clearly implied by the definition.

{\bf(i)} $\varphi_\lambda(x) \leq \varphi(x)$ for all $x \in \R^n$. 

{\bf(ii)} We have for each $x \in \R^n$, $\lambda \in (0,\lambda_g)$, and $\Hat{x} \in \mathcal{G}_\lambda (x)$ that
\begin{equation}\label{FBE=hatxexpress}
\varphi_\lambda (x) = f(x) + \langle \nabla f(x), \Hat{x}-x\rangle + g(\Hat{x}) + \frac{1}{2\lambda}\|\Hat{x}-x\|^2,
\end{equation}
where $\mathcal{G}_\lambda$ is defined in \eqref{prox-grad-map}. 
Furthermore, FBE \eqref{FBE} can be equivalently written as 
\begin{equation}\label{FBE2}
\varphi_\lambda(x)= f(x) - \frac{\lambda}{2}\|\nabla f(x)\|^2+
e_\lambda g\big(x-\lambda\nabla f(x)\big). 
\end{equation} 

It follows from Proposition \ref{C11} and \eqref{FBE2} that  {under fairly mild assumptions}, the FBE  $\ph_\lm$ is {\em ${\cal C}^1$-smooth} with the explicit calculation of its gradient and the preservation of {\em stationary} points of $\ph$.\vspace*{-0.05in}

\begin{Proposition}[\bf smoothness of FBEs and stationary points]\label{FBEdiff}  Let $\ox$ be an {M-stationary point of \eqref{ncvopt}, and let $g$ be continuously  $r$-level prox-regular at $\ox$ for $\ov:=-\nabla f(\ox)$ for some $r>0$}. {Suppose that $\ox$ is a $\bar{\lambda}$-critical point of \eqref{ncvopt} for some $\bar{\lambda} \in (0,1/r).$} Then for all $\lambda  \in (0, \bar{\lambda})$, there exists a	neighborhood $U$ of $\ox$ on which the FBE $\varphi_\lambda$ is continuously differentiable and its gradient mapping is calculated by 
\begin{equation}\label{gradFBE}
\nabla\varphi_\lambda (x) = \lambda^{-1}\big(I-\lambda  \nabla^2f(x)\big)\big(x- \text{\rm Prox}_{\lambda g}(x-\lambda \nabla f(x))\big),\end{equation} 
where $I$ is the identity matrix. Moreover, if $I -\lambda \nabla^2 f(x)$ is nonsingular for  $x \in U$, we have the equivalence	
$$
\nabla\varphi_\lambda (x) = 0 \iff 0 \in \partial \varphi(x).
$$
If in addition $g$ is convex, then $\lambda$ above can be chosen arbitrarily in $(0,\infty)$.
\end{Proposition} \vspace*{-0.1in}

\begin{Remark}[\bf nonsingularity of the shifted matrix and $\mathcal{C}^{1,1}$-property of FBEs] \label{C11FBE}\rm \,

{\bf(i)} The nonsingularity of the shifted matrix $I-\lambda \nabla^2 f(x)$ in Proposition~\ref{FBEdiff} is guaranteed  {if we assume} in addition that $\nabla f$ is Lipschitz continuous on $\R^n$ with modulus $L_f>0$ and that $\lambda\in(0, 1/L_f)$. Indeed, the Lipschitz continuity of $\nabla f$ with modulus $L_f>0$ implies that the matrix  $L_f I - \nabla^2 f(x)$ is positive-semidefinite for all $x \in \R^n$. Therefore, the matrix $I-\lambda \nabla^2 f(x)$ is positive-definite for all $x \in \R^n$ by the choice of $\lambda \in (0, 1/L_f)$. In this case, it follows from \eqref{gradFBE} that
\begin{equation}\label{criticalFBE}
\nabla \varphi_\lambda(x) =0 \iff x = \text{\rm Prox}_{\lambda g}\big(x-\lambda \nabla f(x)\big).
\end{equation}

{\bf(ii)} The  FBE $\varphi_\lambda$ is of class $\mathcal{C}^{1,1}$ around $\ox$ if the mapping $\nabla^2 f$ is locally Lipschitzian around $\ox$. Indeed, Proposition~\ref{C11} tells us that the Moreau envelope $e_\lambda g$ is of class $\mathcal{C}^{1,1}$ around $\ox-\lambda \nabla f(\ox)$. Combining this with the FBE representation in form \eqref{FBE2} verifies the claim.
\end{Remark}\vspace*{-0.05in}

Prior to the design and verification of our new line-search method, we establish some related auxiliary results, which are certainly of their own interest. The first lemma provides a precise formula for representing the generalized Hessian \eqref{2nd} of FBE \eqref{FBE} at critical points.
\vspace*{-0.05in}

\begin{Lemma}[\bf second-order subdifferential of FBEs at critical points]
\label{calculatepsi} Let $\ox$ be an {M-stationary point of \eqref{ncvopt}, and let $g$ be $r$-level prox-regular at $\ox$ for $\ov:=-\nabla f(\ox)$ for some $r>0$}. {Suppose that $\ox$ is a $\bar{\lambda}$-critical point of \eqref{ncvopt} for some $\bar{\lambda} \in (0,1/r)$}, and $\nabla^2 f$ is strictly differentiable at $\ox$. For any $\lambda \in (0,\bar{\lambda})$ such that $Q(\ox):=I-\lambda\nabla^2 f(\ox)$ is nonsingular, we have the FBE generalized Hessian representation
\begin{equation}\label{2ndsubofFBE}
\partial^2\varphi_\lambda (\ox)(w) = Q(\ox)\nabla^2f(\ox)w + Q(\ox)\partial^2 e_\lambda g\big(\oz, -\nabla f(\ox)\big)\big(Q(\ox)w\big)\;\text{ whenever }\; w \in \R^n,
\end{equation}
where $\oz:= \ox- \lambda\nabla f(\ox)$. Consequently, we get the equivalence
\begin{equation}\label{2ndsubofFBEother}
z \in \partial^2 \varphi_\lambda (\ox)(w) \iff z \in \partial^2 \varphi(\ox,0)\big(w-\lambda Q(\ox)^{-1}z\big)\;\text{ for any }\; w \in \R^n
\end{equation}
provided that $g$ is subdifferentially continuous at $\ox$ for $-\nabla f(\ox)$. 
\end{Lemma} 
\begin{proof} To verify representation \eqref{2ndsubofFBE}, fix $\ox,\oz,w,\lambda$ as above and define the functions $\psi:\R^n \to \R$, $h: \R^n\to\overline{\R}$ and the mapping $F: \R^n \to \R^n$ by
$$
\psi(x):= f(x) - \frac{\lambda}{2}\|\nabla f(x)\|^2,\; h(x):= e_{\lambda}g\big(F (x)\big)\;\text{ with}\;
F (x):= x-\lambda \nabla f(x),\quad x \in \R^n.
$$
It follows from the differentiability of $\nabla^2 f$ at $\ox$ that $\psi$ is twice differentiable there. We have by \eqref{FBE2} that $\varphi_\lambda = \psi + h$ and thus apply to this sum the second-order sum rule from \cite[Proposition~1.121]{Mordukhovich06} to get
\begin{equation}\label{cal1}
\partial^2\varphi_\lambda(\bar{x})(w) = \nabla^2 \psi(\ox)w + \partial^2 h(\bar{x})(w),\quad w\in\R^n.
\end{equation}
The Hessian matrix of $\psi$ is clearly calculated by
\begin{equation}\label{3rdderi}
\begin{array}{ll}
\nabla^2 \psi(\ox)&= \nabla^2 f(\ox) -\lambda \nabla^2 f(\ox) \nabla^2 f(\ox) - \lambda \nabla^2 \langle \nabla f(\ox), \nabla f\rangle (\ox) = Q(\ox)\nabla^2 f(\ox) -\lambda \nabla^2 \langle \nabla f(\ox), \nabla f\rangle (\ox).
\end{array} 
\end{equation} 
It follows from Proposition~\ref{C11} that $h$ is continuously differentiable at $\ox$.
By using the second-order chain rule from \cite[Theorem~1.127]{Mordukhovich06}, we have the representation
\begin{equation}\label{cal2}
\partial^2h(\bar{x})(w)  = \nabla^2\langle  \nabla e_\lambda g(\oz), F\rangle(\ox) w + Q(\ox)\partial^2e_\lambda g(\oz, \nabla e_\lambda g(\oz))(Q(\ox)w).
\end{equation}
{Since $\ox$ is a $\bar{\lambda}$-critical point of $\ph$, it is also $\lambda$-critical point due to Remark \ref{crit->Fsta}}. Moreover, it follows from Proposition~\ref{C11} that
\begin{equation}\label{barv}
\nabla e_\lambda g(\oz)= \frac{1}{\lambda}\big(\oz - \text{\rm Prox}_{\lambda g}(\oz)\big) = \frac{1}{\lambda}\big(\ox -  \lambda \nabla f(\ox) - \text{\rm Prox}_{\lambda g}\big(\ox- \lambda \nabla f(\ox)\big)\big)=-\nabla f(\ox). 
\end{equation}
Combining \eqref{cal2} and \eqref{barv} gives us the equality
\begin{equation}\label{2ndh} 
\partial^2 h(\ox)(w) = \lambda \nabla^2 \langle \nabla f(\ox), \nabla f\rangle (\ox)w + Q(\ox)\partial^2e_\lambda g\big(\oz,-\nabla f(\ox)\big)\big(Q(\ox)w\big),
\end{equation} 
which being combined with \eqref{cal1}, \eqref{3rdderi}, and \eqref{2ndh} justifies the claimed representation \eqref{2ndsubofFBE}. 

It remains to verify \eqref{2ndsubofFBEother}. To this end, observe from \eqref{2ndsubofFBE} that $z \in \partial^2\varphi_\lambda(\ox)(w)$ if and only if we have 
$$
Q(\ox)^{-1}z -\nabla^2 f(\ox) w \in \partial^2 e_\lambda g\big(\oz, -\nabla f(\ox)\big)\big(Q(\ox)w\big).  
$$
By Lemma~\ref{2ndofMoreau}, the above inclusion is equivalent  to
$$
Q(\ox)^{-1}z -\nabla^2 f(\ox) w \in \partial^2 g\big(\ox, -\nabla f(\ox)\big)\big(Q(\ox)w -\lambda Q(\ox)^{-1}z +\lambda \nabla^2 f(\ox)w\big),
$$
from which we can conclude that
\begin{equation}\label{calincFBE}
Q(\ox)^{-1}z -\nabla^2 f(\ox) w \in \partial^2 g\big(\ox, -\nabla f(\ox)\big)\big(w - \lambda Q(\ox)^{-1} z\big).
\end{equation}
Applying finally the second-order subdifferential sum rule from  \cite[Proposition~1.121]{Mordukhovich06} gives us
\begin{equation}\label{sumrFBE}
\partial^2\varphi(\ox,0)\big(w -\lambda Q(\ox)^{-1}z\big) = \nabla^2 f(\ox)\big(w -\lambda Q(\ox)^{-1}z\big) + \partial^2 g\big(\ox, -\nabla f(\ox)\big)\big(w - \lambda Q(\ox)^{-1} z\big).
\end{equation}
Due to \eqref{sumrFBE}, the inclusion in \eqref{calincFBE} is equivalent to 
$$
Q(\ox)^{-1}z - \nabla^2 f(\ox)w + \nabla^2 f(\ox)\big(w -\lambda Q(\ox)^{-1}z\big) \in \partial^2\varphi(\ox,0)\big(w -\lambda Q(\ox)^{-1}z\big),
$$
which readily reduces to $z \in \partial^2\varphi(\ox,0)(w -\lambda Q(\ox)^{-1}z)$ and thus completes the proof.
\end{proof}\vspace*{-0.05in}

The next result shows that the {\em tilt stability} of the structured cost function \eqref{ncvopt} at a stationary point yields the {\em strong convexity} in its forward-backward envelope around this point. To proceed, recall that a set-valued mapping $T: \R^n \rightrightarrows \R^n$ is {\em positive-semidefinite} if $\langle z,u \rangle \geq 0$ for any $u \in \R^n$ and $z \in T(u)$. If $\langle z, u\rangle >0$ whenever $u\in \R^n\setminus\{0\}$ and $z \in T(u)$, then $T$ is said to be {\em positive-definite}.\vspace*{-0.05in}

\begin{Lemma}[\bf tilt-stable minimizers and strong convexity of FBEs]\label{localstofFBE}
In the setting of Lemma~{\rm\ref{calculatepsi}}, suppose that $g$ is subdifferentially continuous at $\ox$ for $-\nabla f(\ox)$, and let $\lambda \in (0,\bar{\lambda})$ be such that $Q(\ox)= I-\lambda \nabla^2f(\ox)$ is positive-definite. If $\ox$ is a tilt-stable local minimizer of $\varphi$ in \eqref{ncvopt}, then FBE \eqref{FBE} is strongly convex around $\ox$. 
\end{Lemma}\vspace*{-0.15in}
\begin{proof} It follows from the assumptions that $\varphi$ is continuously prox-regular at $\ox$ for $0$. Combining this with the pointbased  second-order characterization of tilt stability from \cite[Theorem~1.3]{Poli}  gives us the positive-definiteness of $\partial^2\varphi(\ox,0)$. Note that the strict differentiability of $\nabla^2 f$ yields the local Lipschitz continuity of $\nabla^2 f$ around $\ox$. Therefore, $\varphi_\lambda$ is of class $\mathcal{C}^{1,1}$ around $\ox$ due to Remark~\ref{C11FBE}(ii). By \cite[Proposition~4.6]{ChieuLee17}, we verify the strong convexity of $\varphi_\lambda$ around $\ox$ via showing that $\partial^2\varphi_\lambda(\ox)$ is positive-definite.  To furnish this, pick $w \in \R^n\setminus\{0\}$ and $z \in \partial^2\varphi_\lambda(\ox)(w)$. Then we deduce from \eqref{2ndsubofFBEother} that $z \in \partial^2\varphi(\ox,0)(w-\lambda Q(\ox)^{-1}z)$, which ensures that $z \ne 0$ by the positive-definiteness of $\partial^2\varphi(\ox,0)$. This also implies that $\langle z , w - \lambda Q(\ox)^{-1}z\rangle \geq 0$, and hence
$$
\langle z , w  \rangle \geq  \lambda \langle Q(\ox)^{-1}z, z\rangle>0
$$
due to the positive-definiteness of $Q(\ox)^{-1}$. {This verifies the claim.}
\end{proof}\vspace*{-0.05in}

To design and justify our next algorithm to solve the structured optimization model \eqref{ncvopt}, we need the following global Lipschitzian property of the gradient of $f$ in \eqref{ncvopt}.\vspace*{-0.05in}

\begin{Assumption}\label{Lipnablaf}  $\nabla f$ is Lipschitz continuous on $\R^n$ with modulus $L_f>0$.
\end{Assumption}\vspace*{-0.05in}

Having this property allows us to design the novel line-search method to solve \eqref{ncvopt} that involves forward-backward envelopes. We label this algorithm as the \textit{generalized line-search proximal gradient method} because it extends the proximal gradient method as explained below in Remark~\ref{pgmcompare}.\vspace*{-0.1in} 

\begin{algorithm}[H]
\caption{\bf(generalized line-search proximal gradient method)}\label{FBEbasedLM}
\begin{algorithmic}[1]
\Require {$x^0\in\R^n$,   {$\lambda \in (0,\text{\rm min}\{1/L_f,\lambda_g\})$, $\sigma \in \left(0, \frac{\lambda(1-\lambda L_f)}{2(1+\lambda L_f)^2} \right)$}, $\beta \in (0,1)$}
\For{\texttt{$k=0,1,\ldots$}}
\State\text{If $x^k \in \text{\rm Prox}_{\lambda g}(x^k -\lambda \nabla f(x^k))$, stop; otherwise  go to the next step}
\State \text{Define  $\Hat{x}^k$ and $\Hat{v}^k$ by}
$$
\Hat{x}^k \in  \text{\rm Prox}_{\lambda g}(x^k -\lambda \nabla f(x^k)), \; \Hat{v}^k:= \nabla f(\Hat{x}^k) -\nabla f(x^k) +\frac{1}{\lambda} (x^k   -\Hat{x}^k)
$$ 
\State {\bf Abstract direction step:} Select some direction $d^k\in\R^n$ (see specifications below)
\State Set $\tau_k = 1$
\While{$\varphi_\lambda(\Hat{x}^k+\tau_kd^k)>\varphi_\lambda (x^k)-\sigma \left\|\Hat{v}^k\right\|^2 $}
\State\text{set $\tau_k:= \beta\tau_k$}
\EndWhile\label{dampedendwhile}
\State \text{Set $x^{k+1}:=\Hat{x}^k+ \tau_k d^k$}
\EndFor
\end{algorithmic}
\end{algorithm}\vspace*{-0.15in}

In the rest of this section, we analyze the well-posedness of Algorithm~\ref{FBEbasedLM} and the convergence of its iterates. {\em Our achievements can be summarized as follows}:

$\bullet$ Algorithm~\ref{FBEbasedLM} is always well-defined (Lemma~\ref{linesearchwell-defined}), and the \textit{proximal gradient algorithm} is a special case of this algorithm (Remark~\ref{pgmcompare}), which motivates our naming of Algorithm~\ref{FBEbasedLM}. 

$\bullet$ Every accumulation point of $\{x^k\}$ generated by 
Algorithm~\ref{FBEbasedLM} solves \eqref{ncvopt} (Proposition~\ref{limitingNM}).

$\bullet$ If the accumulation point is isolated and the directions satisfy some additional property, the convergence is achieved (Proposition~\ref{abstractconvergence1}). 

$\bullet$ Under appropriate assumptions, both global and linear local convergence are achieved (Theorem~\ref{abstractconvergencelinear}).\vspace*{0.07in}

The next lemma shows that the stepsize $\tau_k$ in Algorithm~\ref{FBEbasedLM} is well-defined. \vspace*{-0.05in}

\begin{Lemma}[\bf stepsize in Algorithm~\ref{FBEbasedLM}]\label{linesearchwell-defined}  Let Assumption~{\rm\ref{Lipnablaf}} be satisfied in \eqref{ncvopt}. For each $\lambda \in (0,\text{\rm min}\{1/L_f,\lambda_g\})$, ${\sigma \in \left(0, \frac{\lambda(1-\lambda L_f)}{2(1+\lambda L_f)^2} \right)}$, $d\in\R^n$, and $x \in \R^n$,  we define  $\Hat{x}$ and $ \Hat{v}$ by
$$
\Hat{x} \in \text{\rm Prox}_{\lambda g}(x-\lambda \nabla f(x))\;\mbox{ and }\:\Hat{v}:= \nabla f(\Hat{x}) -\nabla f(x)  + \frac{1}{\lambda}(x-\Hat{x}).
$$
If $\Hat{x}\ne x$, then $\Hat{v}\ne 0$ while ensuring the estimate
\begin{equation}\label{decredirectwithoutt}
\varphi_\lambda(\Hat{x}) < \varphi_\lambda(x) -\sigma\|\Hat{v}\|^2.
\end{equation} 
Consequently, there exists a number $\bar{\tau} >0$ for which
\begin{equation}\label{decredirect}
\varphi_\lambda(\Hat{x}+\tau d) \leq \varphi_\lambda(x) - \sigma \|\Hat{v}\|^2\; \text{ whenever }\; \tau \in (0,\bar{\tau}).
\end{equation}
\end{Lemma}
\begin{proof} By Assumption \ref{Lipnablaf}, we have the relationships
\begin{equation}\label{estiFnorFnat1>}
\|\Hat{v}\| =  \Big\| \nabla f(\Hat{x} ) -\nabla f(x ) + \frac{1}{\lambda}(x -\Hat{x} ) \Big\| \geq  \frac{1}{\lambda}\|x -\Hat{x} \|  - \|\nabla f(\Hat{x} ) - \nabla f(x )\|\geq \Big(   \frac{1}{\lambda} - L_f\Big) \|\Hat{x}  - x \|,
\end{equation}
\begin{equation}\label{estiFnorFnat1}
\|\Hat{v}\| =  \Big\| \nabla f(\Hat{x}) -\nabla f(x) + \frac{1}{\lambda}(x-\Hat{x}) \Big\| \leq \|\nabla f(\Hat{x}) - \nabla f(x)\| + \frac{1}{\lambda}\|x-\Hat{x}\| \leq \Big(\frac{1}{\lambda}+L_f\Big) \|\Hat{x} - x\|.
\end{equation}
It follows from \eqref{estiFnorFnat1>} that $\Hat{v} \ne 0$ since $x \ne \Hat{x}$. To verify \eqref{decredirectwithoutt}, we deduce from \eqref{FBE=hatxexpress}, the Lipschitz continuity of $\nabla f$, and \cite[Lemma A.11]{Solo14} that 
$$ 
\varphi_\lambda (x)= f(x) +\langle \nabla f(x),\Hat{x}-x\rangle + g(\Hat{x}) + \frac{1}{2\lambda}\|\Hat{x}-x\|^2 
 \geq f(\Hat{x}) -\frac{L_f}{2}\|\Hat{x}-x\|^2+g(\Hat{x}) + \frac{1}{2\lambda}\|\Hat{x}-x\|^2,
$$ 
which implies in turn that 
\begin{equation}\label{ineqcontrary2}
\varphi_\lambda(x) \geq \varphi(\Hat{x}) + \frac{1-\lambda L_f}{2\lambda}\|x-\Hat{x}\|^2.
\end{equation}
Combining \eqref{estiFnorFnat1} and  \eqref{ineqcontrary2} tells us that
\begin{equation}\label{ineqcontrary3}
\varphi_\lambda(x) \geq \varphi(\Hat{x}) +  \frac{\lambda(1-\lambda L_f)}{2(1+\lambda L_f)^2}\|\Hat{v}\|^2,
\end{equation}
which being combined with $\Hat{v} \ne 0$  and \eqref{ineqcontrary3} yields 
$$
\varphi_\lambda(\Hat{x}) \leq \varphi(\Hat{x})  
\le\varphi_\lambda(x) 
-\frac{\lambda(1-\lambda L_f)}{2(1+\lambda L_f)^2}  \|\Hat{v} \|^2  <\varphi_\lambda(x) - \sigma\|\Hat{v}\|^2 
$$
and thus justifies \eqref{decredirectwithoutt}. It follows from 
Proposition~\ref{pbimplycontinuity} that the prox-boundedness of $g$ with threshold $\lambda_g$ implies the continuity of the Moreau envelope $e_\lambda g$ on $\R^n$, and hence $\varphi_\lambda$ is also continuous on $\R^n$ due to \eqref{FBE2}. Using the latter together with \eqref{decredirectwithoutt} verifies \eqref{decredirect} and completes the proof.
\end{proof}
\vspace*{-0.15in}

\begin{Remark}[\bf relationship with the proximal gradient algorithm] \label{pgmcompare} \rm Note that when $d^k:=0$ for all $k\in\N$ in Algorithm~\ref{FBEbasedLM}, this algorithm reduces to the \textit{proximal gradient method} in which
\begin{equation}\label{PPMtrans}
 x^{k+1}:= \Hat{x}^k \in \text{\rm Prox}_{\lambda g}\big(x^k-\lambda \nabla f(x^k)\big),\quad k=0,1,\ldots
\end{equation}
Indeed, it follows from $d^k =0$ and Lemma~\ref{linesearchwell-defined} that
$$
\varphi_\lambda(\Hat{x}^k+1\cdot d^k)= \varphi_\lambda(\Hat{x}^k) \leq \varphi_\lambda (x^k)-\sigma\|\Hat{v}^k\|^2,  
$$
which implies that $\tau_k =1$  for all $k \in \N$, and thus we arrive at \eqref{PPMtrans}.
\end{Remark}\vspace*{-0.05in}

Now we show that Algorithm~\ref{FBEbasedLM} is well-defined and its accumulation points are critical for \eqref{ncvopt}.\vspace*{-0.03in}

\begin{Proposition}[\bf well-posedness of Algorithm~\ref{FBEbasedLM}] \label{limitingNM} Let Assumption~{\rm\ref{Lipnablaf}} hold in \eqref{ncvopt}. Then the sequence $\{x^k\}$  generated by Algorithm~{\rm\ref{FBEbasedLM}} is well-defined, and we have the estimates
\begin{equation}\label{xkhatxkto0}
\Big(\frac{1}{\lambda} - L_f\Big) \|\Hat{x}^k - x^k\| \leq \|\Hat{v}^k\| \le\Big(   \frac{1}{\lambda} + L_f\Big) \|\Hat{x}^k - x^k\|\; \text{ for all }\; k \in \N. 
\end{equation}
If $\inf \varphi>-\infty$,  then  the sequence $\{\varphi_\lambda(x^k)\}$ is convergent and 
\begin{equation}\label{seriesvk}
\sum_{k=0}^\infty \|\Hat{v}^k\|^2 <\infty, \quad \sum_{k=0}^\infty \|\Hat{x}^k-x^k\|^2 <\infty.
\end{equation} 
Consequently,  both sequences $\{\Hat{v}^k\}$ and $\{\Hat{x}^k -x^k\}$ converge to $0$ as $k \to \infty$. Finally, any accumulation point of $\{x^k\}$ is a {$\lambda$-critical point of \eqref{ncvopt} for all  $\lambda \in (0,\text{\rm min}\{1/L_f,\lambda_g\})$.} 
\end{Proposition} \vspace*{-0.15in}
\begin{proof} It follows from Lemma~\ref{linesearchwell-defined} that Algorithm~\ref{FBEbasedLM} is well-defined and exhibits the descent property $\varphi_\lambda(x^{k+1}) < \varphi_\lambda(x^k)$. Arguing similarly to the proof of inequalities \eqref{estiFnorFnat1>} and \eqref{estiFnorFnat1} brings us to \eqref{xkhatxkto0}. By \cite[Theorem~4.4]{stp2}, we have $\inf \varphi_\lambda = \inf \varphi>-\infty$. Remembering that the sequence $\{\varphi_\lambda(x^k) \}$ is monotonically decreasing yields the convergence $\varphi_\lambda(x^k)\to \inf\varphi_\lambda$ as $k \to \infty$. Moreover, we have the estimates
$$
\begin{array}{ll}
\disp\sigma\sum_{i=0}^k \|\Hat{v}^i\|^2 \leq \sum_{i=0}^k \left[\varphi_\lambda(x^i)-\varphi_\lambda(x^{i+1}) \right]=\varphi_\lambda(x^0) -\varphi_\lambda(x^{k+1}) \leq  \varphi_\lambda (x^0) -\inf \varphi_\lambda,\\
\disp\Big(\frac{1}{\lambda} - L_f\Big)^2 \sum_{i=0}^k \|\Hat{x}^i-x^i\|^2 \leq \sum_{i=0}^k \|\Hat{v}^i\|^2 \leq \frac{1}{\sigma}(\varphi_\lambda (x^0) -\inf \varphi_\lambda),
\end{array}
$$
which justify \eqref{seriesvk} and ensures that both sequences $\{\Hat{v}^k\}$ and $\{\Hat{x}^k -x^k\}$ converge to $0$ as $k \to \infty$.

Suppose further that $\ox$ is an accumulation point of $\{x^k\}$ and hence find an infinite set $J\subset \N$ such that $x^k \overset{J}{\rightarrow} \ox$. By Proposition~\ref{pbimplycontinuity}, we deduce from the inclusion $\Hat{x}^k \in \text{\rm Prox}_{\lambda g}(x^k-\lambda \nabla f(x^k))$ that the sequence $\{\Hat{x}^k\}_{k \in J}$ is bounded and that all its accumulation points lie in $\text{\rm Prox}_{\lambda g} (\ox-\lambda \nabla f(\ox))$. Without loss of generality, assume that $\Hat{x}^k \overset{J}{\rightarrow} \Hat{x} \in \text{\rm Prox}_{\lambda g}(\ox-\lambda \nabla f(\ox))$.  By $\Hat{v}^k \to 0$, it yields
$$
0 = \lim_{k\to \infty, k \in J} \|\Hat{v}^k\|=   \| \nabla f(\Hat{x}) +\lambda^{-1}(\ox -\lambda \nabla f(\ox) -\Hat{x})\|, 
$$
which implies in turn that $0 \in \mathcal{F}_\lambda^{\text{\rm nor}} (\ox-\lambda \nabla f(\ox))$. Using Proposition~\ref{characsol}, we conclude that $\ox$ is a $\lambda$-critical point of \eqref{ncvopt} and thus complete the proof of the proposition.
\end{proof}\vspace*{-0.05in}

The next proposition verifies the global convergence of the iterates in Algorithm~\ref{FBEbasedLM} to an {\em isolated} accumulation point under a particular direction choice.\vspace*{-0.05in}

\begin{Proposition}[\bf global convergence to isolated accumulation points]
\label{abstractconvergence1} In the setting of Proposition~{\rm\ref{limitingNM}}, assume that $\ox$ is an isolated accumulation point of $\{x^k\}$ in Algorithm~{\rm\ref{FBEbasedLM}} with
\begin{equation}\label{dkvkrela}
\|d^k\| \leq \zeta\|\Hat{v}^k\|\;\text{ for all large }\;k \in J,
\end{equation}
where $\zeta>0$ is fixed. Then the sequence $\{x^k\}$ converges to $\ox$ as $k\to \infty$. 
\end{Proposition}
\begin{proof}
It follows from Proposition~\ref{limitingNM} that $\ox$ is a $\lambda$-critical point of \eqref{ncvopt} with $\Hat{v}^k\to 0$ and  $\Hat{x}^k -x^k \to 0$  as $k \to \infty$.  Let a subsequence $\{x^k\}_{k\in J}$ converge to $\ox$. For all large $k$, we get from \eqref{dkvkrela} that
$$ 
\|x^{k+1}-x^k\| \leq \|x^{k+1}-\Hat{x}^k\|+ \|\Hat{x}^k-x^k\| = \tau_k \|d^k\| + \|\Hat{x}^k-x^k\| \leq \zeta\|\Hat{v}^k\|+ \|\Hat{x}^k -x^k\|,
$$ 
which implies that  $\|x^{k+1}-x^k\|\overset{J}{\rightarrow} 0$. Then we deduce from Ostrowski’s condition \cite[Proposition~8.3.10]{JPang}that the entire sequence $\{x^k\}$ converges to $\bar{x}$ as $k\to\infty$.
\end{proof}\vspace*{-0.05in}

To establish {\em convergence rates} for Algorithm~\ref{FBEbasedLM}, we need some additional assumptions.\vspace*{-0.05in}

\begin{Assumption}\label{assumonlimitpoint}\rm Let $\ox$ be an {M-stationary point of \eqref{ncvopt}}. Suppose that:
 
{\bf(i)} $\ox$ is a tilt-stable local minimizer of $\varphi$. 

{\bf(ii)} $\nabla^2 f$ is strictly differentiable at $\ox$.

{\bf(iii)} $g$ is {continuously $r$-level prox-regular at $\ox$ for $-\nabla f(\ox)$ for some $r>0$.}
\end{Assumption}\vspace*{-0.03in}

The following lemma taken from \cite[Lemma~4]{kmptmp} provides sufficient conditions for the {\em $R$-linear} and {\em $Q$-linear convergence rates} of general iterative sequences.\vspace*{-0.05in}

\begin{Lemma}[\bf linear convergence rates for abstract sequences]\label{QRlinear}  Let $\{\alpha_k\},\; 
\{\beta_k \}$, and $\{\gamma_k \}$ be sequences of positive numbers. Assume that there are $c_i>0$, $i=1,2,3$, and $k_0\in \N$ such that for all $k \ge k_0$ we have

{\bf (i)} $\alpha_k - \alpha_{k+1} \ge c_1 \beta_k^2$.

{\bf(ii)} $\beta_k \geq c_2 \gamma_k$.

{\bf(iii)} $c_3 \gamma_k^2 \ge \alpha_k$.\\
Then $\{\alpha_k \}$ Q-linearly converges to zero, while $\{\beta_k \}$ and $\{\gamma_k \}$ converge to zero $R$-linearly as $k\to\infty$.
\end{Lemma}\vspace*{-0.05in}

Now we are ready to establish linear convergence rates for Algorithm~\ref{FBEbasedLM}.\vspace*{-0.05in}

\begin{Theorem}[\bf linear convergence rates for the generalized proximal gradient method]\label{abstractconvergencelinear} In the setting of 
Proposition~{\rm\ref{abstractconvergence1}},  suppose that  the accumulation point $\ox$ satisfies all the conditions in Assumption~{\rm\ref{assumonlimitpoint}}. 
In Algorithm~{\rm\ref{FBEbasedLM}}, choose $\lambda \in (0,1/r)$ such that $Q(\ox)= I-\lambda \nabla^2f(\ox)$ is positive-definite. Then we have the assertions:

{\bf(i)} The convergence rate of $\{\varphi_\lambda (x^k)\}$ is at least Q-linear.

{\bf(ii)} The convergence rates of $\{x^k\} $ and $\{\|\nabla\varphi_\lambda (x^k)\|\}$ are at least R-linear.
\end{Theorem}
\vspace*{-0.15in}
\begin{proof} It follows from Proposition~\ref{limitingNM} that $\Hat{v}^k\to 0$ with the estimates in \eqref{xkhatxkto0}. Invoking 
Proposition~\ref{abstractconvergence1} ensures that $x^k\to\ox$ as $k\to\infty$, and $\ox$ is a {$\lambda$-critical point of \eqref{ncvopt} for all $\lambda \in (0,\text{\rm min}\{1/L_f,\lambda_g\})$. Fix $\lambda$ such that $Q(\ox)= I-\lambda \nabla^2f(\ox)$ is positive-definite. By the $r$-level prox-regularity of $g$ at $\ox$ for $-\nabla f(\ox)$,} Proposition~\ref{singlevalueproxmap} tells us that the mapping $x \mapsto \text{\rm Prox}_{\lambda g}(x-\lambda \nabla f(x))$ is single-valued around $\ox$. Then we deduce from Remark~\ref{C11FBE}(i), the {$\lambda$-criticality of $\ox$}, and \eqref{criticalFBE} that $\nabla \varphi_\lambda(\ox)=0$. It follows from Remark~\ref{C11FBE}(ii) that $\varphi_\lambda$ is of class $\mathcal{C}^{1,1}$ around $\ox$, and thus we get that $\varphi_\lambda(x^k) \to \varphi_\lambda(\ox)$ and $\|\nabla \varphi_\lambda(x^k)\| \to 0$ as $k \to \infty$. 
Moreover, Lemma~\ref{localstofFBE} implies that $\varphi_\lambda$ is strongly convex around $\ox$, and hence there exists $\kappa>0$ such that
$$
\varphi_\lambda (x) \geq \varphi_\lambda (\ox) + \langle \nabla \varphi_\lambda(\ox), x- \ox\rangle +  \frac{\kappa}{2}\|x-\ox\|^2 =  \varphi_\lambda (\ox)+\frac{\kappa}{2}\|x-\ox\|^2,
$$
\begin{equation}\label{strongmonoFBSsec5}
\langle \nabla \varphi_\lambda(x), x-\ox\rangle = \langle \nabla \varphi_\lambda(x)-\nabla\varphi_\lambda(\ox),x-\ox\rangle \geq \kappa \|x-\ox\|^2 
\end{equation}
for all $x$ near $\ox$.  By using the descent lemma ~\cite[Lemma A.11]{Solo14}, we find $\ell>0$ with 
\begin{equation}\label{lipnablaFBEsec5}
\varphi_\lambda(x) \leq \varphi_\lambda (\ox)+ \langle \nabla \varphi_\lambda (\ox), x- \ox\rangle +  \frac{\ell}{2}\|x-\ox\|^2= \varphi_\lambda (\ox) +  \frac{\ell}{2}\|x-\ox\|^2
\end{equation}
for all such $x$. Combining the latter with \eqref{xkhatxkto0} tells us that 
$$ 
\|\Hat{v}^k\| \ge\Big(\frac{1}{\lambda} - L_f\Big) \|\Hat{x}^k - x^k\|  =\Big(   \frac{1}{\lambda} - L_f\Big) \|x^k- \text{\rm Prox}_{\lambda g}(x^k-\lambda \nabla f(x^k))\|,
$$ 
which yields by Algorithm \ref{FBEbasedLM} the estimates
\begin{equation}\label{sequenceFBSdecrease}
\varphi_\lambda(x^k)- \varphi_\lambda(x^{k+1}) \geq \sigma\|\Hat{v}^k\|^2\geq \sigma \Big(\frac{1}{\lambda}-L_f\Big)^2 \|x^k -\text{\rm Prox}_{\lambda g}(x^k-\lambda \nabla f(x^k))\|^2
\end{equation}
for large $k$. The Lipschitz continuity of $\nabla f$ on $\R^n$ ensures that the gradient of $x\mapsto \frac{\|x\|^2}{2}-\lambda f(x)$ is also global Lipschitz continuous on $\R^n$. Then there exists $M>0$ such that $\|I-\lambda \nabla^2 f(x)\| \leq M$. 
Due to \eqref{gradFBE}, we have 
\begin{equation}\label{gradFBEbounded}
\|\nabla \varphi_\lambda(x) \|=\left\|\lambda^{-1}\big(I-\lambda \nabla^2f(x)\big)\big(x-\text{\rm Prox}_{\lambda g}(x-\lambda \nabla f(x))\big) \right\| \leq \frac{M}{\lambda}\big\|x-\text{\rm Prox}_{\lambda g}(x-\lambda \nabla f(x))\big\|
\end{equation}
for all $x$ near $\ox$. Combining \eqref{sequenceFBSdecrease} and \eqref{gradFBEbounded} gives us the relationships
\begin{align}\label{cond1FBS}
\begin{array}{ll}
[\varphi_\lambda(x^k) -\varphi_\lambda(\ox)] -[\varphi_\lambda(x^{k+1})-\varphi_\lambda(\ox)] & = \varphi_\lambda(x^k) -\varphi_\lambda(x^{k+1}) \geq\disp\frac{\sigma\lambda^2}{M^2}\Big(\frac{1}{\lambda}-L_f \Big)^2 \|\nabla \varphi_\lambda(x^k)\|^2
\end{array}
\end{align}
for large $k$. Then \eqref{strongmonoFBSsec5} and the Cauchy-Schwarz inequality ensure that 
\begin{equation}\label{cond2FBS}
\|\nabla\varphi_\lambda (x^k)\| \geq \kappa \|x^k -\ox\|
\end{equation}
for all such $k$. The estimate in \eqref{lipnablaFBEsec5} tells us that 
\begin{equation}\label{cond3FBS}
\frac{\ell}{2}\|x^k-\ox\|^2 \geq \varphi_\lambda(x^k)-\varphi_\lambda(\ox)\; \text{ whenever $k$  is sufficiently large}.
\end{equation}
Using \eqref{cond1FBS}, \eqref{cond2FBS}, and \eqref{cond3FBS} and then applying Lemma~\ref{QRlinear} to the sequences 
$$
\alpha_k:= \varphi_\lambda (x^k)-\varphi_\lambda(\ox),\; \beta_k:= \|\nabla \varphi_\lambda(x^k)\|,\; \gamma_k:= \|x^k-\ox\|, \quad k=0,1,\ldots,
$$
verify the claimed assertions and thus completes the proof.
\end{proof}\vspace*{-0.17in}
 
\begin{Remark}[\bf comparison with other FBE-based line search methods] $\,$\rm 

{\bf (i)} The primary difference between the line search method in Algorithm \ref{FBEbasedLM} and the ZeroFPR method presented in \cite{stp2} is that we require the {\em monotonicity} of $\varphi_\lambda $ after each iteration, i.e., the condition in \eqref{dpofFB} must always be satisfied. Additionally, our method relies on the computation of $\Hat{v}^k$, the value of the normal map, which is crucial for justifying our globalized coderivative-based Newton method (GCNM) as discussed in Section \ref{sec:globalNMcompo}. This plays an essential role in the convergence analysis of GCNM, whereas ZeroFPR does not require this information. Regarding the convergence results for the iterative sequence of our line search method in Algorithm~\ref{FBEbasedLM}, the general properties concerning well-definedness and accumulation point behavior are similar to those of ZeroFPR. However, when the accumulation point is isolated, which is crucial for ensuring the superlinear convergence of our generalized Newton method discussed in Section \ref{sec:globalNMcompo}, Theorem \ref{abstractconvergencelinear} guarantees the linear convergence of the {sequence $\{x^k\}$} under Assumption \ref{assumonlimitpoint} and also ensures the local superlinear convergence if $d^k$ is chosen by solving the generalized Newton inclusions as suggested in Algorithm~\ref{NMcompositeglobal}. In contrast, to guarantee the superlinear convergence of ZeroFPR, the authors require strict twice epi-differentiability and the second-order epi-derivative of the regularizer $g$ to be generalized quadratic, which is not often satisfied and is difficult to verify; see \cite[Theorems~5.10 and 5.11]{stp2} for more details. This requirement is not needed in our Assumption~\ref{assumonlimitpoint}, as we do not rely on the $\mathcal{C}^2$-smoothness of FBE around the limiting point, while ZeroFPR requires this condition to apply the classical quasi-Newton methods. Numerical experiments comparing our generalized Newton-type methods with ZeroFPR methods  will be presented in Section~\ref{sec:numerical}.

{\bf (ii)} {The papers} \cite{kmptjogo, kmptmp} also employ FBE to establish line search methods for achieving the global convergence. However, the ideas presented in these works are fundamentally different from those in our current paper. In \cite{kmptjogo, kmptmp}, the authors transform nonsmooth optimization problems into $\mathcal{C}^{1,1}$-smooth optimization problems and then apply coderivative-based Newton methods to the transformed problems, with the classical Armijo's line search designed for continuously differentiable functions. In contrast, our
Algorithm~\ref{FBEbasedLM} does not need to reformulate the original optimization problem. Moreover, the developed convergence analysis does not require any convexity assumptions on the functions or the quadratic form of the loss functions as needed in \cite{kmptjogo, kmptmp}. 
\end{Remark}\vspace*{-0.25in}

\section{Globalization of Coderivative-Based Newton Method  in  Structured Nonconvex  Optimization}\label{sec:globalNMcompo}\vspace*{-0.05in}

In this section, we design a new generalized coderivative-based Newton method to solve \eqref{ncvopt} and justify its {\em fast global convergence}. With the Lipschitz constant $L_f$ from Assumption~\ref{Lipnablaf}, our {\em globalized coderivative-based Newton method} (GCNM) of solving \eqref{ncvopt} is formulated as follows.\vspace*{-0.1in}

\begin{algorithm}[H]
\caption{{\bf(coderivative-based Newtonian algorithm in nonconvex optimization)}}\label{NMcompositeglobal}
\begin{algorithmic}[1]
\Require {$x^0\in\R^n$,  {$\lambda \in (0,\text{\rm min}\{1/L_f,\lambda_g\})$, ${\sigma \in \left(0, \frac{\lambda(1-\lambda L_f)}{2(1+\lambda L_f)^2} \right)}$}, $\beta \in (0,1)$}
\For{\texttt{$k=0,1,\ldots$}}
\State\text{If $x^k \in \text{\rm Prox}_{\lambda g}(x^k-\lambda \nabla f(x^k))$, stop; otherwise  go to the next step}
\State  \text{Define  $\Hat{x}^k$ and $\Hat{v}^k$ by}
$$
\Hat{x}^k \in  \text{\rm Prox}_{\lambda g}(x^k -\lambda \nabla f(x^k))\;\mbox{ and }\;\Hat{v}^k:= \nabla f(\Hat{x}^k) -\nabla f(x^k) +\frac{1}{\lambda} (x^k   -\Hat{x}^k)
$$
\State {\bf Newton step:} Choose $d^k\in\R^n$ satisfying 
$$
-\Hat{v}^k - \nabla^2 f(\Hat{x}^k)d^k \in 
 \partial^2 g(\Hat{x}^k,\Hat{v}^k- \nabla f(\Hat{x}^k))(d^k)
$$
If such $d^k$ doesn't exist, select an arbitrary $d^k \in \R^n$
\State Set $\tau_k = 1$
\While{$\varphi_\lambda(\Hat{x}^k+\tau_kd^k)>\varphi_\lambda (x^k)-\sigma \left\|\Hat{v}^k\right\|^2 $}
\State\text{set $\tau_k:= \beta\tau_k$}
\EndWhile\label{dampedendwhile2}
\State \text{Set $x^{k+1}:=\Hat{x}^k+ \tau_k d^k$}
\EndFor
\end{algorithmic}
\end{algorithm}\vspace*{-0.15in}

Note that Algorithm~\ref{NMcompositeglobal} is a special case of 
Algorithm~\ref{FBEbasedLM} in which the direction $d^k$ is now chosen based on \textit{second-order information}. Under some appropriate assumptions, the unit stepsize is eventually accepted, and thus Algorithm~\ref{NMcompositeglobal} eventually becomes the generalized pure Newton algorithm from Algorithm~\ref{NMcompositelocal}, which ensures its Q-superlinear convergence. The following theorem explains the aforementioned observations in more detail. \vspace*{-0.05in} 

\begin{Theorem}[\bf Q-superlinear convergence of globalized coderivative-based damped Newton algorithm] \label{converNMglobal}  Given $\varphi:\R^n\to\overline{\R}$ in \eqref{ncvopt} under Assumption~{\rm\ref{Lipnablaf}}, suppose that $\inf \varphi>-\infty$.  If $\ox$ is an accumulation point of the sequence $\{x^k\}$ generated by Algorithm~{\rm\ref{NMcompositeglobal}}, then $\ox$ is a {$\lambda$-critical point of \eqref{ncvopt} for all  $\lambda \in (0,\text{\rm min}\{1/L_f,\lambda_g\})$}. Assuming further that the subgradient mapping $\partial g$ is semismooth$^*$ at $(\bar{x},-\nabla f(\ox))$ and the $\ox$ satisfies all the conditions in Assumption~{\rm \ref{assumonlimitpoint}}, choose in Algorithm~{\rm\ref{NMcompositeglobal}} the parameter $\lambda \in (0,1/r)$  such that $Q(\ox)= I-\lambda \nabla^2f(\ox)$ is 
positive-definite. Then Algorithm~{\rm\ref{NMcompositeglobal}} generates the sequences $\{x^k\}$, $\|\nabla\varphi_\lambda(\ox)\|$, and $\{\varphi_\lambda (x^k)\}$ converging respectively to $\ox$, $0$, and $\ph_\lambda(\ox)$ with Q-superlinear convergence rates. 
\end{Theorem}\vspace*{-0.15in}
\begin{proof} 
It follows from Proposition~\ref{limitingNM} that Algorithm~\ref{NMcompositeglobal}  produces sequences $\{x^k\}$, $\{\Hat{x}^k\}$, and $\{\Hat{v}^k\}$ such that both $\{\Hat{v}^k\}$ and $\{\Hat{x}^k-x^k\}$ converge to $0$ as $k \to \infty$, and that $\ox$ is a {$\lambda$-critical point of \eqref{ncvopt} for all  $\lambda \in (0,\text{\rm min}\{1/L_f,\lambda_g\})$}. Let us verify the Q-superlinear convergence of $\{x^k\}$ and $\{\varphi_\lambda(x^k)\}$ under all the conditions in 
Assumption~\ref{assumonlimitpoint} and the semismoothness$^*$ of $\partial g$ at $(\ox,-\nabla f(\ox))$. To proceed, we split the proof into the three claims.\\[0.5ex]
{\bf Claim~1:} {\it Both sequences $\{x^k\}$ and $\{\Hat{x}^k\}$ converge  to $\ox$, and  the sequence $\{d^k\}$ converges to $0$ as $k\to\infty$.} To furnish this, we first show that $\ox$ is an isolated accumulation point of $\{x^k\}$. Indeed, since $\ox$ is a tilt-stable local minimizer of $\varphi$, it follows from  {Proposition \ref{equitiltstr}} that $\partial\varphi$ is strongly metrically regular around $(\ox,0)$. Thus there exists a neighborhood $U$ of $\ox$ such that $\partial\varphi^{-1}(0)\cap U=\{\ox\}$. Suppose that there is an accumulation point $\Hat{x}$ of the sequence $\{x^k\}$ for which $\Hat{x}\in U$ and $\Hat{x}\ne \ox$. Proposition~\ref{limitingNM} ensures that $\Hat{x}$ is a {$\lambda$-critical point of \eqref{ncvopt} for all  $\lambda \in (0,\text{\rm min}\{1/L_f,\lambda_g\})$}, and hence $0 \in \partial \varphi(\Hat{x})$. Then we get that $\Hat{x}\in\partial\ph^{-1}(0)\cap U$ and thus $\Hat x=\ox$, a contradiction meaning that $\ox$ is an isolated accumulation point of $\{x^k\}$. As follows from  Proposition~\ref{solvabilityMOR}, the strong metric regularity of $\partial\varphi$ around $(\ox,0)$ yields the existence of a neighborhood $W$ of $(\ox,0)$ such that for each $(x,v)\in\gph \partial\varphi \cap W$ we find $d\in \R^n$ with
\begin{equation}\label{solv2}
-v\in \partial^2\varphi(x,v)(d) = \nabla^2 f(x)d + \partial^2 g(x, v- \nabla f(x))(d),
\end{equation} 
where the equality follows from the second-order subdifferential sum rule in \cite[Proposition~1.121]{Mordukhovich06}. This allows us to employ Theorem~\ref{abstractconvergencelinear} for clarifying the convergence of $\{x^k\}$ to $\ox$. Picking now an arbitrary subsequence $\{x^k\}_{k\in J}$ of $\{x^k\}$ converging to $\ox$, we intend to show that 
\begin{equation}\label{csvkdk}
\|d^k\| \leq \kappa\|\Hat{v}^k\|\;\text{ for all large }\;k \in J, 
\end{equation} 
where $\kappa > 0$ is a modulus of tilt stability of $\ox$.  Since $f$ is $\mathcal{C}^2$-smooth  around $\ox$ and $g$ is continuously prox-regular at $\ox$ for $-\nabla f(\ox)$, it follows that $\varphi$ is continuously prox-regular at $\ox$ for $0$. 
 {\cite[Corollary 5.7]{kmp22convex} tells us that the  tilt stability of $\ox$ with modulus $\kappa$ guarantees the existence of a positive number $\delta$ such that}  
\begin{equation}\label{varstrong2nd} 
\langle z, w\rangle \geq \kappa^{-1} \|w\|^2\;\text{ for all }\; z \in \partial^2\varphi(x,v)(w), \; (x,v)\in \gph\partial\varphi\cap \mathbb{B}_\delta(\ox,0),\; w \in \R^n.
\end{equation} 
Since $\{x^k -\Hat{x}^k\}$ converges to $0$, we have that  $\Hat{x}^k \overset{J}{\rightarrow} \ox$ as $k \to \infty$. Combining the latter with \eqref{solv2} gives us  directions $d^k\in \R^n$ satisfying the second-order inclusions
$$
-\Hat{v}^k \in \partial^2\varphi(\Hat{x}^k,\Hat{v}^k)(d^k) =  
\nabla^2 f(\Hat{x}^k)d^k + \partial^2 g(\Hat{x}^k,\Hat{v}^k- \nabla f(\Hat{x}^k)(d^k)\;\mbox{ for large }\;k\in J.
$$ 
Moreover, the inclusions $-\Hat{v}^k \in \partial^2\varphi(\Hat{x}^k,\Hat{v}^k)(d^k)$ for all $k \in J$ with the convergence $\Hat{x}^k \overset{J}{\rightarrow}  \ox$ and $\Hat{v}^k \overset{J}{\rightarrow}0$ lead us by \eqref{varstrong2nd} to $\langle -\Hat{v}^k, d^k\rangle \geq \kappa^{-1} \|d^k\|^2$ for large $k \in J$. Then the Cauchy-Schwarz inequality yields \eqref{csvkdk}, and thus $x^k\to\ox$ as $k \to\infty$ due to Theorem~\ref{abstractconvergencelinear}.  Combining this with \eqref{csvkdk} and the convergence $\Hat{v}^k\to 0$, $\Hat{x}^k-x^k\to 0$ as $k \to \infty$ tells us that $d^k \to 0$ and $\Hat{x}^k\to \ox$ as $k\to\infty$, which verifies all the conclusions of this claim.\\[0.5ex] 
{\bf Claim~2:} {\it Whenever $k\in\N$ is large enough, we get the unit stepsize $\tau_k=1$.} Indeed, the {$r$-level prox-regularity} of $g$ at $\ox$ for $-\nabla f(\ox)$ and Proposition~\ref{singlevalueproxmap} tell us that {for a fixed parameter $\lambda\in (0,1/r)$ such that $Q(\ox)= I-\lambda \nabla^2f(\ox)$ is 
positive-definite, the mapping $x \mapsto \text{\rm Prox}_{\lambda g}(x-\lambda \nabla f(x))$ is single-valued around $\ox$.} Due to Remark~\ref{C11FBE}, the FBE $\varphi_\lambda$ is of class $\mathcal{C}^{1,1}$ around $\ox$, and  the $\lambda$-criticality of $\ox$ yields $\nabla \varphi_\lambda(\ox)=0$.  By Lemma~\ref{localstofFBE}, the tilt stability of $\ox$ implies that $\varphi_\lambda$ is strongly convex around $\ox$. Thus we find  $\mu>0$ with
\begin{equation}\label{strongineqVC}
\varphi_\lambda (x) \geq \varphi_\lambda (u) + \langle \nabla \varphi_\lambda(u), x- u\rangle +  \frac{\mu}{2}\|x-u\|^2,
\end{equation}
\begin{equation}\label{strongineqVC2}
\langle \nabla\varphi_\lambda(x) - \nabla\varphi_\lambda(u), x-u\rangle \geq \mu\|x-u\|^2 
\end{equation}
for all $x, u$ near $\ox$. The Lipschitz continuity of $\nabla \varphi_\lambda$ around $\ox$ with modulus $\ell$ implies that 
\begin{equation}\label{lipnablaFBE}
\varphi_\lambda(x)\le\varphi_\lambda (u)+ \langle \nabla \varphi_\lambda (u), x- u\rangle +\frac{\ell}{2}\|x-u\|^2,
\end{equation}
\begin{equation}\label{lipnablaFBE2}
\|\nabla\varphi_\lambda(x) -\nabla\varphi_\lambda(u)\|\leq \ell \|x-u\|
\end{equation}
for all $x, u$ near $\ox$. By Claim 1, we have $\Hat{x}^k +d^k \to \ox$ and $\Hat{x}^k \to \ox$ as $k \to \infty$. Combining this with \eqref{csvkdk}, \eqref{strongineqVC}, and \eqref{lipnablaFBE2}  brings us to the conditions
\begin{eqnarray}\label{estimatesuplinear}
\begin{array}{ll}
\disp\varphi_\lambda (\Hat{x}^k+d^k)-\varphi_\lambda (\Hat{x}^k) \disp\le\langle\nabla\varphi_\lambda (\Hat{x}^k+d^k),d^k\rangle-\frac{\mu}{2}\|d^k\|^2\\
\disp\le\|\nabla\varphi_\lambda(\Hat{x}^k+d^k)\|\cdot\|d^k\|
\disp\le\|\nabla\varphi_\lambda(\Hat{x}^k+d^k)- \nabla \varphi_\lambda (\ox) \|\cdot\|d^k\|\\
\disp\le\frac{\ell}{\kappa} \|\Hat{x}^{k}+d^k-\bar{x}\|\cdot\|\Hat{v}^k\|\;\mbox{ for large }\;k\in\N.
\end{array}
\end{eqnarray} 
It follows from \cite[Proposition~3.6] {Helmut} that the semismoothness$^*$ of $\partial g$ at $(\ox,-\nabla f(\ox))$ yields the semismoothness$^*$ of $\partial\varphi$ at $(\ox,0)$.  By  {Proposition \ref{equitiltstr}}, the tilt stability of
$\ox$ implies that the subgradient mapping $\partial\varphi$ is strongly metrically regular around $(\ox,0)$. Combining the latter statements and using Theorem~\ref{estiNMitercod}, we have 
\begin{equation} \label{xkvkdksuperlinear}
\|\Hat{x}^k + d^k -\ox\| = o(\|(\Hat{x}^k,\Hat{v}^k)-(\ox,0)\|)\;\text{ as }\;k \to \infty 
\end{equation} 
by $\Hat{x}^k \to \ox$, $\Hat{v}^k \to 0$, and $-\Hat{v}^k \in \partial^2\varphi(\Hat{x}^k,\Hat{v}^k)(d^k)$. 
 {Combining this with \eqref{varstrong2nd}, we deduce that $\|\Hat{v}^k\| \geq \kappa^{-1}\|d^k\|$ for large $k$. Moreover, by using $\|d^k\| \geq \|\Hat{x}^k - \ox\| - \|\Hat{x}^k + d^k - \ox\|$ and \eqref{xkvkdksuperlinear}, this  leads to $\|\Hat{x}^k - \ox\| = \mathcal{O}(\|\Hat{v}^k\|)$ and}
\begin{equation}\label{supervk} 
\|\Hat{x}^k + d^k -\ox\| = o(\|\Hat{v}^k\|)\;\text{ as }\;k\to\infty. 
\end{equation} 
Unifying further \eqref{estimatesuplinear} and \eqref{supervk} results in
\begin{equation}\label{limsupphilambda}
\limsup_{k\to \infty} \frac{\varphi_\lambda(\Hat{x}^k +d^k)-\varphi_\lambda(\Hat{x}^k)}{\|\Hat{v}^k\|^2} \leq 0. 
\end{equation}
Choose $\epsilon:=\frac{\lambda(1-\lambda L_f)}{2(1+\lambda L_f)^2} - \sigma >0$ and deduce from \eqref{limsupphilambda} the existence of $k_0 \in \N$ with
\begin{equation}\label{FBEvhat} 
\varphi_\lambda(\Hat{x}^k + d^k) -\varphi_\lambda (\Hat{x}^k) \leq \epsilon\|\Hat{v}^k\|^2\;\text{ for all }\;k\geq k_0.  
\end{equation} 
Furthermore, it follows from \eqref{xkhatxkto0} that
\begin{equation} \label{estiFnorFnat:1}
\|\Hat{v}^k\| \leq \Big(\frac{1}{\lambda}+L_f\Big) \|\Hat{x}^k - x^k\|.
\end{equation} 
By \eqref{FBE=hatxexpress} and the Lipschitz continuity of $\nabla f$, we deduce from \cite[Lemma A.11]{Solo14} that 
\begin{align}\label{estFBE<}
\varphi_\lambda (x^k)  & = f(x^k) +\langle \nabla f(x^k),\Hat{x}^k-x^k\rangle + g(\Hat{x}^k) + \frac{1}{2\lambda}\|\Hat{x}^k-x^k\|^2\nonumber\\
& \geq f(\Hat{x}^k) -\frac{L_f}{2}\|\Hat{x}^k-x^k\|^2+g(\Hat{x}^k) + \frac{1}{2\lambda}\|\Hat{x}^k-x^k\|^2=\varphi(\Hat{x}^k) + \frac{1-\lambda L_f}{2\lambda}\|\Hat{x}^k-x^k\|^2.
\end{align}
Combining finally \eqref{FBEvhat}, \eqref{estiFnorFnat:1}, and \eqref{estFBE<} brings us to the estimates
\begin{align*}
&\varphi_\lambda(\Hat{x}^k +d^k) \leq \varphi_\lambda (\Hat{x}^k) +\epsilon\|\Hat{v}^k\|^2\leq \varphi(\Hat{x}^k) +\epsilon\|\Hat{v}^k\|^2 \le\varphi_\lambda (x^k) + \epsilon\|\Hat{v}^k\|^2  -\frac{1-\lambda L_f}{2\lambda}\|x^k - \Hat{x}^k\|^2\\
&\leq \varphi_\lambda (x^k) + \epsilon\|\Hat{v}^k\|^2  - \frac{\lambda(1-\lambda L_f)}{2(1+\lambda L_f)^2} \|\Hat{v}^k\|^2 = \varphi_\lambda(x^k) - \sigma \|\Hat{v}^k\|^2
\end{align*}
for all $k \geq k_0$, which tells us that $\tau_k =1$ whenever $k\geq k_0$ and thus completes the proof of the claim.\\[0.5ex]
{\bf Claim 3:} {\it The convergence rates of $\{x^k\}$, $\{\|\nabla\varphi_\lambda(x^k)\|\}$ and $\{\varphi_\lambda(x^k)\}$ are at least Q-superlinear.} Indeed, it follows from Claim~2 that the iterative sequence $\{x^k\}$ generated by Algorithm~\ref{NMcompositeglobal} eventually becomes the one $\{x^k\}$ generated by Algorithm~\ref{NMcompositelocal}. Hence the claimed $Q$-superlinear convergence of $\{x^k\}$ follows from Theorem~\ref{localNMcomposite}.

Next we verify the  $Q$-superlinear convergence of $\{\|\nabla\varphi_\lambda(x^k)\|\}$ and $\{\varphi_\lambda(x^k)\}$. To furnish this for the sequence $\{\varphi_\lambda(x^k)\}$, deduce from \eqref{strongineqVC} and $\nabla\varphi_\lambda(\ox)=0$ that  
\begin{equation*}
\varphi_\lambda(x^k)-\varphi_\lambda(\ox)\ge\frac{\mu}{2}\|x^k-\ox\|^2\;\mbox{ for large }\;k\in\N. 
\end{equation*} 
Furthermore, we get from \eqref{lipnablaFBE} whenever $k\in\N$ is large that \begin{equation*}
0< \varphi_\lambda(x^{k+1})-\varphi_\lambda(\ox)\le\frac{\ell}{2}\|x^{k+1}-\ox\|^2, 
\end{equation*} 
and thus combining the two estimates above gives us
\begin{equation*}
\frac{|\varphi_\lambda(x^{k+1})-\varphi_\lambda(\ox)|}{|\varphi_\lambda(x^k)-\varphi_\lambda(\ox)|}\le\frac{\ell}{\mu}\frac{\|x^{k+1}-\ox\|^2}{\|x^k-\ox\|^2}.
\end{equation*} 
This justifies the claimed Q-superlinear convergence of $\{\ph_\lambda(\ox)\}$ by the one established for	$\{x^k\}$.

To finish the proof of the theorem, it remains to show that the sequence $\{\nabla\varphi_\lambda(x^k)\}$ converges to $0$ with the Q-superlinear rate. Indeed, it follows from \eqref{lipnablaFBE2} that
\begin{equation*}
\|\nabla\varphi_\lambda(x^{k+1})\|=\|\nabla\varphi_\lambda(x^{k+1})-\nabla\varphi_\lambda(\ox)\|\le\ell\|x^{k+1}-\ox\|\;\mbox{ for large }\;k\in\N. \end{equation*} 
Therefore, for such $k$ we get from \eqref{strongineqVC2} that
\begin{equation*}
\langle\nabla\varphi_\lambda(x^k)-\nabla\varphi_\lambda(\ox),x^k-\ox\rangle\ge\mu\|x^k-\bar{x}\|^2,
\end{equation*} 
which implies in turn that $\|\nabla\varphi_\lambda(x^k)\|\ge\kappa\|x^k-\ox\|$ and brings us to the estimate
\begin{equation*}
\frac{\|\nabla\varphi_\lambda(x^{k+1})\|}{\|\nabla\varphi_\lambda(x^k)\|}\le\frac{\ell}{\mu}\frac{\|x^{k+1}-\ox\|}{\|x^k-\ox\|}.
\end{equation*} 
This shows that the gradient sequence
$\{\nabla\ph_\lambda(x^k)\}$ converges to 0 with the  Q-superlinear rate $k\to\infty$ due to the established convergence of $x^k\to\ox$, and thus we are done with the proof. 
\end{proof}\vspace*{-0.17in}

\begin{Remark}[\bf on the modulus of prox-regularity] \rm The convergence analysis in Theorem \ref{converNMglobal} guarantees the fast local convergence provided that the parameter $\lambda$ lies in the interval $0 < \lambda < r^{-1}$, where $r$ is the modulus of prox-regularity of $g$ at the limiting point $\ox$. In principle, this condition requires the explicit knowledge of $\ox$. Although such dependence on local and often unknown constants is somewhat uncommon in globalized numerical methods, in many important cases the value of $r$ can be determined independently of $\ox$. For example, when the regularizer $g$ is {\em $r$-weakly convex}, as discussed in \cite{bw21} for several important regularizers in machine learning models, the weak convexity modulus also serves as the modulus of prox-regularity at every point. Similarly, when $g$ is the $\ell_0$-norm, which is {\em variationally convex} at every point as shown in Theorem~\ref{varconsubl0lemma} below, the modulus of prox-regularity is $r = 0$. In this case, $\lambda$ can be chosen as any positive value.
\end{Remark}\vspace*{-0.12in}

\begin{Remark}[\bf comparison with other globalized coderivative-based Newton methods]\label{rem:comar}  \rm Several globally convergent coderivative-based Newton methods were developed in \cite{kmptjogo, kmptmp} to solve \eqref{ncvopt} in the case where $f$ is a convex quadratic function and $g$ is l.s.c.\ and convex. Besides, the generalized damped Newton methods in \cite{kmptjogo, kmptmp}  require the positive-definiteness of $\partial^2\varphi$ on the whole space, which is a restrictive assumption for many practical models. In contrast, our new coderivative-based Newton method in Algorithm~\ref{NMcompositeglobal} doesn't assume that $f$ is convex quadratic and $g$ is convex and doesn't impose the global assumption on $\partial^2\varphi$. Recall to this end that the tilt stability requirement in Assumption~\ref{assumonlimitpoint} is equivalent for continuously prox-regular functions to the positive-definiteness of the generalized Hessian only {\em at one} solution point, not on the entire space.
\end{Remark}\vspace*{-0.25in}
 
\section{Applications and Numerical Experiments}\label{sec:numerical}\vspace*{-0.05in}

This section presents explicit implementations of our newly developed generalized coderivative-based Newton method (abbr.\ GCNM) outlined in Algorithm~\ref{NMcompositeglobal} for nonsmooth, nonconvex optimization models of type \eqref{ncvopt}. We also demonstrate its competitiveness against several state-of-the-art methods. All numerical experiments were performed on a desktop equipped with a 12th Gen Intel(R) Core(TM) i5-1235U processor (10 cores, 12M Cache, up to 4.4 GHz) and 16GB of memory. The code was written in Python and executed on Google Colab.\vspace*{-0.12in} 

\subsection{Nonconvex Least Square Regression Models}\label{sec:appl0l2}\vspace*{-0.05in}

The $\ell_0$-$\ell_2$ regularized least squares regression problem, whose significance has been well recognized in applications to practical models in statistics, machine learning, etc., has been widely studied; see, e.g., \cite{hama20} and the references therein. This problem is formulated as follows:   Given an $m\times n$ matrix $A$ and a vector $b\in\R^m$, the {\em $\ell_0$-$\ell_2$ regularized least square} problem is
\begin{equation}\label{l0}
\text{minimize } \quad \varphi(x):=  \frac{1}{2}\|Ax-b\|^2 + \mu_0 \|x\|_0+ \mu_2\|x\|^2  \quad 
\text{subject to} \quad x \in \R^n,
\end{equation}
where  {$\mu_0 > 0$, $\mu_2 \geq 0$,} and $\|x\|_0$ is the  $\ell_0$-norm of $x$ counting the number of nonzero elements of $x$, i.e.,
\begin{equation}\label{Il0}
\|x\|_0: = \sum_{i=1}^n I(x_i)\;\text{ for }\; x=(x_1,\ldots,x_n)\in \R^n\;\mbox{ with }\;I(t):= \begin{cases}
1&\text{if} \quad t \ne 0,\\
0&\text{otherwise}.
\end{cases}
\end{equation}
Problem \eqref{l0} is nonconvex and nonsmooth, and it is a particular case of model \eqref{ncvopt} with the data
\begin{equation}\label{representl0fg} 
f(x):= \frac{1}{2}\|Ax-b\|^2+ \mu_2 \|x\|^2\;\text{ and }\; g(x):= \mu_0\|x\|_0. 
\end{equation} 
We directly calculate the gradient and Hessian of $f$ by 
\begin{equation}\label{nablafboth}
\nabla f(x) = A^*(Ax-b) + 2\mu_2 x\quad \text{and }\; \nabla^2 f(x) = A^*A +2\mu_2 I, \quad x \in \R^n. 
\end{equation}
It follows from  \cite[Page~121]{abs13} that  the proximal mapping of $g(x)=\mu_0\|x\|_0$ is calculated by 
\begin{equation}\label{proxofgl0}
\big(\text{\rm Prox}_{\lambda g}(x)\big)_i=\begin{cases}
x_i &\text{if}\quad |x_i|> \sqrt{2\lambda\mu_0},\\
\{0,x_i\}&\text{if}\quad |x_i|=\sqrt{2\lambda\mu_0},\\
0&\text{if}\quad |x_i|< \sqrt{2\lambda\mu_0}.
\end{cases}
\end{equation}
To run Algorithm~\ref{NMcompositeglobal}, we need to explicitly determine the direction sequence $\{d^k\}$ in the Newton step, which requires the formula for the second-order subdifferential of $g(x)=\mu_0\|x\|_0$. This is done as follows.\vspace*{-0.05in}
 
\begin{Proposition}[\bf second-order subdifferentials of the $\ell_0$-norm]\label{2ndgraofl0} The limiting subdifferential of the function $g(x)=\mu_0\|x\|_0$ at any $x\in\R^n$ is calculated by 
\begin{equation}\label{1stsubofl0} 
\partial g(x)=\left\{v\in\R^n\;\bigg|\;
\begin{array}{@{}cc@{}}
v_i=0,\;x_i\ne 0,\\
v_i\in \R,\;x_i=0
\end{array}\right\} \quad\text{whenever }\; x\in\R^n.
\end{equation} 
The second-order subdifferential  of $g$ is represented in the form
\begin{equation}\label{2ndsubl0}
\partial^2 g(x,y)(v) = \big\{w\in\R^n\;\big|\;\left(w_i ,-v_i\right)\in F\left(x_i,y_i \right),\;i=1,\ldots,n\big\}
\end{equation}
for each $(x,y)\in\gph \partial g$ and $v=(v_1,\ldots,v_n)\in\R^n$,
where the mapping $F\colon\R^2\tto\R^2$ is defined by
$$
F(t,p):= \begin{cases}
\{0\} \times \R &\text{if }\; t \ne 0, p =0,\\
(\{0\}\times \R)\cup (\R \times \{0\}) &\text{if }\; t=0,p=0,\\
\R \times \{0\} &\text{if }\; t=0, p\ne 0,\\
\emptyset &\text{otherwise}. 
\end{cases}
$$
\end{Proposition}
\begin{proof}
Let $I:\R\to\R$ {be given as in}   \eqref{Il0}. By direct computations, it is not hard to get 
$$
\partial I(t)= \begin{cases}
\{0\} &\text{if }\; t \ne 0,\\
\R &\text{if }\; t =0,
\end{cases}
$$
$$
N_{{\rm \text{\rm gph}}\,\partial I}(t, p) =  \begin{cases}
\{0\} \times \R &\text{if }\; t \ne 0, p =0,\\
(\{0\}\times \R)\cup (\R \times \{0\}) &\text{if }\; t=0,p=0,\\
\R \times \{0\} &\text{if }\; t=0, p\ne 0, 
\end{cases}
$$
which clearly imply the formulas
$$
\partial g(x) = \partial (\mu_0 I)(x_1)\times \ldots \times \partial (\mu_0 I)(x_n) = \left\{v\in\R^n\;\bigg|\;
\begin{array}{@{}cc@{}}
v_i=0,\;x_i\ne 0,\\
v_i\in \R,\;x_i=0
\end{array}\right\}.
$$
This allows us to deduce from \cite[Theorem~4.3]{BorisOutrata} the representation
\begin{equation*}
\partial^2 g(x,y)(v) = \Big\{w\in\R^n\;\Big|\;\Big(\frac{1}{\mu_0}w_i,-v_i\Big)\in N_{{\rm \text{\rm gph}}\,\partial I}\Big(x_i,\frac{1}{\mu_0} y_i\Big),\;i=1,\ldots,n\Big\}
\end{equation*}
verifying \eqref{2ndsubl0} and thus completing the proof of the proposition. 
\end{proof} \vspace*{-0.05in}

As a consequence of the above calculations, we get the following corollary ensuring the positive-semidefiniteness of the cost function in \eqref{l0}.\vspace*{-0.05in}

\begin{Corollary}[\bf positive-semidefiniteness of the cost in $\ell_0$-$\ell_2$ model] \label{PDofl0} Let $\varphi:\R^n\to\R$ {be given as in}  \eqref{l0}.  {Then $\partial^2\varphi(x,y)$ is positive-semidefinite with modulus $\kappa:=\lambda_{\text{\rm min}} (A^*A)+2\mu_2 \geq 0$ for all $(x,y) \in \gph \partial\varphi$,} i.e., 
$$
\langle z, w\rangle \geq \kappa\|w\|^2\;\text{ whenever }\;z \in \partial^2\varphi(x,y)(w), \; (x,y)\in \gph \partial\varphi, \; w \in \R^n.
$$
\end{Corollary} 
\begin{proof} It follows directly from Proposition~\ref{2ndgraofl0} that $\partial^2 g(x,v)$  is positive-semidefinite for all $(x,v)\in \gph \partial g$. Since $f$ in \eqref{representl0fg} is  $\kappa$-convex, then 
$$
\langle \nabla^2 f(x)w,w\rangle \geq \kappa \|w\|^2 \quad \text{for all }\;x\in\R^n, w \in \R^n. 
$$
By using the subdifferential sum rule from \cite[Proposition~1.107]{Mordukhovich06}, and the second-order sum rule from \cite[Proposition~1.121]{Mordukhovich06}, we justify the claimed conclusions. 
\end{proof}\vspace*{-0.05in}

Next we establish some additional variational properties of the $\ell_0$-norm, which are of their own interest and help us to solve the $\ell_0$-$\ell_2$ problem by employing the developed GCNM algorithm.\vspace*{-0.05in}

\begin{Theorem}[\bf variational convexity and subdifferential continuity of the $\ell_0$-norm]\label{varconsubl0lemma} The function $g(x):= 
 \|x\|_0$ is variationally convex at every $x \in \R^n$ for any $v\in\partial g(x)$. If in addition
 \begin{equation}\label{sufcforcont} 
\bar{v}_i \ne 0 \quad \text{whenever }\; \bar{x}_i = 0,\;\bar{v} \in \partial g(\ox),\;\mbox{ and }\;\bar{x}\in\R^n,
\end{equation} 
then $g$ is subdifferentially continuous at $\ox$ for $\ov$. 
\end{Theorem}\vspace*{-0.15in}
\begin{proof} It is not difficult to check that $g$ is l.s.c.\ on $\R^n$.  Picking $(\ox,\ov) \in \gph \partial g$, we now verify via Proposition~\ref{subgradientvar} the variational convexity of $g$ at $\ox$ for $\ov$    by  {showing that there exist neighborhoods $U$ of $\ox$ and $V$ of $\ov$ such that }  
\begin{equation}\label{g-monotonicity}
\langle x - u, v - w\rangle \geq 0\;\text {for all} \; (x,v),(u,w) \in \gph \partial g \cap W, \quad \text{where }\;  W:=\big\{(x,v) \in U \times V\;\big|\; g(x) < g(\ox) +\epsilon\big\}.
\end{equation}
To furnish this, define 
$J(\ox):=\big\{j \in \{1,2,\ldots,n\}\;\big|\;\ox_j \ne 0\big\}$ and observe that $\ox_j \ne 0$ for each $j \in  J(\ox)$. Thus there exists $\delta_j >0$ such that $0 \notin (\ox_j -\delta_j, \ox_j + \delta_j)$. Denoting
$$
U:= \big\{x=(x_1,\ldots,x_n)\in\R^n\;\big|\; x_j \in (\ox_j -\delta_j, \ox_j+\delta_j) \; \text{ if }\;j \in J(\ox)\big\}\;\mbox{ and } \; V := \R^n,
$$
we see that $U\times V$ is a convex neighborhood of $(\ox,\ov)$ and that
$$
\sum_{j \in J(\ox)} I(x_j) = \sum_{j \in J(\ox)} I(\ox_j),\quad\sum_{j \in \{1,\ldots,n\} \setminus J(\ox)} I(\ox_j)=0
$$
for all $x \in U$, where $I(\cdot)$ is taken from \eqref{Il0}. This readily implies that 
$$ 
g(x) - g(\ox)   = \sum_{j \in J(\ox)} I(x_j)+  \sum_{j \in \{1,\ldots,n\} \setminus J(\ox)} I(x_j) - \sum_{j \in J(\ox)} I(\ox_j) - \sum_{j \in \{1,\ldots,n\} \setminus J(\ox)} I(\ox_j) = \sum_{j \in \{1,\ldots,n\} \setminus J(\ox)} I(x_j).
$$
{Next we fix $\epsilon \in (0,1)$ and show that} $\partial g$ is monotone relative to the set 
\begin{align*}
W=\big\{(x,v) \in U \times V\;\big|\; g(x) < g(\ox) +\epsilon\big\}=\Big\{(x,v) \in U \times V\;\Big|\; \sum_{j \in \{1,\ldots,n\} \setminus J(\ox)} I(x_j) < \epsilon\Big\}\\
= \big\{(x,v) \in U \times V \;\big|\; x_j = 0 \; \text{for all } j \in \{1,\ldots,n\} \setminus J(\ox)\big\}.
\end{align*}
Indeed, for $(x,v), (u, w) \in \gph \partial g \cap W$ we have 
\begin{equation}\label{mono1} 
\begin{cases}
x_j = u_j = 0 &\text{whenever } j \in \{1,\ldots,n\} \setminus J(\ox),\\
x_j, u_j \in (\ox_j -\delta_j, \ox_j + \delta_j) & \text{whenever } \; j \in J(\ox).
 \end{cases}
\end{equation}
Observe that $x_j, u_j \ne 0$ for any $j \in J(\ox)$ by $0 \notin (\ox_j -\delta_j, \ox_j + \delta_j)$. Therefore, it follows from \eqref{1stsubofl0} that
\begin{equation}\label{mono2} 
 v_j = w_j = 0 \quad \text{for all } \; j \in J(\ox).    
\end{equation}
Combining \eqref{mono1} and \eqref{mono2}, we get the equality
$$
\langle x-u, v -w\rangle = \sum_{j \in J(\ox)} (x_j - u_j) (v_j - w_j) + \sum_{j  \in \{1,\ldots,n\} \setminus  J(\ox)}  (x_j - u_j) (v_j - w_j)=0,
$$ 
which tells us that {\eqref{g-monotonicity} is verified}. Therefore, $g$ is variationally convex at $\ox$ for $\ov$ due to Proposition~\ref{subgradientvar}.

Now we pick $(\ox,\ov)\in\gph \partial g$ satisfying \eqref{sufcforcont} and show that $g$ is subdifferentially continuous at $\ox$ for $\ov$. Indeed, take $(x^k,v^k)\to (\ox,\ov)$, $v^k \in \partial g(x^k)$ and then check that $g(x^k) \to g(\ox)$ as $k \to \infty$.  For $j \in J(\ox)$, we have $\ox_j  \ne 0$ ensuring that $I$ is continuous at $\ox_j$, which implies that 
\begin{equation}\label{subcont1}
I\big((x^k)_j\big) \to I(\ox_j)\;\text{ as }\; k \to \infty\;\text{ whenever }\;j \in J(\ox).   
\end{equation}
Fixing $j \in \{1,\ldots, n\}\setminus J(\ox)$ gives us $\ox_j =0$, and thus $\ov_j \ne 0$ by \eqref{sufcforcont}. Since $(v^k)_j \to \ov_j$ as $k \to \infty$, we have $(v^k)_j \ne 0$ for all large $k$. By $v^k \in \partial g(x^k)$, it follows from \eqref{1stsubofl0} that $(x^k)_j =0$ for large $k$, which yields 
\begin{equation}\label{subcont2} 
I\big((x^k)_j\big) \to 0 = I(\ox_j)\;\text{ as }\; k \to \infty\;\text{ whenever }\;j \in \{1,\ldots,n\}\setminus J(\ox). 
\end{equation} 
Combining \eqref{subcont1} and \eqref{subcont2} brings us to
$$
g(x^k) = \sum_{j=1}^n I\big((x^k)_j\big) \to \sum_{j=1}^n I(\ox_j) = g(\ox)\; \text{ as }\; k \to \infty, 
$$
which justifies the claimed subdifferential continuity and thus completes the proof of the theorem. 
\end{proof}\vspace*{-0.05in}

 {Given a matrix $B \in \R^{m \times n}$ and an index set $I \subset \{1,\ldots,n\}$, let the matrix $B_I \in \R^{m \times |I|}$ consist of the columns of $B$ corresponding to the indices in $I$, and let $B_i$ represent the $i$th column of $B$.  Now we establish explicit conditions to solve \eqref{l0} by Algorithm~\ref{NMcompositeglobal}. Observe that Proposition~\ref{2ndgraofl0} allows us to calculate the directions $d^k$ in this algorithm for \eqref{l0} via the relationships}
$$
\begin{cases}
\big(-\Hat{v}^k-\nabla^2 f(\Hat{x}^k)d^k \big)_i=0&\text{if}\quad\left(\Hat{x}^k \right)_i\ne 0,\\
\big( d^k\big)_i=0&\text{if}\quad\left(\Hat{x}^k \right)_i=0
\end{cases}
$$
in which the approximate vectors $\Hat x^k$ and $\Hat v^k$ are determined  by
$$
\Hat{x}^k \in \text{\rm Prox}_{\lambda (\mu_0\|\cdot\|_0)}\big(x^k - \lambda\nabla f(x^k\big)), \quad \Hat{v}^k = \nabla f(\Hat{x}^k)- \nabla f(x^k)  + \frac{1}{\lambda}(x^k  -\Hat{x}^k).
$$
 {For each $k\in\N$, define the matrix $X^k$ and the vector $y^k$ with the components
\begin{equation}\label{Ximage}
(X^k)_i := \begin{cases}
\big(\nabla^2 f(\Hat{x}^k)\big)_i & \text{if} \quad \left(\Hat{x}^k \right)_i\ne 0,\\
I_i  & \text{if}\quad \left(\Hat{x}^k \right)_i = 0
\end{cases}
\quad \text{and }\; y^k:= \begin{cases}
(-\Hat{v}^k)_i & \text{if} \quad \left(\Hat{x}^k \right)_i\ne 0,\\
0  & \text{if}\quad \left(\Hat{x}^k \right)_i = 0
\end{cases}
\end{equation}
and calculate the direction $d^k$ as follows:
\begin{equation}\label{dkl0l2}
d^k= \begin{cases}
\text{solution to the linear system } X^k d = y^k & \text{if } X^k \text{ is nonsingular,}\;\\
0 &\text{otherwise}.
\end{cases}
\end{equation}}\vspace*{-0.2in}
 
\begin{Theorem}[\bf solving the $\ell_0$-$\ell_2$ regularized least square problem by GCNM]\label{theoryl0l2conver} Consider the $\ell_0$-$\ell_2$ model \eqref{l0} with notation \eqref{representl0fg} for $f$ and $g$. The following assertions hold:

{\bf(i)} Algorithm~{\rm\ref{NMcompositeglobal}} either stops after finitely many iterations, or produces a sequence $\{x^k\}$ such that any  accumulation point $\ox$ of this sequence is a {$\lambda$-critical point of \eqref{l0} for all $\lambda \in (0, 1/L_f)$, where $L_f:=\lambda_{\text{\rm max}}(A^*A)+2\mu_2$}. 

{\bf(ii)} Assume that the accumulation point $\ox$ satisfies the condition  
\begin{equation}\label{subconl0}
\left(A^*(A \ox -b) + 2\mu_2 \ox\right)_i \ne 0 \quad \text{whenever }\; \ox_i=0,  
\end{equation}
and in addition either $\mu_2 > 0$, or we have   
\begin{equation}\label{metricregl0}
A_{J}  \text{ has full column rank, where } J := \big\{i \in \{1,\ldots,n\}\;\big|\;\left(A^*(A \ox -b) + 2\mu_2 \ox\right)_i = 0\big\}. 
\end{equation}
Then $\ox$ is a tilt-stable local minimizer of $\varphi$, and hence the sequence $\{x^k\}$ converges globally to $\bar{x}$ with the Q-superlinear convergence rate.
\end{Theorem}\vspace*{-0.15in}
\begin{proof}
To verify (i), observe that since $g$ is bounded from below by $0$, it follows from \cite[Exercise~1.24]{Rockafellar98} that $g$ is prox-bounded with threshold $\lambda_g =\infty$. By the calculation in \eqref{nablafboth}, $\nabla f$ is Lipschitz continuous on $\R^n$ with modulus $L_f=\lambda_{\text{\rm max}}(A^*A)+2\mu_2$ expressed via the  largest eigenvalue of $A^*A$. Therefore, 
Assumption~\ref{Lipnablaf} is guaranteed for \eqref{l0}. Applying now Theorem~\ref{converNMglobal} tells us that any accumulation point of $\{x^k\}$ is a {$\lambda$-critical point of \eqref{l0} for all $\lambda \in (0,1/L_f)$}, which therefore justifies (i). 

To prove (ii), we know from (i) that the accumulation point $\ox$ is {$\lambda$-critical} for \eqref{l0}, which implies that $0 \in \partial\varphi(\ox) = \nabla f(\ox) + \partial g(\ox)$, i.e., $-\nabla f(\ox)\in\partial g(\ox)$. 
It follows from Theorem~\ref{varconsubl0lemma} and the formula $\nabla f(\ox) = A^*(A\ox -b) + 2\mu_2\ox$ that condition \eqref{subconl0} yields the subdifferential continuity of $g$  at $\ox$ for $-\nabla f(\ox)$. The prox-regularity of $g$ at $\ox$ for $-\nabla f(\ox)$ is a consequence of the variational convexity of $g$ at $\ox$ for $-\nabla f(\ox)$.  

Finally, we verify the tilt stability of $\ox$ if either $\mu_2 >0$, or \eqref{metricregl0} holds. To proceed for $\mu_2 >0$, observe from Corollary~\ref{PDofl0} that $\partial^2 \varphi(\ox,0)$ is positive-definite, and thus $\ox$ is a tilt-stable local minimizer of $\varphi$ by \cite[Theorem~1.3]{Poli}. In the case where  \eqref{metricregl0} holds, for any $w \in \R^n$ such that $0 \in \partial^2 \varphi(\ox,0)(w)$, we get
$$
-\nabla^2 f(\ox) \in \partial^2 g(\ox,-\nabla f(\ox))(w),
$$
which implies that $A_J w_J =0$ and $w_i=0$ for indices $i$ satisfying $ \left(A^*(A \ox -b) + 2\mu_2 \ox\right)_i \ne 0 $. Combining this with \eqref{metricregl0}, we deduce that $w=0$. It follows from the positive-semidefiniteness of $\partial^2\varphi(\ox,0)$ that
$$
\langle z, w \rangle >0 \quad \text{for all }\; z \in \partial^2\varphi(\ox,0)(w), \; w \ne 0, 
$$
and so $\ox$ is a tilt-stable local minimizer of $\varphi$ by \cite[Theorem~1.3]{Poli}. The graph of $\partial g$ is the union of finitely many closed convex sets, which guarantees by \cite{Helmut} that $\partial g$ is semismooth$^*$ on its graph. All of this allows us to deduce from Theorem~\ref{converNMglobal} the $Q$-superlinear convergence of the iterative sequence $\{x^k\}$ generated by Algorithm~\ref{NMcompositeglobal} to the solution $\ox$ and thus completes the proof of the theorem. 
\end{proof}\vspace*{-0.05in}

Based on the results and calculations obtained for solving the $\ell_0$-$\ell_2$ model \eqref{l0}, we proceed to conduct numerical experiments using GCNM algorithm (Algorithm~\ref{NMcompositeglobal}) and compare its performance with the following generalized second-order Newton-type methods:\vspace*{0.03in}

{\bf(A)} \textit{The globalized semismooth$^*$ Newton method} (GSSN) recently developed in \cite[Algorithm~1]{G25}.

{\bf(B)} \textit{The ZeroFPR method with BFGS updates} (ZeroFPR) proposed in \cite{stp2}.\vspace*{0.03in}

In the conducted numerical experiments, we considered different scenarios obtained by changing the regularization terms $\mu_0$, $\mu_2$ and the size of $A$ keeping a constant column-to-row ratio of $n/m = 5$ for the matrix $A$, which is generated randomly with i.i.d. (identically and independent distributed) standard Gaussian entries, where $b$ is generated randomly with values of components from $0$ to $1$. For all the  algorithms and problems, we used $x^0 = 0$ as the starting iterate. To evaluate the performance of our proposed method, GCNM, we run GCNM, GSSN, and ZeroFPR until the condition $\|x^k - \Hat{x}^k\| \leq 10^{-6}$ is satisfied.  

In both of these experiments, we consider the following two cases. In the first case where $\mu_2 = 0.01 > 0$, generalized Newton directions of GCNM and GSSN always exist, and the superlinear convergence of both algorithms is guaranteed. In the second case where $\mu_2 = 0$, this guarantee does not longer hold, and superlinear convergence may not be achieved. More detailed information for the testing data used in the experiments is presented in Table~\ref{testing data}, where ``TN'' stands for test number. {The numerical results of the tests are summarized in Table~\ref{tab:criteria1}. In the table, ``time'' denotes the computational time in seconds, ``iter'' denotes the number of iterations required to attain the prescribed accuracy, $\delta_k := \|Ax^k - b\|$, $\eta_k := \|x^k - \hat{x}^k\|$, and ``nnz'' denotes the number of nonzero components in the computed approximate solution.}\vspace*{-0.2in}

\begin{center}
\begin{table}[H]
\centering
\begin{tabular}{|lll|} 
  
\hline
\multicolumn{1}{|c}{TN} & \multicolumn{1}{c}{n} & \multicolumn{1}{c|}{$\mu_0$}  \\ 
\hline
1 & 100& $10^{-2}$ \\
2 & 100& $10^{-3}$ \\
3   & 200  & $10^{-2}$ \\
4 & 200& $10^{-3}$ \\
5   & 400  & $10^{-2}$ \\
6 & 400& $10^{-3}$ \\
7 & 800& $10^{-2}$ \\
8  & 800  & $10^{-3}$ \\
9 & 1600& $10^{-2}$ \\
10 & 1600& $10^{-3}$  \\
\hline
\end{tabular}
\caption{Testing data information}
\label{testing data}
\end{table}
\end{center} \vspace*{-0.2in}

\begin{sidewaystable} 
\footnotesize
\centering
\begin{subtable}[t]{\textwidth}
\centering
\begin{tabular}{|l|lllll|lllll|lllll|} 
\hline
& \multicolumn{5}{c|}{GCNM} & \multicolumn{5}{c|}{GSSN} & 
\multicolumn{5}{c|}{ZeroFPR} \\ 
\hline
\multicolumn{1}{|c|}{TN} &    time & iter &  $\delta_k$ & $\eta_k$ & nnz & 
time & iter &  $\delta_k$ & $\eta_k$ & nnz & time & iter &  $\delta_k$ & $\eta_k$ & nnz  \\ 
\hline
1 &   0.007 & 3 & 7.8E-04  & 2.9E-17  & 76  & 0.044 & 13 & 1.2E-03 & 7.2E-18 &  33 &  4.243 & 455 & 1.2E-02 & 6.5E-07 & 85 \\
2 &  0.003 &  3 & 7.7E-04 & 2.7E-17   &  96 & 0.028 & 12 & 7.7E-04 & 2.0E-17 &  84 &  1.747 & 435 & 1.0E-01 & 9.7E-07 & 94\\
3 &   0.021 & 4 & 5.7E-04 & 2.5E-17 &  170 &  0.139 & 12 & 6.1E-04 & 1.5E-17 & 126 &   5.721 & 988 & 7.8E-03 & 7.4E-07 & 184 \\
4 &   0.013 & 3 & 5.7E-04 & 2.3E-17   &  191 & 0.039 & 12 &  5.7E-04 & 1.8E-17 & 162 & 10.721 & 1196 & 1.6E-03 & 3.7E-07 & 195\\
5 &  0.225 & 6  & 4.6E-04 & 3.4E-17   & 320 & 0.757 & 12 & 4.8E-04 & 1.9E-17 &  240 &  29.601 & 1517 & 4.8E-02 & 9.1E-07 & 341\\
6 &   0.015 & 3 & 4.5E-04 & 3.2E-17 &  376 & 0.359 & 14 &  4.5E-04 & 2.7E-17 & 318 &   40.462 & 2234 & 1.2E-02 & 7.2E-07 & 356 \\
7 &   0.281 & 7 & 2.5E-04 & 3.8E-17 & 645   & 1.266 & 13 & 2.8E-04 & 1.8E-17 & 392 &   273.805 & 2839  & 3.8E-02 & 9.8E-07 & 721 \\
8 &   0.158 & 4 & 2.4E-04 & 3.5E-17  & 744  & 2.215 & 13 & 2.4E-04 & 2.9E-17 & 605    & 26.510 & 266 &  2.4E-04 & 8.0E-07 & 765\\
9 &   1.168 &  5 & 1.9E-04 & 5.6E-17   & 1315 & 10.298 & 14 & 2.0E-04 & 3.0E-17 &  877 &   408.838 & 1562 & 3.2E-04 & 9.2E-07 & 1455 \\
10 &  0.630 & 3 & 1.9E-04 & 5.4E-17  & 1517  & 9.756 & 15 & 1.9E-04  & 4.8E-17  & 1370  & 227.216 & 362 & 1.9E-04 & 9.3E-07 & 1456\\
\hline
\end{tabular}
\caption{$\mu_2 > 0$}
\label{tab:mu_positive}
\end{subtable}

\vspace{0.5cm}

\begin{subtable}[t]{\textwidth}
\centering
\begin{tabular}{|l|lllll|lllll|lllll|} 
\hline
  & \multicolumn{5}{c|}{GCNM} & \multicolumn{5}{c|}{GSSN} & \multicolumn{5}{c|}{ZeroFPR} \\ 
\hline
\multicolumn{1}{|c|}{TN} &    time & iter & $\delta_k$ & $\eta_k$ & nnz & 
time & iter & $\delta_k$ & $\eta_k$ & nnz & 
time & iter & $\delta_k$ & $\eta_k$ & nnz  \\ 
\hline
1 &  0.002 & 3 & 5.2E-15 & 5.4E-15 & 60   &  0.012 & 9 & 2.3E-15 & 8.5E-18  & 27  & 0.081 & 37 & 1.2E-05 & 4.9E-07 & 68 \\
2 &   0.003 & 4  & 4.0E-14 & 1.3E-14 & 79   & 0.107 & 22 & 8.9E-06 & 1.3E-07 & 62  & 0.215 & 33 & 1.5E-05 & 8.4E-07 & 99 \\
3 &   0.003 & 2 & 1.5E-14 & 5.2E-15  & 109  & 0.091 & 10 & 7.2E-02 & 7.2E-18 &  37 &   0.809 & 42 & 2.8E-05 & 8.1E-07 & 106 \\
4 &   0.016 & 2 & 1.5E-14 & 4.4E-15  & 169 & 1.394 & 33 & 8.6E-05 & 7.8E-07 &  110 &   0.878 & 36 & 2.9E-05 & 7.5E-07 & 184 \\
5 &  0.028 & 3 & 3.0E-12 & 1.1E-12 & 228   &  0.365 & 10 & 2.3E-02 & 5.2E-18 & 74   & 2.051 & 53 & 5.9E-05 & 8.1E-07 & 176 \\
6 &  0.016 & 2 & 9.9E-14 & 2.8E-14  & 339  & 1.577 & 27 & 5.6E-05 & 4.4E-07   &  231 &  1.774 & 41 & 4.7E-05 & 6.9E-07 & 385 \\
7 &  0.062 & 2 & 1.3E-13 & 8.9E-14    & 448 &  2.068 & 9 & 2.2E-01 & 3.1E-17  & 148    & 10.382 & 49  & 1.2E-04 & 9.1E-07 & 428 \\
8 &   0.145 & 5  & 5.8E-11 & 1.9E-11   &  641 & 9.959 & 36 & 1.2E-04 & 4.9E-07 & 424 &    7.403 & 36 & 8.2E-05 & 9.2E-07 & 757 \\
9 &   0.571 & 4 & 1.9E-11 & 4.9E-12   & 891 &  14.230  & 10 & 1.2E-01 & 3.3E-17   &  300 & 41.439 & 53 & 1.7E-04 & 8.7E-07 & 900\\
10 &  0.297 & 7 & 6.7E-13 & 1.5E-13   & 1303 & 34.760 & 32 & 2.9E-04 & 9.0E-07 & 936 &    26.853 & 37 & 1.0E-04 & 8.2E-07 & 1501\\
\hline
\end{tabular}
\caption{$\mu_2 = 0$}
\label{tab:mu_zero}
\end{subtable}
\caption{Performance of GCNM,  GSSN, ZeroFPR for solving \eqref{l0}, where all algorithms are required to satisfy $\|x^k - \Hat{x}^k\| \leq 10^{-6}$}
\label{tab:criteria1}
\end{sidewaystable}\vspace*{-0.2in}

In many practical applications, the optimization problem \eqref{l0} is employed to find an approximate solution to the underdetermined linear system $Ax = b$, where $A$ is an $m \times n$ matrix with $m < n$. To assess the effectiveness of approximate stationary points obtained from solving \eqref{l0}, we evaluate the quality of the solutions by comparing the norm of the residual $\delta_k=\|Ax^k - b\|$ and the deviation from the  solution $\eta_k=\|x^k - \hat{x}^k\|$ after $k$ iterations. The results presented in Table \ref{tab:criteria1} consistently demonstrate the superior performance of GCNM over GSSN and ZeroFPR across various evaluation metrics. In the first case where $\mu_2 = 0.01 > 0$, GCNM ensures the existence of generalized Newton directions and guarantees superlinear convergence while notably outperforms GSSN in   time. This significant difference can be attributed to the simpler construction of the Newton direction in GCNM, which contrasts with the approach in GSSN that requires boundedness of the direction and involves computationally expensive matrix operations. Additionally, the line search strategy used in GSSN tends to be more complex and requires more steps compared to the more efficient line search method implemented

in Algorithm~\ref{FBEbasedLM} for GCNM. {From Table \ref{tab:criteria1}, we can observe that the solutions obtained by the different methods differ, as expected for a nonconvex problem. The line search strategy used in GSSN generally requires more iterations, whereas GCNM in these experiments appears to take mostly full Newton steps, resulting in fewer iterations.}  Furthermore, in terms of solution quality as discussed earlier, GCNM generally yields smaller residuals $\|Ax^k - b\|$. For larger problem instances, ZeroFPR produces smaller residuals $\|Ax^k - b\|$; however, this algorithm requires higher computational cost and longer runtimes compared to both GCNM and GSSN. It is also worth noting that in many test problems, GSSN produced smaller residuals $\|x^k - \hat{x}^k\|$ compared to GCNM and ZeroFPR.   
The second case, where $\mu_2 = 0$ and superlinear convergence is not guaranteed, further demonstrates the robustness of GCNM. This trend is consistently observed across most tests, with GCNM generally requiring substantially less  time. \vspace*{-0.1in}

\subsection{Student’s t‑Regression with $\ell^0$-Penalty}\vspace*{-0.05in}

In many inverse problem applications, the goal is to compute an approximate solution $\ox \in \R^n$ to the system $Ax = b$, where $A \in \R^{m \times n}$ and $b \in \R^m$ with $m < n$. Apart from  solving the optimization problem \eqref{l0}, an alternative approach is to consider the following optimization formulation:
\begin{equation}\label{student}
\min \quad \varphi(x):= \sum_{i=1}^m \log\left(1+ \frac{(Ax-b)_i^2}{\nu}  \right) +\mu \|x\|_0 \quad 
\text{subject to} \quad x \in \R^n,
\end{equation}
where $\nu, \mu >0$. For more information on Student’s t-distribution, we refer to \cite{afhv12,mu14}. Note  that problem \eqref{student} is nonconvex and nonsmooth, and it is a particular case of the structured model in \eqref{ncvopt} with the data
\begin{equation}\label{representstudent} 
f(x):= \sum_{i=1}^m \log\left(1+ \frac{(Ax-b)_i^2}{\nu}  \right) \;\text{ and }\; g(x):= \mu \|x\|_0. 
\end{equation} 
Observe that $f$ is smooth but nonconvex with the gradient and Hessian matrix calculated by
\begin{equation}\label{nablaimage}
\nabla f(x)=2A^* u, \text{ where } u_i=\frac{(Ax-b)_i}{\nu+(Ax-b)_i^2},\quad i=1,\ldots,n, 
\end{equation} 
\begin{equation}\label{hessianimage}
\displaystyle \nabla^2 f(x) = 2 A^* \begin{pmatrix}
\frac{\nu-(Ax-b)_1^2}{\left(\nu+(Ax-b)_1^2\right)^2}  &  0  & \ldots & 0\\
0 & \frac{\nu-(Ax-b)_2^2}{\left(\nu+(Ax-b)_2^2\right)^2} & \ldots & 0\\
\vdots & \vdots & \ddots & \vdots\\
0 & 0 & \ldots & \frac{\nu-(Ax-b)_n^2}{\left(\nu+(Ax-b)_n^2\right)^2} 
 \end{pmatrix} A \quad \text{for all }\; x \in \R^n. 
\end{equation}
The proximal mapping of $g$  in \eqref{representstudent} is given by \eqref{proxofgl0}. To run Algorithm~\ref{NMcompositeglobal}, {we need to determine} the direction sequence $\{d^k\}$ in the Newton step, which requires calculating the second-order subdifferential of $g(x)=\mu\|x\|_0$ presented in Proposition \ref{2ndgraofl0}.  So we can express $d^k$ in Algorithm~\ref{NMcompositeglobal} for solving \eqref{student} via the conditions
$$
\begin{cases}
\big(-\Hat{v}^k- \nabla^2 f(\Hat{x}^k)d^k \big)_i=0&\text{if}\quad\left(\Hat{x}^k \right)_i\ne 0,\\
\big( d^k\big)_i=0&\text{if}\quad\left(\Hat{x}^k \right)_i=0,
\end{cases}
$$
where the vectors  $\Hat{x}^k, \Hat{v}^k$ are given by
$$
\Hat{x}^k = \text{\rm Prox}_{\lambda g}(x^k - \lambda\nabla f(x^k)), \quad \Hat{v}^k = \nabla f(\Hat{x}^k)- \nabla f(x^k)   + \frac{1}{\lambda}(x^k -\Hat{x}^k), 
$$
and where $\nabla f$ and $\nabla^2 f$ are calculated in \eqref{nablaimage} and \eqref{hessianimage}, respectively. This allows us to choose $d^k$  for each $k\in\N$ by using the explicit formulas in \eqref{dkl0l2} and \eqref{Ximage}.\vspace*{-0.05in}

\begin{Theorem}[\bf solving Student's t-regression problem by GCNM]  Consider problem \eqref{student} with the representations of $f$ and $g$ given in \eqref{representstudent} and the explicit computations presented above. Then  Algorithm~{\rm\ref{NMcompositeglobal}}  either stops after finitely many iterations, or produces a sequence $\{x^k\}$ such that if $\ox$ is an accumulation point of this sequence, then  $\ox$  is a {$\lambda$-critical point of \eqref{student} for all $\lambda \in (0,1/L_f)$, where $L_f=2\|A\|_1\|A\|_\infty$.}
If in addition $\varphi$ is subdifferentially continuous and variationally strongly convex at $\ox$ for $0$, then the sequence $\{x^k\}$  Q-superlinearly converges to $\bar{x}$. 
\end{Theorem}\vspace*{-0.15in}
\begin{proof} Observe that since $g$ is bounded from below by $0$, it follows from \cite[Exercise~1.24]{Rockafellar98} that $g$ is prox-bounded with threshold $\lambda_g =\infty$. Moreover, $f$ is smooth with the gradient and Hessian are computed in \eqref{nablaimage} and \eqref{hessianimage}, respectively. It follows from \eqref{nablaimage} that $\nabla f$ is Lipschitz with constant $L=2\|A\|_1\|A\|_\infty$. 
Therefore, Assumption~\ref{Lipnablaf} is fulfilled for problem \eqref{student}. Thus it follows from Theorem~\ref{converNMglobal} that if $\ox$ is an accumulation point of $\{x^k\}$, then $\ox$ is a {$\lambda$-critical point of \eqref{student} for all $\lambda \in (0,1/L_f)$}. 

It remains to verify the Q-superlinear convergence of $\{x^k\}$ provided that $\varphi$ is subdifferentially continuous and variationally strongly convex at $\ox$ for $0$. This ensures that $\varphi$ is continuously prox-regular  at $\ox$ for $0$, and so $\ox$ is a tilt-stable local minimizer of $\ph$ by Proposition~\ref{equitiltstr}. As mentioned above, the graph of $\partial g$ is the union of finitely many closed convex sets, which ensures that $\partial g$ is semismooth$^*$ on its graph. Therefore, the superlinear convergence to $\ox$ of the iterates $\{x^k\}$ generated by Algorithm~\ref{NMcompositeglobal} is guaranteed by Theorem \ref{converNMglobal}.
\end{proof}\vspace*{-0.05in}
 
Based on the results and computations obtained for solving model \eqref{student}, we conduct numerical experiments using our GCNM Algorithm~\ref{NMcompositeglobal} and compare its performance with the same second-order Newtonian methods (A) and (B) as in Subsection~\ref{sec:appl0l2}. The test examples are randomly generated in a similar manner as in \cite{mu14}. Specifically, we generate a true sparse signal $x^{\text{true}}$ of length $n$ with $k = \floor{n/40}$ nonzero entries, where the indices are selected at random. The nonzero components are computed via $x^{\text{true}}_i = \eta_1(i)10^{\eta_2(i)}$, where $\eta_1(i) \in \{-1,1\}$ is a random sign and $\eta_2(i)$ is independently drawn from a uniform distribution over $[0, 1]$. The observation vector $b \in \R^m$ is generated by adding Student’s t-distributed noise with the degree of freedom $4$ rescaled by $0.1$ to $Ax^{\text{true}}$. We set $\nu = 1$ and examine different scenarios by varying the regularization parameter $\mu \in \{10^{-1}, 10^{-2}, 10^{-3}\}$, while maintaining a constant column-to-row ratio $n/m = 8$ for the matrix $A$. Each algorithm is initialized with $x^0 = A^*b$ and terminated when the condition $\|x^k - \Hat{x}^k\| \leq 10^{-4}$ is satisfied.
 
\begin{sidewaystable} 
\footnotesize
\centering
\begin{subtable}[t]{\textwidth}
\centering
\begin{tabular}{|l|lllll|lllll|lllll|} 
\hline
  & \multicolumn{5}{c|}{GCNM} & \multicolumn{5}{c|}{GSSN} & \multicolumn{5}{c|}{ZeroFPR}\\ 
\hline
\multicolumn{1}{|c|}{n} &   time & iter  &   $\delta_k$ &  $\eta_k$ & nnz & 
   time & iter  &  $\delta_k$ &  $\eta_k$ & nnz & 
   time & iter  &   $\delta_k$ &  $\eta_k$ & nnz   \\ 
\hline
40 &  0.002 & 2 & 2.6E-06 & 2.9E-05  & 34  & 0.023 & 13 &  5.7E-04 & 3.3E-05   & 7 &   0.026 & 15 & 3.9E-05 & 9.0E-05  & 38   \\
80 &   0.003 & 3 &  1.8E-05 & 8.7E-07 & 51 &   0.548 & 349 &  5.2E-03 & 9.9E-05 &28  & 0.045 & 23 & 3.7E-05 & 8.8E-05 & 51 \\
160   & 0.004 & 2 & 1.5E-05 & 3.4E-07 & 85 &   1.707 & 703  & 1.7E-02 & 9.9E-05  &32  &  0.192 & 62 & 2.6E-04 & 8.0E-05 & 135 \\
320 &  0.035  &4 & 5.6E-06 & 8.7E-08 &   265 & 8.181 & 402 & 1.1E-02 & 9.9E-05  & 160 & 0.238 &25  & 4.9E-05 & 6.4E-05 & 311 \\
640 &  0.132 & 5 &   1.3E-07 & 1.2E-09 & 580&  34.331 & 426 & 1.7E-02 & 9.9E-05  & 374 & 1.335 & 24& 2.8E-05 & 7.9E-05 & 581 \\
1280   & 0.460 & 3 & 6.1E-07 & 2.7E-09 & 1172 &  241.641 & 556 & 3.3E-02 & 9.9E-05  & 775 & 5.629 & 15 & 3.3E-05 & 9.8E-05  & 1180  \\
2560   & 2.472 & 7 & 2.4E-07 & 5.6E-10 & 2377&   1276.129 & 435 & 5.6E-02 & 9.9E-05 & 1660   & 88.735 & 36 & 1.9E-04 & 7.3E-05 & 2435 \\
\hline
\end{tabular}
\caption{$\mu =10^{-1}$}
\label{subtab:student1}
\end{subtable}

\vspace{0.5cm}

\begin{subtable}[t]{\textwidth}
\centering
\begin{tabular}{|l|lllll|lllll|lllll|} 
\hline
  & \multicolumn{5}{c|}{GCNM} & \multicolumn{5}{c|}{GSSN} & \multicolumn{5}{c|}{ZeroFPR} \\ 
\hline
\multicolumn{1}{|c|}{n} &   time & iter  &   $\delta_k$ &  $\eta_k$ & nnz & 
   time & iter  &   $\delta_k$ &  $\eta_k$ & nnz & 
   time & iter  &   $\delta_k$ &  $\eta_k$ & nnz  \\ 
\hline
40 &  0.003 &2  & 1.4E-06 & 1.6E-07 &  37 &  0.032 & 25 & 8.0E-04 & 6.9E-05 & 18  & 0.013 & 10 &  2.8E-05 & 5.5E-05 & 37\\
80 &   0.004 & 3 & 9.7E-08 & 5.5E-09   & 73 & 0.106 & 68  & 2.5E-03 & 9.9E-05 & 51  & 0.019 & 9 & 3.2E-05 & 7.8E-05 & 73 \\
160 &   0.010 & 3 &  2.0E-07 & 5.5E-09 & 128 &   1.139 & 162 &  6.5E-03 & 9.9E-05   &  59 & 0.102 & 16 & 1.8E-05 & 5.5E-05 & 155  \\
320  & 0.064 & 5 & 1.6E-09 & 2.9E-11 & 308&   3.356 & 111 & 6.8E-03 & 9.9E-05 & 247  & 0.480 & 10  & 3.1E-05 & 8.9E-05 & 310  \\
640 &   0.034 & 3 & 3.1E-08 & 2.8E-10 & 620 &  10.493 & 98 & 1.3E-02 & 9.9E-05  & 494  & 0.709 & 14 & 5.0E-05 & 7.7E-05 & 621 \\
1280 &  0.510  & 3 & 2.0E-09 & 9.2E-12 & 1251 &   13.604 & 20  & 2.1E-02 & 8.7E-05 & 1067 &   7.435 & 17  & 3.0E-05 & 8.3E-05 & 1255  \\
2560   & 2.646 & 4 & 1.0E-09 & 1.1E-12 &  2515 &  86.497 & 29 & 1.9E-02 & 4.2E-05 & 2194  & 26.892 & 11 & 1.1E-04 & 6.6E-05 & 2525 \\
\hline
\end{tabular}
\caption{$\mu=10^{-2}$}
\label{subtab:student2}
\end{subtable}

\vspace{0.5cm}

\begin{subtable}[t]{\textwidth}
\centering
\begin{tabular}{|l|lllll|lllll|lllll|} 
\hline
  & \multicolumn{5}{c|}{GCNM} & \multicolumn{5}{c|}{GSSN} & \multicolumn{5}{c|}{ZeroFPR}\\ 
\hline
\multicolumn{1}{|c|}{n} &   time & iter  &   $\delta_k$ &  $\eta_k$ & nnz & 
   time & iter  &   $\delta_k$ &  $\eta_k$ & nnz & 
   time & iter  &   $\delta_k$ &  $\eta_k$ & nnz  \\ 
\hline
40 &  0.002 & 2 &  6.1E-08 & 6.7E-09  & 38  &  0.053 & 26 & 7.6E-04 & 7.3E-05 & 25 &   0.010 & 8  & 4.3E-05 & 8.4E-05 & 38   \\
80 &   0.002 & 2 & 1.4E-10 & 7.9E-12 & 79    & 0.003 & 2 & 7.3E-04 & 4.2E-05  & 79 & 0.010 & 5  & 3.2E-05 & 7.3E-05 & 80   \\
160 &   0.005 & 3 &  4.1E-10 & 1.2E-11 & 153 &  0.123 & 32 & 4.2E-03 & 9.9E-05 & 113 &  0.028 & 7  & 1.8E-05 & 6.8E-05 &153 \\
320 &  0.005 & 2 & 1.1E-14 & 2.0E-17 & 320  & 0.002 & 2 & 1.4E-15 & 8.4E-18 & 320 & 0.010 &2 & 1.4E-15 & 9.3E-18 &320 \\
640   & 0.092 & 3 & 2.5E-10 & 2.3E-12 & 632 &  5.104 & 24 & 5.3E-03 & 4.4E-05 &  561 & 1.037 & 7 & 1.9E-05 & 5.8E-05 & 632 \\
1280  & 0.175 & 5 & 1.5E-10 & 6.9E-13 &  1273& 14.976 & 23 &  4.5E-03 & 1.9E-05  & 1171 & 1.995 & 7 & 2.0E-05 & 5.6E-05 & 1270 \\
2560 & 2.504 & 4 & 2.9E-10 & 1.2E-13    & 2548 & 84.921 & 28 & 3.9E-03 & 8.7E-06  & 2383 & 14.313 & 6  & 1.8E-05 & 6.6E-05  & 2553 \\
\hline
\end{tabular}
\caption{$\mu=10^{-3}$}
\label{subtab:student3}
\end{subtable}
\caption{Performance of GCNM,  GSSN, ZeroFPR for solving \eqref{student}, where all algorithms are required to satisfy $\|x^k - \Hat{x}^k\| \leq 10^{-4}$}
\label{tab:studenttest1}
\end{sidewaystable}
 
All the numerical results for the algorithms applied to \eqref{student} are reported in Table~\ref{tab:studenttest1}. In the table, ``time'' denotes the computational time in seconds, ``iter'' denotes the number of iterations required to reach the prescribed accuracy, $\delta_k := \|Ax^k - b\|$, $\eta_k := \|x^k - \Hat{x}^k\|$, and ``nnz'' denotes the number of nonzero components in the obtained approximate solution. Across all tested scenarios, GCNM consistently exhibits the most efficient performance in terms of the  time. It reaches the stopping criterion $\eta_k\leq 10^{-4}$ significantly faster than both GSSN and ZeroFPR regardless of either the problem size, or the regularization parameter $\mu$. Furthermore, GCNM obtains the smallest values of the residual $\delta_k$, often in the order of $10^{-7}$ or lower, indicating therefore highly accurate approximate solutions to the equation $Ax = b$. ZeroFPR, while showing slightly larger values of $\delta_k$ than GCNM, still maintains reasonably accurate approximations and requires a significantly fewer  time than GSSN. It performs especially well at moderate problem sizes and smaller values of $\mu$. {From the table, we can observe that the solutions obtained by the different methods differ, as expected for a nonconvex problem. The line search strategy used in GSSN generally requires more iterations, whereas GCNM in these experiments appears to take mostly full Newton steps, resulting in fewer iterations. We present below the following special 2-dimensional example of \eqref{student}, which demonstrates that the methods can converge to different points, potentially minimizers or $M$-stationary points.}\vspace*{-0.05in}

\begin{Example}{\bf(illustration of different algorithm performances).} \rm Consider the optimization problem \eqref{student} with the initial data:
$$
m=1, \;n=2,\; A=[1 \quad 1 ], \; b =1,\; \nu=1,\; \mu =\frac{1}{10},
$$
which corresponds to the explicit formulation
\begin{equation}\label{specialstudent}
\min \quad \varphi(x_1,x_2)=\log\left[1+(x_1+x_2 - 1)^2\right] + \frac{1}{10} \|(x_1,x_2)\|_0 \quad \text{\rm subject to }\; (x_1,x_2) \in \R^2. 
\end{equation}
Using the subdifferential formula for the $\ell_0$-norm from Proposition~\ref{2ndgraofl0} together with the gradient of the smooth term $f$ in \eqref{nablaimage}, we can find the M-stationary points of \eqref{specialstudent} as follows:
$$
\mathcal{X}_M:= \{(x_1,x_2) \in \R^2\;|\; (0,0)\in\partial\varphi(x_1,x_2)\}= \{(x_1,x_2) \in \R^2|\; x_1 + x_2 = 1\}.  
$$
It is clear that $(1,0)$ and $(0,1)$ are global minimizers of \eqref{specialstudent}. We now conduct numerical experiments using our GCNM Algorithm~\ref{NMcompositeglobal} together with other methods including GSSN and ZeroFPR. All the methods are tested from two different starting points and are terminated once the condition $\|x^k - \Hat{x}^k\| \leq 10^{-4}$ is satisfied.
The results are summarized in Table \ref{tab:l0student2dimension}. As indicated by the numerical results, both GCNM and GSSN converge to the global optimal solution $(0,1)$ when initialized at $(5,5)$, whereas ZeroFPR converges to the M-stationary point $(1/2,1/2)$. When all algorithms are initialized at $(-5,5)$, they converge to the M-stationary point $(-9/2,\,11/2)$. These observations motivate further investigation into the conditions under which our algorithm converges to a global optimal solution rather than to a merely M-stationary point, a behavior that is guaranteed only under the theoretical assumptions developed in the previous sections.

\begin{table}
\centering
\caption{Numerical results for different starting points}
\label{tab:l0student2dimension}
\begin{tabular}{|c|c|c|c|}
\hline
starting point 
& method 
& iterations 
& approximate solution \\ \hline

\multirow{3}{*}{$(5,5)$ }
& GCNM    & 71 & $(0,0.99999977)$ \\ 
& GSSN    & 18 & $(0,0.99999939)$ \\ 
& ZeroFPR & 30 & $(0.50008385,0.50008385)$ \\ \hline

\multirow{3}{*}{$(-5,5)$}
& GCNM    & 7 & $(-4.50144135,   5.50144135)$ \\ 
& GSSN    & 4 & $(-4.51120306,5.51120444)$ \\ 
& ZeroFPR & 21 & $(-4.49992114,  5.50007886)$\\ \hline

\end{tabular}
\end{table}
    
\end{Example}

\vspace*{-0.12in}

\subsection{Restoration of Noisy Blurred Images}\vspace*{-0.05in}

The next numerical experiment addresses the restoration of a noisy, blurred image using a nonconvex regularization approach. Given a blur operator $A \in \R^{n \times n}$ and an observation vector $b \in \R^n$ representing blurred and noisy images, the goal is to estimate the unknown original image $\bar{x} \in \R^n$ that approximately satisfies $A\bar{x} = b$.

We consider the restoration of a $256 \times 256$  cameraman test image. The degradation process involves applying a Gaussian blur with a kernel of size $9 \times 9$ and standard deviation $4$, followed by the addition of standard Gaussian noise with standard deviation $10^{-3}$.   The restoration is formulated as a nonconvex optimization problem \eqref{l0}. We evaluate the performance under different settings of the smooth regularization parameter $\mu_2 \in \{5 \times 10^{-2}, 5 \times 10^{-3}\}$ and of the nonconvex regularization parameter $\mu_0 \in \{10^{-4}, 10^{-5}\}$. All methods are initialized at $x^0 = b$. We evaluate and compare the performance of our proposed \textit{GCNM algorithm} (Algorithm~\ref{NMcompositeglobal}) against two efficient methods for image restoration: the \textit{proximal gradient method} (PGM), also known as the \textit{forward-backward splitting method}, as developed in~\cite[Section~5.1]{abs13}, and the \textit{ZeroFPR method with BFGS updates} introduced in~\cite{stp2}.  We run GCNM,  PGM, and ZeroFPR until the stopping criterion $\eta_k := \|x^k - \Hat{x}^k\| \leq 10^{-2}$ is satisfied. The original and blurred images are shown in Figure~\ref{fig:blur}. As illustrated in Figure~\ref{fig:recovered2}, all three solvers produce visually satisfactory restorations. Table~\ref{tablerestore} presents the residuals $\eta_k$, and the computational time (in seconds) required to reach the prescribed accuracy. Among the three methods, GCNM and ZeroFPR consistently outperform PGM, achieving the target accuracy in smaller time. GCNM  yields the smallest residual $\eta_k$ within a comparable or even significantly shorter runtime, highlighting its superior efficiency in these scenarios. Figure~\ref{fig:recovered2graph} shows convergence rates of the algorithms in one of the generated
problems.

\begin{figure}[H]
\centering
\includegraphics[scale=0.6]{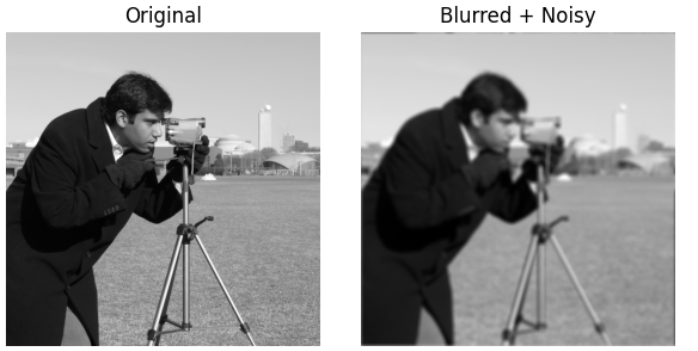}
\caption{The original and blurred cameramen test images}
\label{fig:blur}
\end{figure}
 
\begin{figure}[H]
\centering
\includegraphics[scale=0.6]{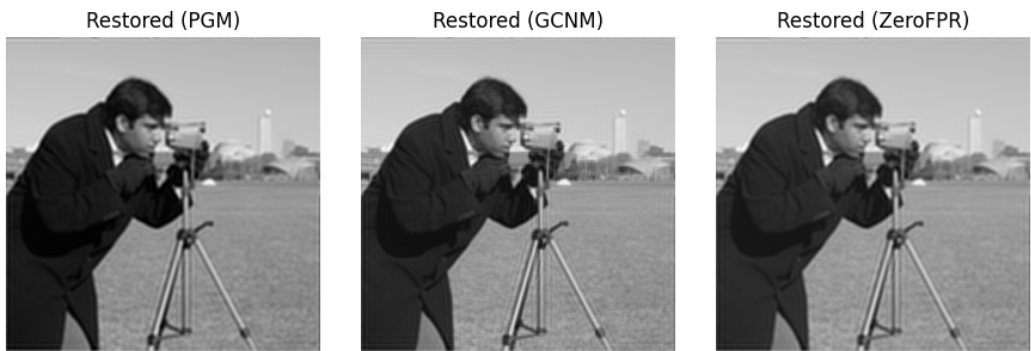}
\caption{Recovered images with the three solvers with $\mu_0 =10^{-4}$, $\mu_2 = 5\times 10^{-3}$} 
\label{fig:recovered2}
\end{figure}\vspace*{-0.2in}
 
\begin{figure}[H]
\centering
\includegraphics[scale=0.45 ]{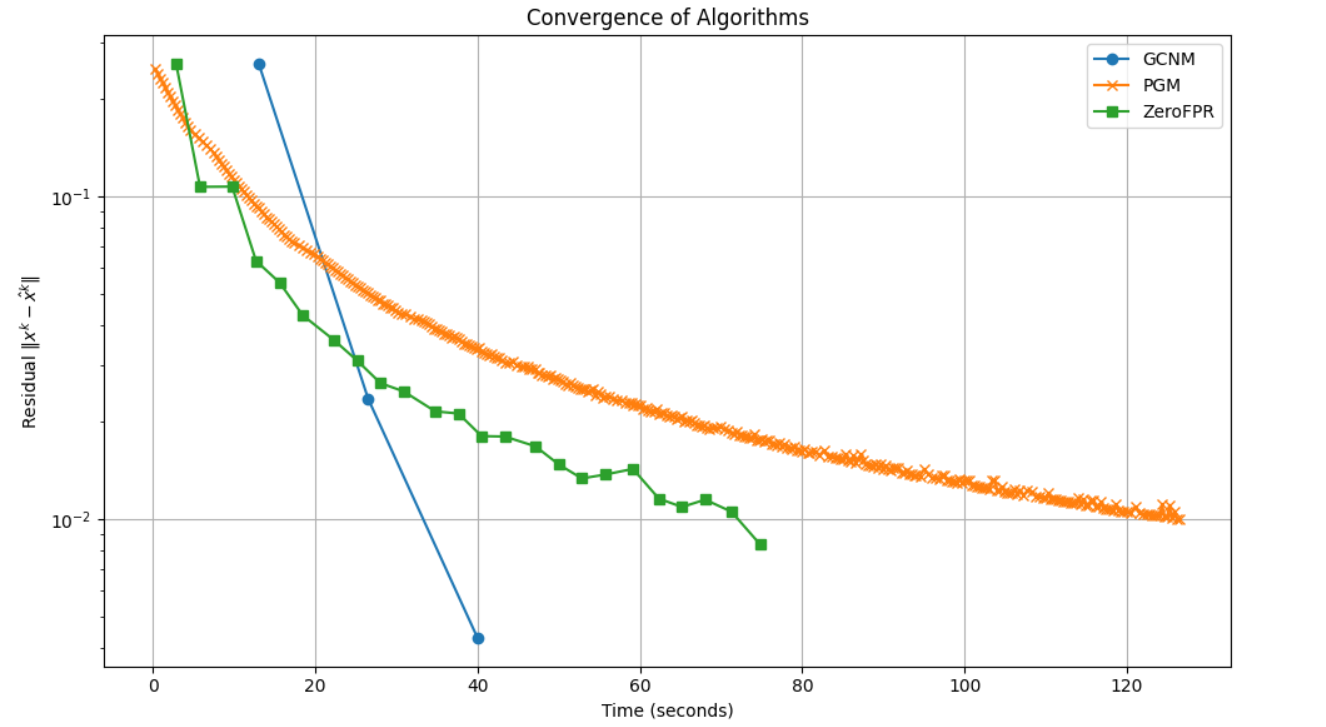}
\caption{Residual and  times of three solvers with $\mu_0 =10^{-4}$, $\mu_2 = 5\times 10^{-3}$} 
\label{fig:recovered2graph}
\end{figure}\vspace*{-0.2in}

\begin{center}
\begin{table}[H]
\centering
\begin{tabular}{|ll|ll|ll|ll|} 
\hline
\multicolumn{2}{|c}{Parameters}  &
\multicolumn{2}{|c|}{GCNM}  & \multicolumn{2}{c|}{PGM}& \multicolumn{2}{c|}{ZeroFPR}
\\ 
\hline
\multicolumn{1}{|c}{$\mu_0$} & 
\multicolumn{1}{c|}{$\mu_2$} &
\multicolumn{1}{c}{$\eta_k$} & \multicolumn{1}{c|}{time}   & \multicolumn{1}{c}{$\eta_k$} & \multicolumn{1}{c|}{time} &
\multicolumn{1}{c}{$\eta_k$} & \multicolumn{1}{c|}{time} 
\\ 
\hline
$10^{-4}$& $5.10^{-2}$ &  3.1E-03& 21.24    &    9.8E-03& 49.02   &    7.3E-03& 33.57   \\
$10^{-4}$& $5.10^{-3}$ &  5.3E-03& 40.27     & 9.9E-03& 125.29   &     9.8E-03& 76.69   \\
$10^{-5}$& $5. 10^{-2}$ &  3.5E-03& 24.51    &    9.9E-03& 49.04 &    8.8E-03& 27.67     \\
$10^{-5}$& $5.10^{-3}$  & 4.4E-03  & 33.15    & 9.9E-03  & 157.64  &    9.9E-03&  74.94   \\
\hline
\end{tabular}
\caption{Numerical comparisons of three solvers on nonconvex image restoration}
\label{tablerestore}
\end{table}
\end{center} \vspace*{-0.5in}

\section{Conclusions and Future Research}\label{sec:conclusion}\vspace*{-0.05in}

In this paper, we propose and develop fast converging local and global generalized Newton methods to solve problems of nonconvex and nonsmooth optimization. The developed algorithms are far-going extensions of the classical Newton method, where the standard Hessian is replaced by its generalized version specifically designed for nonsmooth (of first-order and second-order) functions. This construction is coderivative-generated while giving rise to the names of our generalized Newton methods.  {The results highlight the good performance of the proposed algorithms across a range of applications including  regularized least squares regression, Student’s t-regression models, and image restoration tasks, as demonstrated through comprehensive numerical experiments. In our future research, we plan to extend the spectrum of practical models, which can be solved by using our algorithms and their further developments.}\vspace*{-0.2in}

\section*{Acknowledgements} \vspace*{-0.05in}
The authors are very grateful to two anonymous referees for their helpful remarks and suggestions, which allowed us to significantly improve the original presentation. \vspace*{-0.2in}

\section*{Funding and Conflicts of Interests} 

Research of Pham Duy Khanh is funded by Ho Chi Minh City University of Education Foundation for Science and Technology under grant number CS.2025.19.04TD. Research of Boris S. Mordukhovich was partly supported by the US National Science Foundation under grant DMS-2204519 and by the Australian Research Council under Discovery Project DP250101112. Research of Vo Thanh Phat was partly supported by the Early Career Scholars Program 2026, University of North Dakota under project 43700-2375-UND0031286. The authors declare that the presented results are new, and there is no any conflict of interest.
\end{document}